\documentclass[11pt]{amsart}

\usepackage{xspace}
\usepackage{graphicx}
\usepackage{longtable}
\usepackage{mathtools}
\usepackage{latexsym}
\usepackage{mathrsfs}
\usepackage{stmaryrd}
\usepackage{fullpage}
\usepackage{amsfonts,amsmath,amscd,amssymb}
\usepackage{hyperref}
\input{xy}
\xyoption{all}
\xyoption{curve}

\hypersetup{
	colorlinks=true,
	linkcolor=blue,
	filecolor=magenta,      
	urlcolor=cyan,
}

\newcommand{\bD}{{\mathbb D}}
\newcommand{\bG}{{\mathbb G}}
\newcommand{\bK}{{\mathbb K}}
\newcommand{\bF}{{\mathbb F}}
\newcommand{\bZ}{{\mathbb Z}}

\newcommand{\fg}{{\mathfrak g}}
\newcommand{\fl}{{\mathfrak l}}
\newcommand{\fh}{{\mathfrak h}}
\newcommand{\fm}{{\mathfrak m}}
\newcommand{\fn}{{\mathfrak n}}
\newcommand{\fo}{{\mathfrak o}}
\newcommand{\fp}{{\mathfrak p}}
\newcommand{\fs}{{\mathfrak s}}
\newcommand{\fu}{{\mathfrak u}}

\newcommand{\sC}{{\mathscr C}}

\newcommand{\sF}{{\mathscr F}}
\newcommand{\sG}{{\mathscr G}}

\newcommand{\sM}{{\mathscr M}}

\newcommand{\sX}{{\mathscr X}}
\newcommand{\fb}{{\mathfrak b}}

\newcommand{\Aut}{{{\mbox{\rm Aut}}}}

\newcommand{\cen}{{{\mbox{\rm cen}}}} 
\newcommand{\Ext}{{{\mbox{\rm Ext}}}}
\newcommand{\Lie}{{{\mbox{\rm Lie}}}}

\newcommand{\Hom}{{{\mbox{\rm Hom}}}}

\newcommand{\Id}{{{\mbox{\rm Id}}}}

\newcommand{\Rad}{{{\mbox{\rm Rad}}}}

\newcommand{\End}{{{\mbox{\rm End}}}}

\newcommand{\Mod}{{{\mbox{\rm Mod}}}}

\newcommand{\twopartdef}[4]{\left\{
	\begin{array}{ll}
		#1 & \mbox{if } #2 \\
		#3 & \mbox{if } #4
	\end{array}
	\right.}

\newtheorem{theorem}{Theorem}[section]
\newtheorem{prop}[theorem]{Proposition}
\newtheorem{lemma}[theorem]{Lemma}
\newtheorem*{dfn}{Definition}
\newtheorem{cor}[theorem]{Corollary}

\newtheorem{rmk}{Remark}

\newcommand{\arxiv}[1]{{\tt arXiv:#1}}

\begin{document}
\title{On graded representations of modular Lie algebras over commutative algebras}
\author{Matthew Westaway}
\email{M.P.Westaway@bham.ac.uk}
\address{School of Mathematics, University of Birmingham, Birmingham, B15 2TT, UK}
\date{\today}
  
\begin{abstract}
	We develop the theory of a category $\sC_A$ which is a generalisation to non-restricted $\fg$-modules of a category famously studied by Andersen, Jantzen and Soergel for restricted $\fg$-modules, where $\fg$ is the Lie algebra of a reductive group $G$ over an algebraically closed field $\bK$ of characteristic $p>0$. Its objects are certain graded bimodules. On the left, they are graded modules over an algebra $U_\chi$ associated to $\fg$ and to $\chi\in\fg^{*}$ in standard Levi form. On the right, they are modules over a commutative Noetherian $S(\fh)$-algebra $A$, where $\fh$ is the Lie algebra of a maximal torus of $G$. We develop here certain important modules $Z_{A,\chi}(\lambda)$, $Q_{A,\chi}^I(\lambda)$ and $Q_{A,\chi}(\lambda)$ in $\sC_A$ which generalise familiar objects when $A=\bK$, and we prove some key structural results regarding them.
\end{abstract}
\maketitle
\tableofcontents

\section{Introduction}

Suppose $G$ is a connected reductive algebraic group over an algebraically closed field $\bK$ of characteristic $p>0$, and let $\fg$ be its Lie algebra. Each simple module over $\fg$ can be assigned a $p$-character, a linear map $\chi:\fg\to\bK$ controlling how a certain central subalgebra of the universal enveloping algebra $U(\fg)$ acts on the module. We define the {\bf reduced enveloping algebra} as $U_\chi(\fg)=U(\fg)/\langle x^p-x^{[p]}-\chi(x)^p\,\vert\,x\in \fg\rangle$, where $x\mapsto x^{[p]}$ is the $p$-th power map on $\fg$ which equips it with a restricted Lie algebra structure. Every simple $\fg$-module is a simple $U_\chi(\fg)$-module for some $\chi\in\fg^{*}$, and thus understanding the representation theory of reduced enveloping algebras is critical to understand the representation theory of modular Lie algebras.

For suitable $\lambda\in \fh^{*}$, where $\fh$ is the Lie algebra of a maximal torus $T$ of $G$, we may define (so long as we make certain reasonable assumptions on $\chi$) a $U_\chi(\fg)$-module $Z_\chi(\lambda)$, which we call a {\bf baby Verma module}. The key quality possessed by baby Verma modules is that every irreducible $U_\chi(\fg)$-module is a quotient of some $Z_\chi(\lambda)$. In general, a baby Verma module may have more than one irreducible quotient, and an irreducible module may be a quotient of more than one baby Verma module. However, these complications can be removed in certain cases. If we assume $\chi$ is in so-called {\bf standard Levi form}, each $Z_\chi(\lambda)$ has a unique irreducible quotient. We also find in this case that if an irreducible module is a quotient of two baby Verma modules then these baby Verma modules are isomorphic.

There is a bijection between the set of $\chi\in\fg^{*}$ in standard Levi form and the set of subsets $I$ of simple roots of $\fg$. If $\chi\in\fg^{*}$ corresponds to $I$, we may equip $U_\chi(\fg)$ with an $X/\bZ I$-grading, where $X\coloneqq X(T)$ is the character group of $T$, and we may consider $X/\bZ I$-graded $U_\chi(\fg)$-modules. For technical reasons it is convenient not to work with the whole category of such modules, but instead a certain subcategory $\sC$, copies of which the wider category is a direct sum. This category has been extensively studied in \cite{Jan4} (see also \cite{Jan2,Jan}).

The most well-known example of $\chi\in\fg^{*}$ in standard Levi form is $\chi=0$. In this case, $I$ is empty and the category $\sC$ becomes precisely the category of $G_1T$-modules ($G_1$ being the first Frobenius kernel of $G$). Andersen, Jantzen and Soergel, in their seminal work \cite{AJS}, study when $\chi=0$ a category they call $\sC_A$ (they also study the analogous category for quantum groups). Here, $A$ is a commutative Noetherian $S(\fh)$-algebra and the objects of $\sC_A$ are $U_0\otimes A$-modules with an $X$-grading satisfying some conditions, where $U_0$ is a certain quotient of $U(\fg)$ lying between $U(\fg)$ and $U_0(\fg)$. When $A=\bK$, made into an $S(\fh)$ algebra via the counit, one recovers precisely Jantzen's category $\sC$ for $\chi=0$.

The motivation for Andersen, Jantzen and Soergel's study of this category was their interest in proving Lusztig's conjecture \cite{L}, which they were able to do for $p\gg 0$. Lusztig's conjecture, the reader will recall, gives a prediction for the characters of certain simple $G$-modules in terms of combinatorial properties of the affine Weyl group; one can also understand it as predicting composition multiplicities of simple $G$-modules in Weyl modules. It is a positive characteristic analogue of the famous Kazhdan-Lusztig conjecture \cite{KL} (independently proved by Beilinson-Bernstein \cite{BB} and by Brylinski-Kashiwara \cite{BK}), a fundamental result in the characteristic zero theory.

In the last few years, there have been several major developments in our understanding of Lusztig's conjecture. Williamson \cite{Wi} has showed that Lusztig's conjecture cannot hold under the expected assumptions; however Riche and Williamson showed in \cite{RW1} that it would be possible to prove a modification of Lusztig's conjecture (using so-called $p$-Kazhdan-Lusztig polynomials) if one could find a nice action of the Hecke category on a certain category of $G$-modules. Recent preprints by Bezrukavnikov and Riche \cite{BR} and by Ciappara \cite{Ci} are able to find such an action.

Work of Abe \cite{Abe1,Abe2} also investigates actions of the Hecke category, and is able to define such an action on the category of $X$-graded $U_0(\fg)$-modules (i.e. $G_1T$-modules). As observed above, the category of $G_1T$-modules is a special case of the category $\sC$ studied by Jantzen, and so a natural question is whether one can construct a similar categorical action on the category $\sC$ for any $\chi$ in standard Levi form. Lusztig conjectured in \cite{L1} a variation of his conjecture for such $\chi$; finding such a categorical action could plausibly enable us to prove a variation of this conjecture, just as Riche-Williamson \cite{RW1} were able to do for the classical version.

Abe's construction \cite{Abe1} of the Hecke action requires an in-depth understanding of a number of a categories and functors introduced by Andersen, Jantzen and Soergel in \cite{AJS}. To generalise this to work for non-zero $\chi$ in standard Levi form, one must first generalise the Anderson-Jantzen-Soergel categories and functors, particularly the category $\sC_A$ discussed above, to work for such $\chi$. Jantzen discusses this situation briefly in \cite{Jan4}, but doesn't give an full treatment. We do so here. In subsequent work, we hope to be able to use these categories and functors to generalise Abe's description of the category of $G_1T$-modules, and thus find an appropriate categorial action.

In this paper, we study a category which we, similar to \cite{AJS}, denote by $\sC_A$ (since we generally fix $\chi$ in standard Levi form throughout this paper, we do not incorporate $\chi$ into the notation). Its objects are certain $U_\chi\otimes A$-modules which have an $X/\bZ I$-grading and which satisfy several other conditions. The algebra $U_\chi$ here is defined similarly to the algebra $U_0$, and indeed the notations are compatible.

This category turns out to have a number of nice properties, many of which are similar to those explored for $\chi=0$ in \cite{AJS} and for $A=\bK$ in \cite{Jan4} (under assumptions on $G$ and $p$ similar to those required in the latter). There are a number of induction functors which interact nicely with the category. This allows us, for example, to define baby Verma modules $Z_{A,\chi}(\lambda)$, for $\lambda\in X$, in $\sC_A$. Note that in graded settings we use the character group $X$ instead of $\fh^{*}$ to define baby Verma modules. When $A=F$ is a field, these baby Verma modules have unique irreducible quotients, denoted $L_{F,\chi}(\lambda)$ (see Proposition~\ref{UniqIrredQuot}) and, under certain conditions, we may characterise when two baby Verma modules are isomorphic (see Corollary~\ref{ZIsomFull}).

The key distinction from the $\chi=0$ case is that there are a few other categories of interest at play here. We write $\fg_I$ for the Levi subalgebra of $\fg$ obtained from the set of simple roots $I$ corresponding to $\chi$. One may then look at the category $\sC_A^I$ which is obtained from $\fg_I$ in the same way that $\sC_A$ is obtained from $\fg$. Since $\chi\vert_{\fg_I}$ is regular nilpotent (i.e. $I$ contains all the simple roots for this Lie algebra) the category $\sC_A^I$ has a lot of nice structure which doesn't usually exist. We may then define induction functors which get us from $\sC_A^I$ to $\sC_A$.

The fundamental benefit of this is that it allows us, subject to some reasonable conditions on $A$, to define objects $Q_{A,\chi}^I(\lambda)$, $\lambda\in X$, obtained by inducing to $\sC_A$ certain projective objects $Q_{A,I,\chi}(\lambda)\in\sC_A^I$ (which will be projective covers if $A$ is a field). These objects are central to understanding $\sC_A$ - indeed their versions over $\bK$ play a critical role in understanding $\sC$ in \cite{Jan4} and \cite{Jan}.

One important application of the $Q_{A,\chi}^I(\lambda)$ is that they allow us to prove the existence, when $A$ is a local ring with residue field $F$, of projective objects $Q_{A,\chi}(\lambda)$, $\lambda\in X$, such that $Q_{A,\chi}(\lambda)\otimes_{A} F$ equals the projective cover of $L_{F,\chi}(\lambda)$ in $\sC_F$. This is the content of Theorems~\ref{ProjCovQ} and \ref{ProjCovQ2}, which generalise Proposition 4.18 and Theorem 4.19 in \cite{AJS}.

We begin the paper by defining the category $\sC_A$ that we work with. We define the algebras necessary to construct the category, the category itself, and certain variations of the category over subalgebras in Section~\ref{Sec3}. In Section~\ref{Sec4}, we see how the categories are interrelated using induction functors, and define baby Verma modules in this context. We also see the extra categories and induction functors we get when $\chi$ is non-zero. We survey some miscellaneous results about the category $\sC_A$ and its objects in Section~\ref{Sec5}, which will be useful for what follows. In Section~\ref{Sec6} we discuss some properties of baby Verma modules in this context. Finally, Sections~\ref{Sec7} and \ref{Sec8} contain the most substantive results in the paper. We explore in detail the projective objects in the category $\sC_A$, working with regular nilpotent $\chi$ in Section~\ref{Sec7} and more general $\chi$ in Section~\ref{Sec8}. In Section~\ref{Sec8}, we also introduce the objects $Q_{A,\chi}^I(\lambda)$ and we conclude by showing, when $A$ is local with residue field $F$, the connection between projective objects in the categories $\sC_A$ and $\sC_F$.

The author was supported during this research by EPSRC grant EP/R018952/1. He would like to thank Simon Goodwin for many useful discussions about this subject and for his opinions on an earlier version of this paper, as well as the referee for their comments. The author has no competing interests to declare.

\section{Notation}\label{Sec2}

Let $G$ be a connected reductive algebraic group over an algebraically closed field $\bK$ of characteristic $p>0$. We fix a maximal torus $T\leq G$ of rank $d$, and a Borel subgroup $B$ containing $T$. We denote by $X$ the character group $X(T)=\Hom(T,\bG_m)$, where $\bG_m$ is the multiplicative group of the ground field $\bK$. For the Lie algebras, we write $\fg=\Lie(G)$, $\fh=\Lie(T)$ and $\fb=\Lie(B)$. We then write $R\subset X$ for the set of roots of $\fg$ (i.e. the weights of $\fg$ under the adjoint action of $T$), and we set $R^{+}$ to be the set of positive roots and $\Pi$ to be the set of simple roots corresponding to our choice of $B$. We write $\Pi=\{\alpha_1,\ldots,\alpha_n\}$ and $R^{+}=\{\beta_1,\ldots,\beta_r\}$ with $\beta_i=\alpha_i$ for $1\leq i\leq n$. We also write $Y(T)=\Hom(\bG_m,T)$ for the cocharacters of $T$, and we write $\langle\cdot,\cdot\rangle:X(T)\times Y(T)\to \bZ$ for the natural pairing. Given $\alpha\in R$, we denote by $\alpha^\vee\in Y(T)$ the coroot associated to $\alpha$.

We associate to $\fg$ a basis $\{e_\alpha,\,h_i\,\vert\,\alpha\in R,\,1\leq i\leq d\}$ of $\fg$, where $e_\alpha\in \fg_{\alpha}$, the root space of $\alpha\in R$, and $h_1,\ldots,h_d$ is a basis of $\fh$ with the property that $h_i^{[p]}=h_i$ for each $1\leq i\leq d$. Here, $x\mapsto x^{[p]}$ is the $p$-th power map on $\fg$. We write $h_\alpha\coloneqq[e_\alpha,e_{-\alpha}]$ for each $\alpha\in R$, and we set $\fn^{+}=\bigoplus_{\alpha\in R^{+}}\fg_\alpha$ and $\fn^{-}=\bigoplus_{\alpha\in R^{+}}\fg_{-\alpha}$. Given $\chi\in \fg^{*}$, we write $U_\chi(\fg)$ for the {\bf reduced enveloping algebra} of $\fg$, whose definition we shall explain in more detail in the body of the article. Note that $U_\chi(\fg)\cong U_{g\cdot\chi}(\fg)$ for each $g\in G$ (where $g\cdot\chi$ denotes the image of $\chi$ under the coadjoint action of $G$).

We shall make Jantzen's standard assumptions as in Section 6.3 of \cite{Jan}. In other words, we assume (1) that $G$ has simply-connected derived subgroup, (2) that $p$ is good for $G$, and (3) that there exists a non-degenerate $G$-invariant bilinear form on $\fg$. Recall here that a prime $p$ being {\bf good} for $G$ means that it is not {\bf bad} for any irreducible component of $R$,  i.e. $p>2$ if $R$ has a component of type $B_n$, $C_n$, or $D_n$, $p>3$ if $R$ has a component of type $E_6$, $E_7$, $F_4$, or $G_2$, and $p>5$ if $R$ has a component of type $E_8$. Under Jantzen's assumptions, we may assume in studying the representation theory of $U_\chi(\fg)$ that $\chi$ is nilpotent; in fact, that $\chi(\fb)=0$. However, we will go further and assume throughout this paper that $\chi$ is in {\bf standard Levi form}, i.e., that $\chi(\fb)=0$ and that there exists a subset $I\subseteq\Pi$ such that for $\alpha\in R^{+}$ the map $\chi$ is defined by $$\chi(e_{-\alpha})=\twopartdef{1}{\alpha\in I,}{0}{\alpha\notin I.}$$ In this case, $\bZ I$ is a subgroup of the character group $X$, and the quotient group $X/\bZ I$ can be equipped with a partial ordering such that:
$$\lambda+\bZ I\geq \mu+\bZ I\quad\mbox{if and only if}\quad \lambda-\mu+\bZ I=\sum_{i=1}^n m_i \alpha_i +\bZ I\,\,\mbox{for some}\,\, m_1,\ldots, m_n\geq 0.$$

\section{The category $\sC_A$}\label{Sec3}

\subsection{Definition of algebras}\label{Sec3.1}

Since $\fg$ is a Lie algebra over $\bK$, we may of course define the {\bf universal enveloping algebra} $U(\fg)$ as $$U(\fg)=\frac{T(\fg)}{\langle x\otimes y - y\otimes x-[x,y]\,\vert\,x,y\in\fg\rangle},$$ where $T(\fg)$ is the tensor algebra of $\fg$. Given $\chi\in\fg^{*}$, the {\bf reduced enveloping algebra} of $\fg$ is defined as $$U_\chi(\fg)=\frac{U(\fg)}{\langle x^{p}-x^{[p]}-\chi(x)^p\,\vert\, x\in\fg\rangle}.$$ When working over the field $\bK$, the reduced enveloping algebras are key objects of study in the representation theory of $\fg$. When working instead with a general commutative algebra $A$, however, we will need to also consider an algebra lying between $U(\fg)$ and $U_\chi(\fg)$, which we will call $U_\chi$. This is defined as $$U_\chi\coloneqq U(\fg)/\langle e_{\alpha}^p-\chi(e_{\alpha})^p\,\vert\,\alpha\in R\rangle.$$

In particular, $U_\chi$ has a $\bK$-basis consisting of elements of the form $$e_{-\beta_r}^{a_r}\cdots e_{-\beta_1}^{a_1} h_1^{b_1}\cdots h_d^{b_d} e_{\beta_1}^{c_1}\cdots e_{\beta_r}^{c_r}$$ with $0\leq a_i,c_i<p$ and $b_i\geq 0$. We also need notation for certain subalgebras of $U_\chi$. Specifically, we set $$U^{+}\coloneqq U_\chi(\fn^{+})=U_0(\fn^{+})\subseteq U_\chi,$$
$$U^{-}\coloneqq U_{\chi}(\fn^{-})\subseteq U_\chi,$$
and
$$U^{0}\coloneqq U(\fh)\subseteq U_\chi.$$
Recalling that $\chi$ is in {\bf standard Levi form} with associated subset $I$ of simple roots, we furthermore define $U^I$ to be the subalgebra of $U_\chi$ generated by $\fh$ and the root vectors $e_\alpha$ for $\alpha\in R\cap \bZ I$. Going forward, we will sometimes write $R_I\coloneqq R\cap\bZ I$ and $R_I^{+}= R^{+}\cap \bZ I$. If we write $\fg_I$ for the Lie subalgebra of $\fg$ generated by these elements, then we may also describe $U^I$ as 
$$U^I=\frac{U(\fg_I)}{\langle e_\alpha^p-\chi(e_\alpha)^p\,\vert\,\alpha\in R_I\rangle}.$$
If we now write $\fu^{+}$ for the Lie subalgebra of $\fg$ generated by the root vectors $e_\alpha$ for $\alpha\in R^{+}\setminus \bZ I$, and $\fu^{-}$ for the analogous Lie subalgebra for the negative roots, we get that $\fg=\fu^{-}\oplus\fg_I\oplus\fu^{+}$. We may then additionally define the subalgebras $$U_I^{+}\coloneqq U_\chi(\fu^{+})=U_0(\fu^{+})\subseteq U_\chi$$ and $$U_I^{-}\coloneqq U_\chi(\fu^{-})=U_0(\fu^{-})\subseteq U_\chi.$$

Observe that, as $\bK$-vector spaces, $$U_\chi=U^{-}\otimes U^0\otimes U^{+}=U_I^{-}\otimes U^I\otimes U_I^{+}.$$

Now, it is straightforward to see that $U_\chi$ may be equipped with an $X/\bZ I$-grading in the following way, where we write $(U_\chi)_{\lambda+\bZ I}$ for the $\lambda+\bZ I$-graded part of $U_\chi$:
$$e_\alpha\in (U_\chi)_{\alpha+\bZ I}\quad\mbox{for all }\,\alpha\in R;\quad \mbox{and}\quad \fh\subseteq (U_\chi)_{0+\bZ I}.$$  The commutative Lie algebra $\fh$ then acts on each $(U_\chi)_{\lambda+\bZ I}$, which means that each $(U_\chi)_{\lambda+\bZ I}$ decomposes into weight spaces for this action. In particular, we have a decomposition into $\bK$-subspaces $$(U_\chi)_{\lambda+\bZ I}=\bigoplus_{\substack{d\mu\in \fh^{*} \\ \mu\in \lambda+\bZ I +pX }}(U_\chi)_{\lambda+\bZ I}^{d\mu},$$ where $h\in\fh$ acts on $(U_\chi)_{\lambda+\bZ I}^{d\mu}$ as multiplication by $d\mu(h)$. We then have the following lemma (c.f. Lemma 1.4 in \cite{AJS}).

\begin{lemma}\label{Aut}
	There exists a group homomorphism $X\to \Aut_{\bK-alg}(U^0)$, $\mu\mapsto \widetilde{\mu}$, with the property that $$su=u\widetilde{\mu}(s)$$ for all $s\in U^{0}$, $\lambda\in X$, $\mu\in \lambda+\bZ I+pX$, and $u\in (U_\chi)_{\lambda+\bZ I}^{d\mu}$.
\end{lemma}

\begin{proof}
	This follows as in Lemma 1.4 in \cite{AJS}. Namely, for $\mu\in X$ we define $\widetilde{\mu}:U^0\to U^0$ by setting $\widetilde{\mu}(h)=h+d\mu(h)$ for each $h\in\fh$, and extending in the natural way. 
\end{proof}

\subsection{Definition of the category $\sC_A$}\label{Sec3.2}

Let $A$ be a commutative Noetherian algebra over $U^0$ with structure map $\pi:U^0\to A$. For example, we could take $A=\bK$ with $\pi(h)=0$ for all $h\in\fh$. 

We define the category $\sC_A$ in the following way. The objects of $\sC_A$ are $U_\chi\otimes A$-modules $M$ which, as $\bK$-vector spaces, decompose as $$M=\bigoplus_{\lambda+\bZ I\in X/\bZ I} M_{\lambda+\bZ I}$$ for some subspaces $M_{\lambda+\bZ I}$, subject to some additional properties. In general, we treat a $U_\chi\otimes A$-module as being a left $U_\chi$-module and a right $A$-module, so that $(u\otimes a)m=uma$ for $u\in U_\chi$, $m\in M$ and $a\in A$. The defining conditions which objects of $\sC_A$ must satisfy are the following:
\begin{enumerate}
	\item[(A)] The $A$-action preserves the $X/\bZ I$-grading, i.e. $M_{\lambda+\bZ I}A\subseteq M_{\lambda+\bZ I}$ for all $\lambda+\bZ I\in X/\bZ I$.
	\item[(B)] There are only finitely many $\lambda+\bZ I\in X/\bZ I$ with $M_{\lambda+\bZ I}\neq 0$, and each $M_{\lambda+\bZ I}$ is finitely-generated as an $A$-module.
	\item[(C)] For any $\sigma+\bZ I, \lambda +\bZ I\in X/\bZ I$, we have $(U_\chi)_{\sigma+\bZ I} M_{\lambda+\bZ I}\subseteq M_{\sigma + \lambda+\bZ I}$.
	\item[(D)] For each $\lambda+\bZ I\in X/\bZ I$ there is a decomposition $$M_{\lambda+\bZ I}=\bigoplus_{\substack{d\mu\in \fh^{*} \\ \mu\in \lambda+\bZ I +pX}} M_{\lambda+\bZ I}^{d\mu}$$ with the property that $$sm=m\pi(\widetilde{\mu}(s))$$ for each $s\in U^0$, $\mu\in \lambda+\bZ I+pX$, and $m\in M_{\lambda+\bZ I}^{d\mu}$. Furthermore, if $\alpha\in R$ and $m\in M_{\lambda+\bZ I}^{d\mu}$ then $e_{\alpha}m\in M_{\lambda+\alpha+\bZ I}^{d(\mu+\alpha)}$, and $ma\in M_{\lambda+\bZ I}^{d\mu}$ for all $a\in A$.
\end{enumerate}
A morphism $M\to N$ in $\sC_A$ is then a homomorphism of $U_\chi\otimes A$-modules which sends $M_{\lambda+\bZ I}^{d\mu}$ to $N_{\lambda+\bZ I}^{d\mu}$ for each $\lambda+\bZ I\in X/\bZ I$ and $\mu\in\lambda+\bZ I+pX$.

Let us make a couple of observations about condition (D). First, we note that, by construction, $\widetilde{\mu}(s)=\widetilde{\mu+p\tau}(s)$ for any $\mu,\tau\in X$ and $s\in U^0$, and so we indeed only care about $d\mu$ rather than $\mu$ itself in the direct sum. Second, the assumption that $e_\alpha M_{\lambda+\bZ I}^{d\mu}\subseteq M_{\lambda+\alpha+\bZ I}^{d(\mu+\alpha)}$, for each $\alpha\in R$, $\lambda+\bZ I\in X/\bZ I$ and $\mu\in\lambda+\bZ I+pX$, is not strictly necessary, as it will follow from the other conditions once we observe that $$M_{\lambda+\bZ I}^{d\mu}=\{m\in M_{\lambda+\bZ I}\,\vert\, hm=m\pi(\widetilde{\mu}(h))\,\mbox{ for all }\, h\in\fh\}.$$ For this latter equality to hold, it is key to note that, for $\mu,\sigma\in\lambda+\bZ I+pX$ and $h\in \fh$, we have $\widetilde{\mu}(h)-\widetilde{\sigma}(h)=h+d\mu(h)-h-d\sigma(h)=d(\mu-\sigma)(h)\in\bK$. Finally, as a notational matter, we shall call the decomposition in condition (D) a {\bf (D)-decomposition} when it helps make things clearer.

At times, it will be useful to think about $\sC_A$ in a different way. We observe that $U_\chi$ is $X/p\bZ I$-graded, with $(U_\chi)_{\lambda+p\bZ I}=(U_\chi)_{\lambda+\bZ I}^{d\lambda}$ for each $\lambda+p\bZ I\in X/p\bZ I$. That this is indeed a grading can be checked directly, but also may be seen using a similar argument to that used in the proof of Lemma~\ref{XZIvXpZI}. We may then define $\widetilde{\sC_A}$ to be the category whose objects are $U_\chi\otimes A$-modules $M$ with $\bK$-decompositions $$M=\bigoplus_{\lambda+p\bZ I\in X/p\bZ I} M_{\lambda+p\bZ I}$$ for some $\bK$-subspaces $M_{\lambda+p\bZ I}$, subject to the following properties (where, as above, we treat a $U_\chi\otimes A$-action as both a left $U_\chi$-action and a right $A$-action):
\begin{enumerate}
	\item[(A$'$)] The $A$-action preserves the $X/p\bZ I$-grading, i.e. $M_{\lambda+p\bZ I}A\subseteq M_{\lambda+p\bZ I}$ for each $\lambda+p\bZ I\in X/p\bZ I$.
	\item[(B$'$)] There are only finitely many $\lambda+p\bZ I\in X/p\bZ I$ with $M_{\lambda+p\bZ I}\neq 0$, and each $M_{\lambda+p\bZ I}$ is finitely-generated as an $A$-module.
	\item[(C$'$)] For any $\sigma+p\bZ I,\lambda +p\bZ I\in X/p\bZ I$, we have $(U_\chi)_{\sigma+p\bZ I} M_{\lambda+p\bZ I}\subseteq M_{\sigma + \lambda+p\bZ I}$.
	\item[(D$'$)] We have that $$sm=m\pi(\widetilde{\mu}(s))$$ for each $s\in U^0$, $\mu\in \lambda+p\bZ I$, and $m\in M_{\mu+p\bZ I}$.
\end{enumerate}

In this category, morphisms are $U_\chi\otimes A$-morphisms which preserve the grading.

\begin{lemma}\label{XZIvXpZI}
	There is an equivalence of categories between $\sC_A$ and $\widetilde{\sC_A}$.
\end{lemma}

\begin{proof}
	Let $\Phi:\sC_A\to \widetilde{\sC_A}$ be the identity map on objects (as $U_\chi\otimes A$-modules), and let $M\in\sC_A$. We want to equip $\Phi(M)$ with an $X/p\bZ I$-grading.
	
	Let $\lambda+\bZ I\in X/\bZ I$ and $\mu\in \lambda+\bZ I +pX$. Then there exists $\tau\in X$ and $\sigma\in\bZ I$ such that $\mu=\lambda+\sigma+p\tau$. Therefore $\lambda+\bZ I=\mu-p\tau +\bZ I$ and $d(\mu-p\tau)=d\mu$, and thus $M_{\lambda+\bZ I}^{d\mu}=M_{\mu-p\tau+\bZ I}^{d(\mu-p\tau)}$. This means that each $M_{\lambda+\bZ I}^{d\mu}$ can be written as $M_{\gamma+\bZ I}^{d\gamma}$ for some $\gamma \in X$. Furthermore, if $\gamma+\bZ I=\gamma'+\bZ I$ and $d\gamma=d(\gamma')$, then $\gamma-\gamma'\in \bZ I\cap pX$. Since we make Jantzen's assumptions on $\fg$ and $p$, we have $\bZ I\cap pX=p\bZ I$ (see Section 11.2 in \cite{Jan}).
	
	We may hence equip $\Phi(M)$ with an $X/p\bZ I$-grading by setting $\Phi(M)_{\lambda+p\bZ I}=M_{\lambda+\bZ I}^{d\lambda}$. We have $\Phi(M)=\bigoplus_{\lambda+p\bZ I}\Phi(M)_{\lambda+p\bZ I}$ since $$M=\bigoplus_{\lambda+\bZ I\in X/\bZ I} \bigoplus_{\substack{d\mu\in \fh^{*}\\ \mu\in \lambda+\bZ I +pX}} M_{\lambda+\bZ I}^{d\mu}.$$ This makes $\Phi(M)$ into an $X/p\bZ I$ graded $U_\chi$-module using conditions (C) and (D). The object $\Phi(M)$ satisfies condition (A$'$) because of the last clause of condition (D), satisfies condition (B$'$) because $p\bZ I$ has finite index in $\bZ I$ and because $A$ is Noetherian (so $A$-submodules of $M_{\lambda+\bZ I}$ are finitely generated), and satisfies condition (D$'$) by construction. So $\Phi(M)\in\widetilde{\sC_A}$. It is then easy to see that $\Phi$ sends morphisms to morphisms, and so defines a functor $\sC_A\to\widetilde{\sC_A}$.
	
	On the other hand, let $\Psi:\widetilde{\sC_A}\to\sC_A$ be the identity map on objects (as $U_\chi\otimes A$-modules), and let $N\in\widetilde{\sC_A}$. We wish to define the $A$-submodules $\Psi(N)_{\lambda+\bZ I}$ and $\Psi(N)_{\lambda+\bZ I}^{d\mu}$.
	
	Set $\Psi(N)_{\lambda+\bZ I}=\bigoplus_{\tau+p\bZ I\subset \lambda+\bZ I}N_{\tau+p\bZ I}$. This clearly gives $\Psi(N)$ an $X/\bZ I$-grading as a $U_\chi$-module, so $\Psi(N)$ satisfies condition (C). It clearly also satisfies condition (A) by condition (A$'$) and condition (B) by condition (B$'$). For condition (D$'$), let $\lambda+\bZ I\in X/\bZ I$ and $\mu\in \lambda+\bZ I +pX$, and set $\Psi(N)_{\lambda+\bZ I}^{d\mu}\coloneqq N_{\gamma+p\bZ I}$, where $\gamma+p\bZ I$ is the unique element of $X/p\bZ I$ with $\gamma+\bZ I=\lambda+\bZ I$ and $\gamma+pX=\mu+pX$. This exists by above. It is clear that $$\Psi(N)_{\lambda+\bZ I}= \bigoplus_{\substack{d\mu\in \fh^{*} \\ \mu\in \lambda+\bZ I +pX}} \Psi(N)_{\lambda+\bZ I}^{d\mu}$$ since, if $\mu+p\bZ I\subset \lambda+\bZ I$, we have $N_{\mu+p\bZ I}=\Psi(N)_{\mu+\bZ I}^{d\mu}$ and if $\mu\in \lambda+\bZ I +pX$ we have $\Psi(N)_{\lambda+\bZ I}^{d\mu}=M_{\gamma+p\bZ I}$ for some $\gamma+p\bZ I\in X/p\bZ I$ with $\gamma+p\bZ I\subseteq \lambda+\bZ I$. The remaining parts of condition (D) follow easily from conditions (D$'$) and (A$'$). Clearly $\Psi$ sends morphisms to morphisms, and thus defines a functor $\widetilde{\sC_A}\to\sC_A$.
	
	All that remains is to see that $\Psi$ and $\Phi$ are inverse equivalences of categories. This follows since $\Phi(\Psi(N))_{\lambda+p\bZ I}=\Psi(N)_{\lambda+\bZ I}^{d\lambda}=N_{\lambda+p\bZ I}$ and $\Psi(\Phi(M))_{\lambda+\bZ I}^{d\mu}=\Phi(M)_{\gamma+p\bZ I}=M_{\gamma+\bZ I}^{d\gamma}$ for $\gamma+p\bZ I\in X/p\bZ I$ such that $\gamma+\bZ I=\lambda+\bZ I$ and $d\gamma=d\mu$.
\end{proof}

\begin{lemma}
	The category $\sC_A$ has kernels, cokernels and images.
\end{lemma}
\begin{proof}
	This is easy to see in $\widetilde{\sC_A}$, using the fact that $A$ is Noetherian.
\end{proof}

It is clear from the definition of $\widetilde{\sC_A}$ that if $I=\emptyset$ (i.e. if $\chi=0$) then $\widetilde{\sC_A}$ is precisely the category called $\sC_A$ in \cite{AJS}. The same observation can be made with a tiny bit more work for the category we call $\sC_A$.

From now on, while we principally work in the category $\sC_A$, we occasionally shift to the category $\widetilde{\sC_A}$ without comment. In particular, the equivalence of these categories means that instead of defining, say, $M_{\lambda+\bZ I}^{d\mu}$ for each $\lambda+\bZ I\in X/\bZ I$ and $\mu\in\lambda+\bZ I+pX$, it suffices to define $M_{\lambda+p\bZ I}$ for each $\lambda+p\bZ I\in X/p\bZ I$.

\subsection{Categories defined over subalgebras}\label{Sec3.3}
%
%
%

Let us write $\sC_A'$ for the category with the same definition as $\sC_A$ but where the objects are $U^0U^{+}\otimes A$-modules rather that $U_\chi\otimes A$-modules. Objects in this category are $X/\bZ I$-graded by definition. However, since $\chi(\fb)=0$, the algebra $U^0U^{+}$ is also $X$-graded. These gradings are related by $$(U^0U^{+})_{\lambda+\bZ I}=\bigoplus_{\mu\in\lambda+\bZ I} (U^0U^{+})_{\mu}.$$ We may hence define by $\widehat{\sC_A'}$ the category of $X$-graded $U^0U^{+}\otimes A$-modules $M$ which satisfy the following:
\begin{enumerate}
	\item[(\^{A})] The $A$-action preserves the $X$-grading, i.e. $M_{\lambda}A\subseteq M_{\lambda}$ for each $\lambda\in X$.
	\item[(\^{B})] There are only finitely many $\lambda\in X$ with $M_{\lambda}\neq 0$, and each $M_{\lambda}$ is finitely-generated as an $A$-module.
	\item[(\^{C})] For any $\sigma,\lambda\in X$, we have $(U^0U^{+})_{\sigma} M_{\lambda}\subseteq M_{\sigma+\lambda}$.
	\item[(\^{D})] For each $\mu\in X$, $m\in M_\mu$, and $s\in U^0$, we have $sm=m\pi(\widetilde{\mu}(s))$. Here, $\widetilde{\mu}$ is as in Lemma~\ref{Aut}.
\end{enumerate}
Morphisms are homomorphisms of $U^0U^{+}\otimes A$-modules which preserve the grading. In particular, the category $\widehat{\sC_A'}$ is precisely the category which is called $\sC_A'$ in \cite{AJS}.

\begin{prop}\label{UpsilonDef}
	There exists a functor $\Upsilon_A:\widehat{\sC_A'}\to \sC_A'$.
\end{prop}

\begin{proof}
	
	We define $\Upsilon_A:\widehat{\sC_A'}\to \sC_A'$ as follows. Given $M\in\widehat{\sC_A'}$, we set $\Upsilon_A(M)=M$ as a $U^0U^{+}\otimes A$-module. We need to equip $\Upsilon_A(M)$ with an $X/\bZ I$-grading, and define the subspaces $\Upsilon_A(M)_{\lambda+\bZ I}^{d\mu}$ for $\lambda+\bZ I\in X/\bZ I$ and $\mu\in\lambda+\bZ I+pX$. We do that as follows:
	$$\Upsilon_A(M)_{\lambda+\bZ I}\coloneqq\bigoplus_{\sigma\in\lambda+\bZ I}M_\sigma.$$
	$$\Upsilon_A(M)_{\lambda+\bZ I}^{d\mu}\coloneqq\bigoplus_{\substack{\sigma\in \lambda+\bZ I \\ d\sigma=d\mu}} M_{\sigma}.$$
	
	In order to check that $\Upsilon_A(M)\in \sC_A'$, we need to check conditions (A), (B), (C) and (D).
	
	{\bf Condition (A):} For $\lambda+\bZ I\in X/\bZ I$, we have $$\Upsilon_A(M)_{\lambda+\bZ I} A=\left(\bigoplus_{\sigma\in\lambda+\bZ I} M_\sigma \right)A\subseteq \bigoplus_{\sigma\in\lambda+\bZ I} (M_\sigma A)\subseteq\bigoplus_{\sigma\in\lambda+\bZ I} M_\sigma=\Upsilon_A(M)_{\lambda+\bZ I}.$$
	
	{\bf Condition (B):} The set of $\lambda+\bZ I\in X/\bZ I$ such that $\Upsilon_A(M)_{\lambda+\bZ I}\neq 0$ is the set of $\lambda+\bZ I$ such that there exists $\sigma\in\lambda+\bZ I$ with $M_\sigma\neq 0$. Since there are only finitely many such $\sigma$, this set is finite. Furthermore, since each $M_\sigma$ is finitely generated over $A$ and there are only finitely many $\sigma\in\lambda+\bZ I$ with $M_\sigma\neq 0$, we see that $\Upsilon_A(M)$ is finitely generated over $A$.
	
	{\bf Condition (C):} Since $(U^0U^{+})_{\lambda+\bZ I}=\bigoplus_{\nu\in\lambda+\bZ I} (U^0U^{+})_{\nu}$ we have  $$(U^0U^{+})_{\sigma+\bZ I}\Upsilon_A(M)_{\lambda+\bZ I}\subseteq \sum_{\substack{\varepsilon\in\sigma+\bZ I\\\nu\in\lambda+\bZ I}}(U^0U^{+})_{\varepsilon}M_{\nu}\subseteq \sum_{\substack{\varepsilon\in\sigma+\bZ I\\\nu\in\lambda+\bZ I}} M_{\varepsilon+\nu}\subseteq\sum_{\tau\in\sigma+\lambda+\bZ I}M_\tau=\Upsilon_A(M)_{\sigma+\lambda+\bZ I}.$$
	The last inclusion follows because if $\varepsilon\in\sigma+\bZ I$ and $\nu\in\lambda+\bZ I$ then $\varepsilon+\nu\in\sigma+\lambda+\bZ I$.
	
	{\bf Condition (D):} First, we need to see that for $\lambda+\bZ I\in X/\bZ I$ we have, $$\Upsilon_A(M)_{\lambda+\bZ I}=\bigoplus_{\substack{d\mu\in\fh^{*}\\\mu\in\lambda+\bZ I+pX}}\Upsilon_A(M)_{\lambda+\bZ I}^{d\mu}.$$
	The left-hand side is equal to
	$$\Upsilon_A(M)_{\lambda+\bZ I}=\bigoplus_{\mu\in\lambda+\bZ I}M_\mu,$$
	while the right-hand side is equal to 
	$$\bigoplus_{\substack{d\mu\in\fh^{*}\\\mu\in\lambda+\bZ I+pX}}\Upsilon_A(M)_{\lambda+\bZ I}^{d\mu}=\bigoplus_{\substack{d\mu\in\fh^{*}\\\mu\in\lambda+\bZ I+pX}}\bigoplus_{\substack{\sigma\in \lambda+\bZ I \\ d\sigma=d\mu}} M_\sigma.$$
	Equality will follow once we note that $$\lambda+\bZ I=\{\sigma\in\lambda+\bZ I\,\vert\,d\sigma=d\mu\,\mbox{for some}\,\mu\in\lambda+\bZ I+pX\},$$ which is easy to see. Furthermore, the right-hand side is indeed a direct sum since we index the summands by elements of $\fh^{*}$, not of $X$.
	
	Now, for $\lambda+\bZ I\in X/\bZ I$ and $\mu\in\lambda+\bZ I+pX$, suppose $m\in \Upsilon_A(M)_{\lambda+\bZ I}^{d\mu}$ and so $$m=\sum_{\substack{\sigma\in\lambda+\bZ I\\d\sigma=d\mu}}m_{\sigma}$$ for $m_\sigma\in M_{\sigma}$. Let $s\in U^{0}$. Then we have $$sm=\sum_{\substack{\sigma\in\lambda+\bZ I\\d\sigma=d\mu}}sm_{\sigma}= \sum_{\substack{\sigma\in\lambda+\bZ I\\d\sigma=d\mu}}(m_{\sigma}\pi(\widetilde{\sigma}(s))).$$ Note that $\widetilde{\sigma}=\widetilde{\mu}$, since $\widetilde{\mu}(h)=h+d\mu(h)=h+d\sigma(h)=\widetilde{\sigma}(h)$ for all $h\in\fh$. Hence, we get that $$sm= \sum_{\substack{\sigma\in\lambda+\bZ I\\d\sigma=d\mu}}(m_{\sigma}\pi(\widetilde{\mu}(s)))=\left(\sum_{\substack{\mu\in\lambda+\bZ I\\d\sigma=d\mu}}m_{\sigma}\right)\pi(\widetilde{\mu}(s))=m\pi(\widetilde{\mu}(s)).$$
	We furthermore have that, for $\alpha\in R$, $$e_\alpha m=\sum_{\substack{\sigma\in\lambda+\bZ I\\d\sigma=d\mu}}e_\alpha m_{\sigma}\in \bigoplus_{\substack{\sigma\in\lambda+\bZ I\\d\sigma=d\mu}}M_{\sigma+\alpha}=\bigoplus_{\substack{\tau\in\lambda+\alpha+\bZ I\\ d\tau=d(\mu+\alpha)}}M_\tau=\Upsilon_A(M)_{\lambda+\alpha+\bZ I}^{d(\mu+\alpha)}.$$ 
	Finally, it is easy to see that $ma\in \Upsilon_A(M)_{\lambda+\bZ I}^{d\mu}$ for all $a\in A$. Hence, condition (D) holds.
	
	In conclusion, we indeed have that $\Upsilon_A(M)\in \sC_A'$. Furthermore, it is clear from the construction that $\Upsilon_A$ sends homomorphisms to homomorphisms.
\end{proof}

We may similarly define $\sC_A''$ to be the category of $U^0\otimes A$-modules with an $X/\bZ I$-grading, which satisfy conditions appropriately analogous to conditions (A), (B), (C) and (D), and also define $\widehat{\sC_A''}$ to be the corresponding category with $X$-gradings instead (so that $\widehat{\sC_A''}$ is the category that was called $\sC_A''$ in \cite{AJS}). Once again, there is a well-defined functor $\Upsilon_A:\widehat{\sC_A''}\to\sC_A''$ defined in the same way as above. 

\begin{prop}\label{U0Equiv}
	There is an equivalence of categories between $\sC_A''$ and the category of finitely-generated $X/p\bZ I$-graded $A$-modules.
\end{prop}

\begin{proof}
	We saw in Lemma~\ref{XZIvXpZI} that $\sC_A$ and $\widetilde{\sC_A}$ are equivalent categories. A similar argument shows that $\sC_A''$ and $\widetilde{\sC_A''}$, where the latter is defined analogously to $\widetilde{\sC_A}$ for $U^0\otimes A$-modules, are equivalent. That the forgetful functor from $\widetilde{\sC_A''}$ to the category of finitely-generated $X/p\bZ I$-graded $A$-modules is an equivalence of categories then follows from a similar argument to that of Lemma 2.5 in \cite{AJS}.
\end{proof}

\begin{cor}\label{ProjA}
	Modules in $\sC_A''$ are projective if and only if they are projective as $A$-modules. Hence, $\sC_A''$ has enough projectives.
\end{cor}

\section{Induction}\label{Sec4}
\subsection{Induction along $U^0$}\label{Sec4.1}

Induction is one of the key tools in representation theory. Since we have defined a number of categories at this point, let us use induction to see how they all fit together.

For each $M\in\sC_A''$, we define $$\Phi_A'(M)\coloneqq U^0U^{+}\otimes_{U^0}M.$$ We let $u\in U^0 U^{+}$ act on the first factor via left multiplication, while $a\in A$ acts on the right as it does on $M$. We then want to show that $\Phi_A'(M)$ lies in $\sC_A'$. For $\lambda+\bZ I\in X/\bZ I$ we define $$(U^0U^{+}\otimes_{U^0} M)_{\lambda+\bZ I}=\bigoplus_{\nu+\bZ I\in X/\bZ I} (U^{0}U^{+})_{\nu+\bZ I}\otimes_{U^0} M_{\lambda-\nu+\bZ I},$$ and we define, for $\mu\in\lambda+\bZ I+pX$,
$$(U^0U^{+}\otimes_{U^0}M)_{\lambda+\bZ I}^{d\mu}=\bigoplus_{\nu+\bZ I\in X/\bZ I} \bigoplus_{\substack{d\sigma\in \fh^{*} \\ \sigma\in \nu+\bZ I +pX}}(U^{0}U^{+})_{\nu+\bZ I}^{d\sigma}\otimes_{U^0} M_{\lambda-\nu+\bZ I}^{d(\mu-\sigma)}.$$
Since $U^0$ preserves each $M_{\lambda+\bZ I}^{d\mu}$ and $(U^{0}U^{+})_{\nu+\bZ I}^{d\sigma}$, the tensor products are well-defined. Furthermore, $\sigma\in\nu+\bZ I+pX$ implies $\lambda-\sigma\in \lambda-\nu+\bZ I+pX$, so the right-hand side of the tensor product in the second decomposition makes sense. The distributivity of direct sums over tensor products shows that we indeed have direct sum decompositions of the relevant spaces. It is easy to see that conditions (A), (B) and (C) hold for $\Phi_A'(M)$. For condition (D), we let $s\in U^{0}$, $u\in(U^0U^{+})_{\lambda+\bZ I}^{d\mu}$, and $m\in M_{\nu+\bZ I}^{d\sigma}$ (where $\mu\in\lambda+\bZ I+pX$ and $\sigma\in\nu+\bZ I + pX$). Then $$s(u\otimes m)=(su)\otimes m=(u\widetilde{\mu}(s))\otimes m=u\otimes (\widetilde{\mu}(s)m)=u\otimes (m\pi(\widetilde{\sigma}\widetilde{\mu}(s)))=(u\otimes m)\pi(\widetilde{(\mu+\sigma)}(s)).$$

Checking easily that $\Phi_A'$ is compatible with morphisms, we hence have obtained a functor $\Phi_A':\sC_A''\to \sC_A'$. We may similarly define a functor $$\Phi_A:\sC_A''\to\sC_A,\\ \quad\Phi_A(M)=U_\chi\otimes_{U^0}M.$$

\begin{lemma}\label{PhiProj}
	The functors $\Phi_A$ and $\Phi_A'$ are exact and are left adjoint to the corresponding forgetful functors. Therefore, $\Phi_A$ and $\Phi_A'$ both map projective modules to projective modules.
\end{lemma}

\begin{proof}
	This can easily be adapted from the proof of Lemma 2.7 in \cite{AJS}, since $U^0U^{+}$ is free over $U^0$.	
\end{proof}

\begin{lemma}
	The categories $\sC_A$ and $\sC_A'$ have enough projectives.
\end{lemma}

\begin{proof}
	We have already seen in Lemma~\ref{ProjA} that $\sC_A''$ has enough projectives. The proof then works as in Lemma 2.7 in \cite{AJS}.
\end{proof}


\subsection{Baby Verma modules}\label{Sec4.2}

There is another type of induction functor in this area which is key in the modular representation theory of Lie algebras and which we need in order to define {\bf baby Verma modules}.
We aim to define the induction functor $Z_{A,\chi}:\sC_A'\to\sC_A$ by sending $M\in\sC_A'$ to $$Z_{A,\chi}(M)\coloneqq U_\chi\otimes_{U^0U^{+}}M.$$

Observe, as in Section 2.10 of \cite{AJS}, that we have as $\bK$-vector spaces that $$Z_{A,\chi}(M)\cong U^{-}\otimes_{\bK} M.$$ We therefore define, for $\lambda+\bZ I\in X/\bZ I$, $$Z_{A,\chi}(M)_{\lambda+\bZ I}=\bigoplus_{\nu+\bZ I\in X/\bZ I} (U^{-})_{\nu+\bZ I}\otimes M_{\lambda-\nu+\bZ I}$$ and, for $\mu\in \lambda+\bZ I+pX$, $$Z_{A,\chi}(M)_{\lambda+\bZ I}^{d\mu}=\bigoplus_{\nu+\bZ I\in X/\bZ I}\bigoplus_{\substack{d\sigma\in \fh^{*} \\ \sigma\in \nu+\bZ I +pX}} (U^{-})_{\nu+\bZ I}^{d\sigma}\otimes M_{\lambda-\nu+\bZ I}^{d(\mu-\sigma)}.$$

Since the $U^0U^{+}$-action on $M$ is compatible with the grading, by conditions (C) and (D), the natural surjection $U\otimes M\to U\otimes_{U^0U^{+}}M$ induces a surjection
$$\bigoplus_{\nu+\bZ I\in X/\bZ I} (U_\chi)_{\nu+\bZ I}\otimes M_{\lambda-\nu+\bZ I}\twoheadrightarrow Z_{A,\chi}(M)_{\lambda+\bZ I}$$ and a surjection
$$\bigoplus_{\nu+\bZ I\in X/\bZ I}\bigoplus_{\substack{d\sigma\in \fh^{*} \\ \sigma\in \nu+\bZ I +pX}} (U_\chi)_{\nu+\bZ I}^{d\sigma}\otimes M_{\lambda-\nu+\bZ I}^{d(\mu-\sigma)}\twoheadrightarrow Z_{A,\chi}(M)_{\lambda+\bZ I}^{d\mu}.$$
This implies that conditions (C) and (D) are satisfied for $Z_{A,\chi}(M)$, and so we indeed have $Z_{A,\chi}(M)\in\sC_A$. It is clear that this process sends morphisms to morphisms, and thus we have a functor $\sC_A'\to\sC_A$. The following lemma can be proved in the standard way.

\begin{lemma}\label{ZFrobRec}
	Let $M\in\sC_A'$ and $N\in\sC_A$. Then $$\Hom_{\sC_A}(Z_{A,\chi}(M),N)\cong \Hom_{\sC_A'}(M,N).$$
\end{lemma}
The functor $Z_{A,\chi}$ goes from $\sC_A'$ to $\sC_A$, but as in the usual representation theory of $U_\chi(\fg)$ we shall be mostly interested in applying it to modules most naturally thought of as lying in $\sC_A''$. We fix this incongruity by constructing a more-or-less trivial functor $\sC_A''\to\sC_A'$ in the following way.

There is a natural algebra surjection $U^0U^{+}\twoheadrightarrow U^0$, and so we may view a $U^0\otimes A$-module $M$ as a $U^0U^{+}\otimes A$-module via this surjection. We want to show that this procedure sends modules in $\sC_A''$ to modules in $\sC_A'$. If $M\in\sC_A''$ then, as a $\bK$-module, $M$ is clearly still $X/\bZ I$-graded when viewed as a $U^0U^{+}\otimes A$-module, and conditions (A) and (B) follow from those conditions in $\sC_A''$. Furthermore, if $\mu+\bZ I\neq 0$, then $(U^0U^{+})_{\mu+\bZ I}$ is in the kernel of the surjection $U^0U^{+}\twoheadrightarrow U^0$, so condition (C) is satisfied. Condition (D) follows easily, and so, in sum, we have the following:

\begin{prop}
	There exists a fully-faithful functor $\sC_A''\to\sC_A'$ sending $M$ to $M$.
\end{prop}

For each $\lambda\in X$, we define an object $A^\lambda\in\sC_A''$ as follows: As a right $A$-module, it is just the algebra $A$ with the usual right multiplication. We set $(A^{\lambda})_{\sigma+\bZ I}$ to be $A$ if $\lambda\in\sigma+\bZ I$ and zero otherwise, and we set $(A^{\lambda})_{\lambda+\bZ I}^{d\tau}$ to be $A$ if $d\tau=d\lambda$ and zero otherwise. We then define $sa=a\pi(\widetilde{\mu}(s))$ for $a\in A^\lambda$ and $s\in U^0$. Hence, $A^\lambda\in\sC_A''$ and so we may define the {\bf baby Verma module} $$Z_{A,\chi}(\lambda)\coloneqq Z_{A,\chi}(A^\lambda)\in\sC_A.$$ Similarly, we define $\Phi_A(\lambda)\coloneqq\Phi_A(A^\lambda)$ and $\Phi_A'(\lambda)\coloneqq\Phi_A'(A^\lambda)$. Baby Verma modules have the particular strength that it is relatively easy to see when there exists a map from a baby Verma module to another module, as we now see.

Let $M\in\sC_A$. As in Lemma~\ref{ZFrobRec}, for $\lambda\in X$ we have $$\Hom_{\sC_A}(Z_{A,\chi}(\lambda),M)\cong\Hom_{\sC_A'}(A^\lambda,M).$$
It is easy to see that $$\Hom_{\sC_A'}(A^\lambda,M)\cong (M_{\lambda+\bZ I}^{d\lambda})^{\fn^{+}},$$ where $$(M_{\lambda+\bZ I}^{d\lambda})^{\fn^{+}}\coloneqq \{m\in M_{\lambda+\bZ I}^{d\lambda}\,\vert\, e_{\alpha}m=0\,\,\mbox{for all}\,\,\alpha\in R^{+}\}.$$ In other words, there exists a non-zero homomorphism $Z_{A,\chi}(\lambda)\to M$ in $\sC_A$ if and only if there exists a non-zero $m\in M_{\lambda+\bZ I}^{d\lambda}$ with $e_\alpha m=0$ for all $\alpha\in R^{+}$.

\begin{prop}\label{ZSurj}
	Let $M\in\sC_A$. Then there exists $\lambda\in X$ such that there is a homomorphism $Z_{A,\chi}(\lambda)\to M$ in $\sC_A$.
\end{prop}
\begin{proof}
	Let $m\in M$ with $e_\alpha m=0$ for all $\alpha\in R^{+}$. Then $m$ can be written $m=\sum_{\tau+p\bZ I\in X/p\bZ I} m_{\tau+\bZ I}^{d\tau}$ with $m_{\tau+\bZ I}^{d\tau}\in M_{\tau+\bZ I}^{d\tau}$. So, $0=e_\alpha m=\sum_{\tau+p\bZ I\in X/p\bZ I} e_\alpha m_{\tau+\bZ I}^{d\tau}$ with $e_\alpha m_{\tau+\bZ I}^{d\tau}\in M_{\tau+\alpha+\bZ I}^{d(\tau+\alpha)}$ for all $\tau+p\bZ I\in X/\bZ I$. Since $\tau+\alpha+p\bZ I=\sigma+\alpha+p\bZ I$ implies $\tau+p\bZ I=\sigma+p\bZ I$, we get that $e_\alpha m=0$ implies that $e_\alpha m_{\tau+\bZ I}^{d\tau}=0$ for all $\tau+p\bZ I$. So to get the result, it suffices to show that there exists $m\in M$ with $e_\alpha m=0$ for all $\alpha\in R^{+}$. This fact follows from an inductive argument as in \cite[B.3]{Jan2}.

\end{proof}

\subsection{$Z$-Filtrations}\label{Sec4.3}
Let us say that $M\in\sC_A$ has a {\bf $Z$-filtration} if it has a filtration of submodules $$0=M_0\subset M_1\subset M_2\subset\cdots\subset M_{r-1}\subset M_r=M$$ such that each $M_{i}/M_{i-1}\cong Z_{A,\chi}(\lambda_i)$ for some $\lambda_i\in X$. In order to show that certain modules have $Z$-filtrations, it will be useful to see how our induction functors interact with the functor $\Upsilon_A$ defined in Proposition~\ref{UpsilonDef}. To do this, recall the functor which we will call $\widehat{\Phi_A'}$ but was denoted $\Phi_A'$ in \cite{AJS}, which is a functor $\widehat{\sC_A''}\to\widehat{\sC_A'}$ (using our notation for the categories) defined analogously to our functor $\Phi_A':\sC_A''\to\sC_A'$.

\begin{prop}\label{commdiag}
	The following diagram commutes:
	$$\xymatrix{
		\widehat{\sC_A''}\ar@{->}[rr]^{\Upsilon_A} \ar@{->}[d]^{\widehat{\Phi_A'}} & & \sC_A'' \ar@{->}[d]^{\Phi_A'}& & \\
		\widehat{\sC_A'} \ar@{->}[rr]^{\Upsilon_A} & & \sC_A'. & & \\
	}
	$$
\end{prop}

\begin{proof}
	The diagram clearly commutes as a diagram of ungraded modules (since the maps on the left and the right are the same, and the horizontal arrows are identities, for ungraded modules). So all that must be shown is that the decompositions correspond. In other words, we must show for $M\in\widehat{\sC_A''}$, and for each $\lambda+\bZ I\in X/\bZ I$ and $\mu\in\lambda+\bZ I+pX$, that $$\Phi_A'(\Upsilon_A(M))_{\lambda+\bZ I}=\Upsilon_A(\widehat{\Phi_A'}(M))_{\lambda+\bZ I}$$
	and 
	$$\Phi_A'(\Upsilon_A(M))_{\lambda+\bZ I}^{d\mu}=\Upsilon_A(\widehat{\Phi_A'}(M))_{\lambda+\bZ I}^{d\mu}.$$
	
	We may see that $$\Phi_A'(\Upsilon_A(M))_{\lambda+\bZ I}=\bigoplus_{\nu+\bZ I\in X/\bZ I}\bigoplus_{\substack{\sigma\in\nu+\bZ I\\ \gamma\in\lambda-\nu+\bZ I}}(U^0U^+)_{\sigma}\otimes_{U^0}M_\gamma$$ 
	and $$\Upsilon_A(\widehat{\Phi_A'}(M))_{\lambda+\bZ I}=\bigoplus_{\tau\in\lambda+\bZ I}\bigoplus_{\nu\in X} (U^0U^+)_\nu\otimes_{U^0} M_{\tau-\nu}.$$
	
	Then we must show that $$\{(\sigma,\gamma)\in(\nu+\bZ I)\times (\lambda-\nu+\bZ I)\,\,\mbox{for some}\,\,\nu+\bZ I\in X/\bZ I\} = \{(\nu,\tau-\nu)\in X\times X\,\,\mbox{for some }\,\tau\in\lambda+\bZ I\}.$$
	
	If $(\nu,\tau-\nu)$ is in the right-hand side, so $\nu\in X$ and $\tau\in\lambda+\bZ I$, then $(\nu,\tau-\nu)\in (\nu+\bZ I)\times (\lambda-\nu+\bZ I)$, and so $(\nu,\tau-\nu)$ also lies in the left-hand side. Conversely, if $(\sigma,\gamma)$ is in the left-hand side, so $\sigma\in\nu+\bZ I$ and $\gamma\in\lambda-\nu+\bZ I$ for some $\nu+\bZ I\in X/\bZ I$, then $(\sigma,\gamma)=(\sigma,(\gamma+\sigma)-\sigma)$ where $\gamma+\sigma\in \lambda+\bZ I$. Hence, $(\sigma,\gamma)$ lies in the right-hand side, and we have equality.
	
	We may also see that $$\Phi_A'(\Upsilon_A(M))_{\lambda+\bZ I}^{d\mu}=\bigoplus_{\nu+\bZ I\in X/\bZ I}\bigoplus_{\substack{d\sigma\in\fh^{*}\\ \sigma\in\nu+\bZ I+pX}}\bigoplus_{\substack{\tau\in\nu+\bZ I\\ d\tau=d\sigma}}\bigoplus_{\substack{\varepsilon\in\lambda-\nu+\bZ I\\ d\varepsilon=d(\mu-\sigma)}}(U^0U^+)_{\tau}\otimes_{U^0}M_\varepsilon$$ 
	and $$\Upsilon_A(\widehat{\Phi_A'}(M))_{\lambda+\bZ I}^{d\mu}=\bigoplus_{\substack{ \sigma\in\lambda+\bZ I\\ d\sigma=d\mu}}\bigoplus_{\tau\in X} (U^0U^+)_\tau\otimes_{U^0} M_{\sigma-\tau}.$$
	
	Now, set $$A_1=\left\{(\tau,\varepsilon)\in X\times X \middle\vert\begin{array}{l}\,\qquad\exists \nu+\bZ I\in X/\bZ I\,\mbox{and}\,\,d\sigma\in\fh^{*}\,\,\mbox{with}\,\,\sigma\in\nu+\bZ I + pX\\ \mbox{such that} \,\,\tau\in\nu+\bZ I,\,\,d\tau=d\sigma,\,\, \varepsilon\in\lambda-\nu+\bZ I, \,\,d\varepsilon=d(\mu-\sigma) \end{array}\right\}$$ and 
	$$A_2=\left\{(\tau,\sigma-\tau)\in X\times X\,\vert\,\,\sigma\in\lambda+\bZ I\,\mbox{and}\,\,d\sigma=d\mu\right\}.$$
	
	Then we must show that $A_1=A_2$. If $(\tau,\varepsilon)\in A_1$ then $(\tau,\varepsilon)\in A_2$ since $\varepsilon=(\varepsilon+\tau)-\tau$, where $\varepsilon+\tau\in\lambda+\bZ I$ and $d(\varepsilon+\tau)=d\mu$. Conversely, if $(\tau,\sigma-\tau)\in A_2$, so $\sigma\in \lambda+\bZ I$ and $d\sigma=d\mu$, then $(\tau,\sigma-\tau)\in A_2$ since $\tau\in\tau+\bZ I$, $d\tau=d\tau$, $\sigma-\tau\in\lambda-\tau+\bZ I$ and $d(\sigma-\tau)=d(\lambda-\tau)$. Thus, $A_1=A_2$ as required.

\end{proof}

\begin{cor}
	Let $M\in\widehat{\sC_A'}$ such that $M_\sigma$ is free over $A$ for each $\sigma\in X$. Then $\Upsilon_A(M)$ has a filtration with factors $A^\lambda$ for $\lambda\in X$. Furthermore, $Z_{A,\chi}(\Upsilon_A(M))$ has a $Z$-filtration.
\end{cor}

\begin{proof}
	We know that each $M\in\widehat{\sC_A'}$ with the given property has a filtration with sections $A^\lambda$ for $\lambda\in X$ (see Lemma 2.12 in \cite{AJS}). It is clear that $\Upsilon_A(A^\lambda)=A^\lambda$ (abusing notation to denote by $A^\lambda$ the analogous object in each category), and it is straightforward to see that $\Upsilon_A$ is exact. Hence, $\Upsilon_A(M)$ has the desired filtration. Since $Z_{A,\chi}$ is also exact, $Z_{A,\chi}(\Upsilon_A(M))$ has a $Z$-filtration.
\end{proof}

\begin{cor}\label{PhiFilt}
	Let $M\in\widehat{\sC_A''}$ such that $M_\lambda$ is free over $A$ for each $\lambda\in X$. Then $\Phi_{A}(\Upsilon_A(M))$ has a $Z$-filtration. In particular, $\Phi_A(\lambda)$ has a $Z$-filtration for each $\lambda\in X$.
\end{cor}

\begin{proof}
	This follows from the fact that $\Phi_{A}(\Upsilon_A(M))\cong Z_{A,\chi}(\Phi_A'(\Upsilon_A(M)))=Z_{A,\chi}(\Upsilon_A(\widehat{\Phi_A'}(M)))$, using Proposition~\ref{commdiag} for the last equality.
\end{proof}

\begin{cor}\label{ZZFilt}
	Let $\lambda\in X$. There exists a projective module $Q\in\sC_A$ which surjects onto $Z_{A,\chi}(\lambda)$, and has a $Z$-filtration. Specifically, $Q=\Phi_A(\lambda)$.
\end{cor}

\begin{proof}
	There is, in $\widehat{\sC_A'}$, a surjection $\widehat{\Phi_A'}(\lambda)\to A^\lambda$. This induces a surjection $\Upsilon_A(\widehat{\Phi_A'}(\lambda))\to \Upsilon_A(A^\lambda)$, i.e., a surjection $\Phi_A'(\lambda)\to A^\lambda$. This, in turn, induces a surjection $\Phi_A(\lambda)=Z_{A,\chi}(\Phi_A'(\lambda))\to Z_{A,\chi}(\lambda)$. The object $\Phi_A(\lambda)$ has a $Z$-filtration by Corollary~\ref{PhiFilt}. Furthermore, $\Phi_A(\lambda)$ is projective since $\Phi_A$ maps projectives to projectives by Lemma~\ref{PhiProj} and $A^\lambda$ is projective in $\sC_A''$ by Corollary~\ref{ProjA}.
\end{proof}

This in fact holds more generally:

\begin{theorem}
	Let $M\in\sC_A$. There exists a projective module $Q\in\sC_A$ which surjects onto $M$, and which has a $Z$-filtration.
\end{theorem}

\begin{proof}
	
	We may restrict $M$ to an element of $\sC_A''$. Since $\sC_{A}''$ has enough projectives, we can find a projective $P\in\sC_A''$ with $P\twoheadrightarrow M$. Since being projective in $\sC_A''$ is equivalent to each $P_{\lambda+\bZ I}^{d\mu}$ being projective as an $A$-module, we may assume that each $P_{\lambda+\bZ I}^{d\mu}$ is free over $A$ (so $P$ is free over $A$). 
	
	If $P=\Upsilon_A(\widehat{P})$ for some $\widehat{P}\in\widehat{\sC_A''}$ with each $\widehat{P}_\lambda$ free over $A$, then by Corollary~\ref{ZZFilt} we have that $\Phi_A(P)$ has a $Z$-filtration. Since $$\Hom_{\sC_A}(\Phi_A(P),M)\cong\Hom_{\sC_A''}(P,M)$$ we would then get a map $\Phi_A(P)\to M$ in $\sC_A$, which would be surjective. Furthermore, as $\Phi_A$ sends projectives to projectives, we may take $Q=\Phi_A(P)$ to get the result.
	
	All that remains, therefore, is to find $\widehat{P}$. The following lemma tells us that we can do this.
\end{proof}

\begin{lemma}\label{Xpreim}
	Let $M\in \sC_A''$. Then there exists $\widehat{M}\in\widehat{\sC_A''}$ with $\Upsilon_A(\widehat{M})=M$. Furthermore, if each $M_{\lambda+\bZ I}^{d\mu}$ is free over $A$ then $\widehat{M}$ may be chosen with each $\widehat{M}_\tau$ free over $A$.
\end{lemma}

\begin{proof}
	
	We have already observed in Proposition~\ref{U0Equiv} that $\sC_A''$ is equivalent (via the forgetful functor) to the category $\sG$ of finitely-generated $X/p\bZ I$-graded $A$-modules, and Lemma 2.5 in \cite{AJS} shows that $\widehat{\sC_A''}$ is equivalent (also via the forgetful functor) to the category $\widehat{\sG}$ of finitely-generated $X$-graded $A$-modules.
	
	There is a natural morphism $\Upsilon_{A,\sG}:\widehat{\sG}\to\sG$ such that $\Upsilon_{A,\sG}(M)=M$ as $A$-modules and such that $\Upsilon_{A,\sG}(M)_{\lambda+p\bZ I}=\bigoplus_{\sigma\in\lambda+p\bZ I} M_\sigma$ for each $\lambda+p\bZ I\in X/p\bZ I$. It is straightforward to see that $\Upsilon_A$ and $\Upsilon_{A,\sG}$ are compatible with the equivalences of categories induced by the forgetful functors. Therefore, it is enough to show that, for each $M\in\sG$, there exists $\widehat{M}\in\widehat{\sG}$ with $\Upsilon_{A,\sG}(\widehat{M})=M$.
	
	Define a map $f:X/p\bZ I\to X$ which, for each $\lambda+p\bZ I\in X/p\bZ I$, picks a representative $\lambda\in\lambda+p\bZ I$ (using the axiom of choice). Then, define $\widehat{M}\in\widehat{\sG}$ to be equal to $M$ as an $A$-module, and to have $X$-grading $$\widehat{M}_\sigma=\twopartdef{M_{\sigma+p\bZ I}}{\sigma\in f(X/p\bZ I),}{0}{\mbox{otherwise.}}$$ This clearly lies in $\widehat{\sG}$, and it is straightforward to see that $\Upsilon_{A,\sG}(\widehat{M})=M$.
	
	Finally, it is clear from the construction that if each $M_{\lambda+p\bZ I}$ is free over $A$ then each $\widehat{M}_\sigma$ is free over $A$.
	
\end{proof}

\subsection{Induction from parabolic subcategories}\label{Sec4.4}

The induction functors discussed so far are all analogous to functors defined in \cite{AJS}. For $\chi$ non-zero, however, we may also define some new functors (which coincide with the previously discussed functors when $I=\emptyset$). In order to do this, we start by defining categories from which we can induce.

Write $\sC_A^I$ for the category corresponding to $\sC_A$ for the Lie algebra $\fg_I$ with $\chi_I=\chi\vert_{\fg_I}$ (i.e. the category of $U^I\otimes A$-modules with an $X/\bZ I$-grading and the relevant conditions). Similarly, we write $\sC_A^{I,+}$ for the category corresponding to $\sC_A$ for $\fp=\fg_I\oplus\fu^{+}$ (i.e. the category of $U^IU_I^{+}\otimes A$-modules with an $X/\bZ I$-grading and the relevant conditions). Note that $\sC_A^I$ has enough projectives by the same argument as for $\sC_A$.

Let $M\in\sC_A^{I,+}$ and define $\Gamma_{A,\chi}:\sC_A^{I,+}\to\sC_A$ by $$\Gamma_{A,\chi}(M)=U_\chi\otimes_{U^I U_I^{+}}M.$$ This is a $U_\chi$-module via left multiplication and a right $A$-module via the $A$-action on $M$. We now want to see that $\Gamma_{A,\chi}(M)$ lies in $\sC_A$. For the $X/\bZ I$-grading, we note that as $\bK$-vector spaces, we have $$U_\chi\cong U_I^{-}\otimes U^I U_I^{+}$$ and so $$\Gamma_{A,\chi}(M)\cong U_I^{-}\otimes M.$$

We then define, for $\lambda+\bZ I\in X/\bZ I$, $$\Gamma_{A,\chi}(M)_{\lambda+\bZ I}=\bigoplus_{\nu+\bZ I\in X/\bZ I} (U_I^{-})_{\nu+\bZ I}\otimes M_{\lambda-\nu+\bZ I}$$ and, for $\mu \in \lambda+\bZ I+pX$, we set 
$$\Gamma_{A,\chi}(M)_{\lambda+\bZ I}^{d\mu}=\bigoplus_{\nu+\bZ I\in X/\bZ I}\bigoplus_{\substack{d\sigma\in \fh^{*} \\ \sigma\in \nu+\bZ I +pX}} (U_I^{-})_{\nu+\bZ I}^{d\sigma}\otimes M_{\lambda-\nu+\bZ I}^{d(\lambda-\sigma)}.$$


We can show that this indeed defines an object in $\sC_A$ by using an almost identical argument to the argument that $Z_{A,\chi}$ send $\sC_A'$ to $\sC_A$ used in Subsection~\ref{Sec4.2}. Since $\Gamma_{A,\chi}$ clearly sends morphisms to morphisms, it is a functor $\Gamma_{A,\chi}:\sC_A^{I,+}\to\sC_A$.
%

For baby Verma modules, we saw that each $M\in\sC_A''$ may viewed as an object in $\sC_A'$ via the surjection $U^0U^{+}\twoheadrightarrow U^0$. A similar argument shows that each $M\in\sC_A^I$ may be viewed as an object in $\sC_A^{I,+}$, using the surjection $U^I U_I^{+}\twoheadrightarrow U^I$. This does indeed give an object of $\sC_A^{I,+}$: conditions (A), (B) and (D) are obvious, and condition (C) follows from the fact that each $(U^I U_I^{+})_{\lambda+\bZ I}$ with $\lambda\notin\bZ I$ lies in the kernel of the surjection.


We may also define the functor $\Phi_A^{I,+}:\sC_A^I\to\sC_A^{I,+}$ by $$\Phi_A^{I,+}(M)=U^IU_I^{+}\otimes_{U^I} M.$$ Similar to the above, one can show that $\Phi_A^{I,+}(M)\in\sC_A^{I,+}$. Note that, unless $I=\emptyset$, we do not have that $U^I\subseteq (U_\chi)_{0+\bZ I}^{d0}$. Therefore, the proof that $\Phi_A^{I,+}$ is well-defined, and the definition of the gradings and (D)-decompositions, more closely resemble the analogous proofs and definitions for $\Gamma_{A,\chi}$ than for $\Phi_A'$.

Finally, we define the functor:
$$\Phi_A^I=\Gamma_{A,\chi}\circ\Phi_A^{I,+}:\sC_A^I\to\sC_A,\qquad M\mapsto U_\chi\otimes_{U^I}M.$$ We may easily check the following:

\begin{prop}\label{FrobRecPara}
	Let $L\in\sC_A^I$, $M\in\sC_A^{I,+}$ and $N\in\sC_A$. Then $$\Hom_{\sC_{A}}(\Gamma_{A,\chi}(M),N)\cong\Hom_{\sC_{A}^{I,+}}(M,N),$$
	$$\Hom_{\sC_A^{I,+}}(\Phi_A^{I,+}(L),M)\cong\Hom_{\sC_{A}^I}(L,M)$$ and $$\Hom_{\sC_A}(\Phi_A^I(L),N)\cong\Hom_{\sC_A^I}(L,N).$$
\end{prop}
\begin{cor}
	The functors $\Gamma_{A,\chi}$, $\Phi_A^{I,+}$ and $\Phi_A^I$ are exact and map projectives to projectives. In particular, $\sC_A^{I,+}$ has enough projectives.
\end{cor}
%
%


Given $\lambda\in X$, let us write $Z_{A,I,\chi}(\lambda)\in\sC_A^I$ for the {\bf baby Verma module} in $\sC_A^I$ corresponding to $\lambda$. It is then easy to check that, viewing $Z_{A,I,\chi}(\lambda)$ as an object of $\sC_A^{I,+}$, $$\Gamma_{A,\chi}(Z_{A,I,\chi}(\lambda))\cong Z_{A,\chi}(\lambda)\in\sC_A.$$

Furthermore, for any $M\in\sC_A^I$ with $M_{\lambda+\bZ I}=M$, one may straightforwardly check that $$\Phi_A^{I,+}(M)_{\lambda+\bZ I}\cong M\in\sC_A^I$$ and similarly, for any $M\in\sC_A^{I,+}$ with $M_{\lambda+\bZ I}=M$, we have $$\Gamma_{A,\chi}(M)_{\lambda+\bZ I}\cong M\in\sC_A^{I,+}.$$

The following lemma is a direct analogue of Proposition 2.9 in \cite{AJS}.

\begin{lemma}\label{ProjResI}
	Let $M$ be a module in $\sC_A^{I,+}$ that is projective in $\sC_A^I$. Then $M$ has a projective resolution $P_\bullet$ in $\sC^{I,+}$ such that, for each $N\in\sC_A^{I,+}$, 
	$$\Hom_{\sC_A^{I,+}}(P_i,N)=0\quad \mbox{for all}\,\, i>r,$$
	for some $r\geq 0$ depending on $N$.
\end{lemma}

\begin{proof}
	The proof of this lemma works in much the same way as the proof of Proposition 2.9 in \cite{AJS}. We briefly summarise the key differences. We use a group homomorphism $h:X/\bZ I\to \bZ$ with $h(\alpha+\bZ I)>0$ for all $\alpha\in R^{+}\setminus\bZ I$. One needs to employ the functor $\Phi_A^{I,+}$ instead of the functor denoted $\Phi_A'$ in \cite{AJS}. Finally, one has to use the fact that, regarded as a $\bK$-vector space, $\Phi_A^{I,+}(M)$ is the direct sum of all $(U_I^{+})_{\mu+\bZ I}\otimes M_{\lambda+\bZ I}$, the fact that $(U_I^{+})_{0+\bZ I}=\bK$ and the fact that $h(\nu+\bZ I)> 0$ for all other $\nu+\bZ I$ with $(U_I^{+})_{\nu+\bZ I}\neq 0$. 
	
	%
	%
	%
	%
	
\end{proof}

To conclude this section, note that we may also define, analogously to $\sC_A^{I,+}$, the category $\sC_A^{I,-}$, which is the category of $U_I^{-}U^I\otimes A$-modules with an $X/\bZ I$-grading and the relevant conditions (which corresponds to $\fp'=\fu^{-}\oplus\fg_I$). We may then similarly define functors $$\Gamma'_{A,\chi}:\sC_A^{I,-}\to\sC_A,\qquad M\mapsto U_\chi\otimes_{U_I^{-}U^I} M,$$
and
$$\Phi_A^{I,-}:\sC_A^I\to\sC_A^{I,-},\qquad M\mapsto U_I^{-}U^I\otimes_{U^I} M.$$ In particular, we also have $\Phi_A^I=\Gamma'_{A,\chi}\circ \Phi_A^{I,-}$. These also satisfy the Frobenius reciprocities discussed in Proposition~\ref{FrobRecPara}. Furthermore, any $M\in\sC_A^I$ may be viewed as a module in $\sC_A^{I,-}$ by letting $\fu^{-}$ act trivially. This is permissible since $\chi(\fu^{-})=0$. 

For most of this paper, we focus our attention on the category $\sC_A^{I,+}$ and the related functors defined earlier, rather than on $\sC_A^{I,-}$ and the functors just defined. Nonetheless, it will occasionally be useful to have this other category and these other functors as well. Everything we say going forward about the category $\sC_A^{I,+}$ and its associated functors will also apply for $\sC_A^{I,-}$ (suitably adjusted).

\section{Miscellaneous}\label{Sec5}
\subsection{The category $\sC_A^I$}\label{Sec5.1}

We have been talking here about the category $\sC_A^I$ a fair amount, so let us understand it a little more. We observe that $(U^I)_{0+\bZ I}=U^I$ and $(U^I)_{\lambda+\bZ I}=0$ if $\lambda+\bZ I\neq 0+\bZ I$.

Let $M\in \sC_A^I$. Then, by definition, $M=\bigoplus_{\lambda+\bZ I\in X/\bZ I}M_{\lambda+\bZ I}$ as $A$-modules. The observation about the graded structure of $U^I$ then clearly shows that each $M_{\lambda+\bZ I}$ lies in $\sC_A^I$, so the decomposition $M=\bigoplus_{\lambda+\bZ I\in X/\bZ I}M_{\lambda+\bZ I}$ is a decomposition in $\sC_A^I$.

If we write $\sC_A^I(\lambda+\bZ I)$ for the full subcategory of objects $M$ in $\sC_A^I$ with $M_{\lambda+\bZ I}=M$, then this observation can be rewritten as $$\sC_{A}^I=\bigoplus_{\lambda+\bZ I}\sC_A^I(\lambda+\bZ I).$$

Now, suppose that $\lambda\in X$ and $M$ is a $U^I\otimes A$-module with a decomposition $$M=\bigoplus_{\substack{d\mu\in\fh^{*}\\ \mu\in\lambda+\bZ I+pX}} M^{d\mu}$$ such that $$sm=m\pi(\widetilde{\mu}(s))$$ for each $s\in U^0$, $\mu\in\lambda+\bZ I+pX$ and $m\in M^{d\mu}$. Then $M$ can be made into an object in $\sC_A^I(\lambda+\bZ I)$ by placing it entirely in grade $\lambda+\bZ I$ and setting $M_{\lambda+\bZ I}^{d\mu}=M^{d\mu}$. This works, in particular, for the baby Verma modules $Z_{A,I,\chi}(d\lambda)$. (Note that we write $Z_{A,I,\chi}(d\lambda)$ instead of $Z_{A,I,\chi}(\lambda)$ here, since the ungraded $U^I\otimes A$-module depends only on the derivative). Furthermore, since $$M^{d\mu}=\{m\in M\,\vert\, hm=m\pi(\widetilde{\mu}(h))\,\mbox{for all}\, h\in\fh\},$$ any $U^I\otimes A$-module homomorphism $M\to N$ sends each $M^{d\mu}$ to $N^{d\mu}$. Thus, we also get a morphism in $\sC_A^I(\lambda+\bZ I)$. 

Conversely, if $M\in\sC_A(\lambda+\bZ I)$ then $M$ is clearly a $U^I\otimes A$-module with a decomposition as above, setting $M^{d\mu}=M_{\lambda+\bZ I}^{d\mu}$. Furthermore, each morphism in $\sC_A^I(\lambda+\bZ I)$ clearly induces a $U^I\otimes A$-module homomorphism. The category $\sC_A^I(\lambda+\bZ I)$ is then precisely the category of $U^I\otimes A$-modules with a decomposition as above.

\subsection{Truncations}\label{Sec5.2}

The following section should be compared with Chapter 3 of \cite{AJS}.

For each $\nu+\bZ I\in X/\bZ I$, we define $\sC_A(\leq\nu+\bZ I)$ to be the full subcategory of $\sC_A$ consisting of all $M\in \sC_A$ with the property that $M_{\lambda+\bZ I}\neq 0$ only if $\lambda+\bZ I\leq \nu+\bZ I$. It is easy to see that $\sC_A(\leq\nu+\bZ I)$ is closed under submodules, quotients, and extensions, and that it contains $Z_{A,\chi}(\lambda)$ for each $\lambda+\bZ I\in X/\bZ I$ with $\lambda+\bZ I\leq \nu+\bZ I$.


Applying a similar argument to the one in Section 3.6 in \cite{AJS}, one can show that $M/\bigcap_{j\in J} M_j\in \sC_A(\leq\lambda+\bZ I)$ for any collection of $M_j\in\sC_A$ with the property that $M/M_j\in\sC_A(\leq\lambda+\bZ I)$. We may therefore define, for each $M\in \sC_A$, the submodule $O^{\lambda+\bZ I}M$ to be the intersection of all submodules $M'$ of $M$ with the property that $M/M'\in\sC_A(\leq\lambda+\bZ I)$. Then $M/O^{\lambda+\bZ I}M\in\sC_A(\leq\lambda+\bZ I)$, and we will denote this object by $T^{\lambda+\bZ I} M$.

Suppose now that $M\in \sC_A$, $N\in\sC_A(\leq\lambda+\bZ I)$ and that $f:M\to N$ is a morphism in $\sC_A$. It follows that $M/\ker(f)\in \sC_A(\leq\lambda+\bZ I)$ and thus that $O^{\lambda+\bZ I} M\subseteq \ker(f)$. This implies that the map $$T^{\lambda+\bZ I}:\sC_A\to\sC_A(\leq\lambda+\bZ I)$$ is a functor, with the property that for all $M\in\sC_A$ and $N\in \sC_A(\leq\lambda+\bZ I)$ there is an isomorphism $$\Hom_{\sC_A(\leq\lambda+\bZ I)}(T^{\lambda+\bZ I}M,N)\cong \Hom_{\sC_A}(M,N).$$ In other words, $T^{\lambda+\bZ I}$ is left adjoint to the (exact) inclusion functor $\sC_A(\leq\lambda+\bZ I)\to\sC_A$. This proves the following lemma:
\begin{lemma}\label{TruncProj}
	Let $M\in\sC_A$ and $\lambda+\bZ I\in X/\bZ I$. If $M$ is projective in $\sC_A$ then $T^{\lambda+\bZ I} M$ is projective in $\sC_A(\leq\lambda+\bZ I)$.
\end{lemma}

We now want to talk about $Z$-filtrations. We will need the following lemma (cf. Lemma 2.14(a) in \cite{AJS}):

\begin{lemma}
	Let $\lambda\in X$. If $\Ext(Z_{A,\chi}(\lambda),M)\neq 0$ then $\lambda+\bZ I\leq \mu+\bZ I$ for some $\mu+\bZ I\in X/\bZ I$ with $M_{\mu+\bZ I}^{d\mu}\neq 0$.
\end{lemma}
\begin{proof} 
	Suppose $\lambda+\bZ I\nleq \mu+\bZ I$ for all $\mu+\bZ I\in X/\bZ I$ with $M_{\mu+\bZ I}\neq 0$. Let $$0\to M\to N\to Z_{A,\chi}(\lambda)\to 0$$ be an extension of $Z_{A,\chi}(\lambda)$ by $M$. Write $v_\lambda$ for the standard generator of $Z_{A,\chi}(\lambda)$ and pick an element $w_\lambda\in N$ which maps to $v_\lambda$. In particular, we may assume $w_\lambda\in N_{\lambda+\bZ I}^{d\lambda}$. 
	
	For each $\alpha\in R^{+}$, $e_\alpha w_\lambda\in N_{\lambda+\alpha+\bZ I}^{d(\lambda+\alpha)}$. If $M_{\lambda+\alpha+\bZ I}^{d(\lambda+\alpha)}\neq 0$ then $\lambda+\bZ I\nleq \lambda+\alpha+\bZ I$, which clearly cannot happen (although, unlike the $\chi=0$ case, it can happen that $\lambda+\bZ I=\lambda+\alpha+\bZ I$). Hence $M_{\lambda+\alpha+\bZ I}^{d(\lambda+\alpha)}=0$. But, $e_\alpha w_\lambda$ is in the kernel of the map $N\to Z_{A,\chi}(\lambda)$ (as $e_\alpha v_\lambda=0$), which is precisely $M$. In particular, $e_\alpha w_\lambda\in M_{\lambda+\alpha+\bZ I}^{d(\lambda+\alpha)}=0$ and so $e_\alpha w_\lambda=0$. 
	
	There therefore exists $w_\lambda\in N_{\lambda+\bZ I}^{d\lambda}$ with $e_\alpha w=0$ for all $\alpha\in R$ and $w_\lambda\mapsto v_\lambda$. Hence, the extension is split, so $\Ext(Z_{A,\chi}(\lambda),M)= 0$.
\end{proof}

\begin{cor}
	Let $\lambda,\mu\in X$. If $\Ext(Z_{A,\chi}(\lambda),Z_{A,\chi}(\mu))\neq 0$ then $\lambda+\bZ I\leq \mu+\bZ I$.
\end{cor}

This corollary in particular gives us the power to reorder somewhat the terms of a $Z$-filtration, as we now see. Suppose $M\in \sC_A$ has a $Z$-filtration $$0=M_0\subset M_1\subset M_2\subset\cdots\subset\cdots M_n=M$$ with sections $M_i/M_{i-1}\cong Z_{A,\chi}(\lambda_i)$ for $\lambda_i\in X$.

Let us assume that there exist $j<i$ with $\lambda_i+\bZ I\nleq\lambda_j+\bZ I$. In particular, there exists $k$ with $\lambda_{k+1}+\bZ I\nleq\lambda_{k}+\bZ I$. We then have a short exact sequence $$0\to \frac{M_k}{M_{k-1}}\to\frac{M_{k+1}}{M_{k-1}}\to \frac{M_{k+1}}{M_k}\to 0.$$

By the above lemma, we see that $\Ext(Z_{A,\chi}(\lambda_{k+1}),Z_{A,\chi}(\lambda_k))= 0$, and so the extension splits. We may therefore swap the order in which $Z_{A,\chi}(\lambda_k)$ and $Z_{A,\chi}(\lambda_{k+1})$ appear in the filtration.

We obtain the following results, comparable to Lemma 2.14(c) and Lemma 3.7(b) in \cite{AJS}.

\begin{prop}
	Let $M\in \sC_A$ have a $Z$-filtration, with notation as above. Then $M$ has a $Z$ filtration with the property that $j<i$ implies $\lambda_j+\bZ I\leq \lambda_i+\bZ I$.
\end{prop}

\begin{prop}\label{ZFiltOrd}
	Let $M\in \sC_A$ have a $Z$-filtration, with notation as above. Let $\lambda+\bZ I\in X/\bZ I$. Then $M$ has a $Z$ filtration with the property that there exists $k$ such that $\lambda_i+\bZ I\nleq \lambda+\bZ I$ for all $i\leq k$ and $\lambda_i+\bZ I\leq \lambda+\bZ I$ for all $i>k$.
\end{prop}

\begin{prop}
	Let $M\in \sC_A$ have a $Z$-filtration. Then $T^{\lambda+\bZ I}M$ has a $Z$-filtration.
\end{prop}

\begin{proof}
	We may assume, as in Proposition~\ref{ZFiltOrd}, that $M$ has a $Z$-filtration with the property that there exists $k$ such that $\lambda_i+\bZ I\nleq \lambda+\bZ I$ for all $i\leq k$ and $\lambda_i+\bZ I\leq \lambda+\bZ I$ for all $i>k$.
	
	Under this assumption, we clearly have that $M/M_k\in\sC_A(\leq\lambda+\bZ I)$, so $O^{\lambda+\bZ I}M\subseteq M_k$. Conversely, for each $N\in\sC_A(\leq\lambda+\bZ I)$ and each $\mu+\bZ I\in X/\bZ I$ we get $\mu+\bZ I\nleq\lambda+\bZ I$ only if $\Hom_{\sC_A}(Z_{A,\chi}(\mu),N)=0$, since the standard generator of $Z_{A,\chi}(\mu)$ must map to an element of $N_{\mu+\bZ I}^{d\mu}=0$. This particularly means that $\Hom_{\sC_A}(M_k,M/O^{\lambda+\bZ I}M)=0$ and so $M_k\subseteq O^{\lambda+\bZ I}M$. 
	
	Hence $M_k=O^{\lambda+\bZ I}M$. This easily shows that $T^{\lambda+\bZ I}M=M/O^{\lambda+\bZ I}M=M/M_k$ has a $Z$-filtration.
\end{proof}

\subsection{Extension of scalars}\label{Sec5.3}

Maintaining our usual notation for $A$, let $A'$ be an $A$-algebra. Clearly $A'$ is also a $U^0$-algebra via the structure map $\pi:U^0\to A\to A'$. It will be important for us to understand how the categories $\sC_A$ and $\sC_{A'}$ can be compared, which we do in this subsection. Much of this works the same way as in Section 3.1 of \cite{AJS}. Similar to \cite{AJS}, we define the category $\sG\sC_A$ to be defined in the same way as the category $\sC_A$, but where we replace condition (B) with condition (F):
\begin{enumerate}
	\item[(F)] The set of all $\mu+\bZ I\in X/\bZ I $ with $M_{\mu+\bZ I}\neq 0$ is finite.
\end{enumerate}
In other words, we drop the requirement that each $M_{\lambda+\bZ I}$ is finitely-generated over $A$.

We would like to define a restriction functor $\sC_{A'}\to\sC_A$. However, if $A'$ is not finitely-generated over $A$ then restriction would clearly not preserve the finite-generation of the $M_{\lambda+\bZ I}$. Thus, in general, the restriction functor is a map $\sG\sC_{A'}\to\sG\sC_A$. However, if we assume $A'$ is finitely-generated over $A$ then we indeed have a restriction map $\sC_A\to\sC_{A'}$.

Going in the other direction, we need an extension of scalars. Given $M\in \sG\sC_A$ we define $M\otimes_A A'\in\sG\sC_{A'}$ in the obvious way as a  $U_\chi\otimes A'$-module, which we equip with grading given by $(M\otimes_{A}A')_{\lambda+\bZ I}=M_{\lambda+\bZ I}\otimes_{A} A'$ and (D)-decomposition given by $(M\otimes_{A}A')_{\lambda+\bZ I}^{d\mu}=M_{\lambda+\bZ I}^{d\mu}\otimes_{A} A'$. It is straightforward to check that these objects lie in $\sG\sC_{A'}$, and furthermore to check that $M\otimes_{A}A'\in\sC_{A'}$ whenever $M\in\sC_A$. So we may also discuss the extension of scalars functor $\sC_A\to\sC_{A'}$.

If $M,N\in\sG\sC_{A}$ then there exists an isomorphism $$\Hom_{\sG\sC_A}(M,N)\xrightarrow{\sim}\Hom_{\sG\sC_{A'}}(M\otimes_{A}A',N).$$ So extension of scalars is left adjoint to the restriction functor. We may of course define similar maps for $\sC_A'$, $\sC_A''$, $\sC_A^I$ and $\sC_A^{I,+}$, with similar adjointness properties.

Given $M\in \sC_A^I$, we see that there exists a canonical isomorphism $$\Phi_A^I(M)\otimes_{A}A'=(U_\chi\otimes_{U^I} M)\otimes_{A} A'=U_\chi\otimes_{U^I} (M\otimes_{A} A')=\Phi_{A'}^I(M\otimes_{A} A').$$ Similarly, given $M\in\sC_A'$, 
$$Z_{A,\chi}(M)\otimes_{A}A'=(U_\chi\otimes_{U^0U^{+}} M)\otimes_{A} A'=U_\chi\otimes_{U^0U^{+}} (M\otimes_{A} A')=Z_{A',\chi}(M\otimes_{A} A').$$ We may also derive analogous results for $\Phi_A$, $\Phi_A'$, $\Phi_A^{I,+}$ and $\Gamma_{A,\chi}$.

\begin{lemma}\label{projscal}
	Let $M\in\sC_A$. If $M$ is projective in $\sC_A$ then $M\otimes_{A} A'$ is projective in $\sC_{A'}$.
\end{lemma}

\begin{proof}
	A similar argument to that of Remark 2.7 in \cite{AJS} shows that if $M$ projective in $\sC_A$, then $M$ projective in $\sG\sC_A$. Now, we argue as in Lemma 3.1 in \cite{AJS}: $M$ being projective in $\sC_{A}$ implies $M$ being projective in $\sG\sC_A$. The adjointness property then tells us that $M\otimes_{A}A'$ is projective in $\sG\sC_{A'}$. It is then clearly projective in $\sC_{A'}$.
	

\end{proof}

\begin{lemma}
	Let $M\in\sC_A$ have a $Z$-filtration. Then $M\otimes_{A}A'$ has a $Z$-filtration.
\end{lemma}

\begin{proof}
	Since $M$ has a $Z$-filtration it is free over $A$. Therefore, any short exact sequence where all of the terms have a $Z$-filtration splits over $A$. Since $Z_{A,\chi}(\mu)\otimes_{A} A'=Z_{A',\chi}(\mu)$, the result follows.
\end{proof}

\begin{cor}
	Let $M\in\sC_A$ have a $Z$-filtration, and let $A'$ be an $A$-algebra. Then there is an equality $T^{\lambda+\bZ I}(M\otimes_A A')=T^{\lambda+\bZ I}(M)\otimes_A A'$.
\end{cor}
%
%
%
%
%
%
%
%

\subsection{Duality}\label{Sec5.4}

Let $A$ be a commutative Noetherian $U^0$-algebra with structure map $\pi:U^0\to A$. If $M$ is a module in $\sC_A$, the $A$-module $\Hom_A(M,A)$ may be equipped with a $U_{-\chi}\otimes A$-module structure as follows. Given $x\in \fg$, $f\in \Hom_A(M,A)$ and $m\in M$, we define $$(x\cdot f)(m)=f(-x\cdot m).$$ It is straightforward to check that $x\cdot f\in\Hom_{A}(M,A)$ and that this defines a $U_{-\chi}$-module structure on $\Hom_{A}(M,A)$ which clearly commutes with the $A$-module structure. 

Furthermore, we may equip $\Hom_{A}(M,A)$ with an $X/\bZ I$-grading such that $$\Hom_{A}(M,A)_{\lambda+\bZ I}=\{f\in\Hom_{A}(M,A)\,\vert\,f(M_{\sigma+\bZ I})=0\,\mbox{ for all }\,\sigma+\bZ I\neq -\lambda+\bZ I\}$$ for each $\lambda+\bZ I\in X/\bZ I$. It is straightforward to see that $$\Hom_A(M,A)=\bigoplus_{\lambda+\bZ I\in X/\bZ I}\Hom_{A}(M,A)_{\lambda+\bZ I}.$$ Similarly, given $\lambda+\bZ I\in X/\bZ I$ and $\mu\in \lambda+\bZ I+pX$, we may define $$\Hom_A(M,A)_{\lambda+\bZ I}^{d\mu}=\{f\in \Hom_A(M,A)_{\lambda+\bZ I}\,\vert\,f(M_{\sigma+\bZ I}^{d\omega})=0\,\,\mbox{for all}\,\,\sigma+\bZ I\neq -\lambda+\bZ I\,\,\mbox{and}\,\, d\omega\neq -d\mu\},$$ so that $$\Hom_A(M,A)_{\lambda+\bZ I}=\bigoplus_{\substack{d\mu\in\fh^{*}\\ \mu\in \lambda+\bZ I+pX}}\Hom_{A}(M,A)_{\lambda+\bZ I}^{d\mu}.$$ 

In order to describe in which category the object $\Hom_{A}(M,A)$ lies, we must introduce some notation. Since we use the notation $\sC_A$ to denote a category of $U_\chi\otimes A$-modules, we shall denote by $\overline{\sC_A}$ the analogous category of $U_{-\chi}\otimes A$-modules. Note that, although $-\chi$ is not in standard Levi form as it has been described in this paper, since it still has the quality that it is only non-zero on a set of negative simple roots, all the results in this paper work equally well for $\overline{\sC_A}$ as they do for $\sC_A$. We also denote by $\overline{A}$ the $U^0$-algebra which is equal to $A$ as a $\bK$-algebra, but whose structure map $\overline{\pi}:U^0\to\overline{A}$ is defined by extending the assignment  $\overline{\pi}(h)=-\pi(h)$ for all $h\in\fh$.

\begin{prop}
	Let $M\in \sC_A$. Then $\Hom_{A}(M,A)$ lies in $\overline{\sC_{\overline{A}}}$.
\end{prop}

\begin{proof}
	We need to check conditions (A), (B), (C) and (D) in $\overline{\sC_{\overline{A}}}$. Condition (A) is easy to see, and condition (B) is straightforward once one observes that $\Hom_{A}(M,A)=\Hom_{A}(M_{\lambda+\bZ I},A)$ as an $A$-module. One may check condition (C) by direct computation, so all that remains is to check condition (D). It is obvious that the (D)-decomposition summands should be the ones described above, so we check that $U^0$ acts appropriately on them. Let $h\in\fh$, $\lambda+\bZ I\in X/\bZ I$, $\mu\in\lambda+\bZ I+pX$, $f\in\Hom_{A}(M,A)_{\lambda+\bZ I}^{d\mu}$ and $m\in M$. Then $$m=\sum_{\sigma+\bZ I\in X/\bZ I}\sum_{\substack{d\omega\in\fh^{*}\\ \omega\in\sigma+\bZ I +pX}} m_{\sigma+\bZ I}^{d\omega}$$ and so we have 	\begin{equation*}
		\begin{split}
			(h\cdot f)(m) & =\sum_{\sigma+\bZ I\in X/\bZ I}\sum_{\substack{d\omega\in\fh^{*}\\ \omega\in\sigma+\bZ I +pX}} f(-h\cdot m_{\sigma+\bZ I}^{d\omega}) \\ & =\sum_{\sigma+\bZ I\in X/\bZ I}\sum_{\substack{d\omega\in\fh^{*}\\ \omega\in\sigma+\bZ I +pX}} f( m_{\sigma+\bZ I}^{d\omega})\cdot(-\pi(\widetilde{\omega}(h))) \\ & =f(m_{-\lambda+\bZ I}^{-d\mu})\cdot(-\pi(\widetilde{-\mu}(h))) \\ & = f(m)\overline{\pi}(\widetilde{\mu}(h))
		\end{split}
	\end{equation*}
	as required, noting that $-\pi(\widetilde{-\mu}(h))=-\pi(h-d\mu(h))=-\pi(h)+d\mu(h)=\overline{\pi}(\widetilde{\mu}(h))$. This is then enough to conclude the result.
\end{proof}

This in particular gives us a contravariant functor $\sC_A\to\overline{\sC_{\overline{A}}}$. Furthermore, if $M$ is free over $A$ we have $$\Hom_A(\Hom_A(M,A),A)\cong M$$ in $\sC_A$ (noting easily that $\overline{\overline{\sC_{\overline{\overline{A}}}}}=\sC_A$).

\begin{cor}\label{HomDual}
	Let $A=F$ be a field, and let $M,N\in\sC_F$. Then $$\Hom_{\sC_F}(M,N)\cong \Hom_{\overline{\sC_{\overline{F}}}}(\Hom_{F}(N,F),\Hom_{F}(M,F)).$$
\end{cor}

Another construction will be necessary in order to understand duality in the category $\sC_A$. As in \cite{Jan5}, the Lie algebra $\fg$ and the algebraic group $G$ may be equipped with compatible automorphisms, both of which we shall call $\tau$. The map $\tau$ preserves $\fh$ and $T$, induces $\chi\circ\tau^{-1}=-\chi$ and induces $-w_I$ on $X$, where $w_I$ is the longest element in $W_I$, the Weyl group corresponding to $R_I$ (this will be explained a little more in Subsection 2). This therefore induces an isomorphism $\tau:U_{\chi}\to U_{-\chi}$, and so an isomorphism $\tau^{-1}:U_{-\chi}\to U_\chi$, which sends $U^0$ to $U^0$ (and $\fh$ to $\fh$).

We denote by $\,^{\tau}A$ the commutative Noetherian $U^0$-algebra which is equal to $A$ as a $\bK$-algebra but which has structure map $\,^{\tau}\pi:U^0\to\,^{\tau}A$ extended from $\,^{\tau}\pi(h)=\pi(\tau^{-1}(h))$ for all $h\in\fh$. We then define, for each $M\in\sC_A$, the $U_{-\chi}\otimes A$-module $\,^{\tau}M$ which has the same $A$-module structure as $M$ but with the $U_{-\chi}$ action given by $x\cdot m=\tau^{-1}(x)m$ for $x\in U_{-\chi}$ and $m\in M$.

\begin{prop}
	If $M\in \sC_A$ then $\,^{\tau}M\in\overline{\sC_{\,^{\tau}A}}$.
\end{prop}

\begin{proof}
	We give $\,^{\tau}M$ an $X/\bZ I$-grading by setting $$(\,^{\tau}M)_{\lambda+\bZ I}=M_{-\lambda+\bZ I}$$ for $\lambda+\bZ I\in X/\bZ I$ and we define the summands of the (D)-decomposition by $$(\,^{\tau}M)_{\lambda+\bZ I}^{d\mu}=M_{-\lambda+\bZ I}^{d(\tau^{-1}(\mu))}$$ for $\mu\in\lambda+\bZ I+pX$. Let us ensure that the right-hand-side makes sense, i.e. that $\tau^{-1}(\mu)\in -\lambda+\bZ I+pX$ for $\mu\in \lambda+\bZ I+pX$. Suppose $\mu=\lambda+\kappa+p\delta$. Then \begin{equation*}
		\begin{split}\tau^{-1}(\mu) & =\tau^{-1}(\lambda)+\tau^{-1}(\kappa) +\tau^{-1}(p\delta) \\ & = -w_I\lambda-w_I\kappa+p\tau^{-1}(\delta) \\ & = -\lambda+(\lambda-w_I\lambda) - \kappa - (w_I\kappa-\kappa) -p\tau^{-1}(\delta) \in -\lambda+\bZ I+pX
		\end{split}
	\end{equation*}
	since $w_I\lambda-\lambda\in\bZ I$. The same argument shows that $\tau(\mu)\in -\lambda+\bZ I+pX$ for $\mu\in \lambda+\bZ I+pX$ (since $\tau$ and $\tau^{-1}$ induce the same maps on $X$). Hence, $\mu\mapsto \tau^{-1}(\mu)$ is a permutation of $-\lambda+\bZ I+pX$, and we have $$(\,^{\tau}M)_{\lambda+\bZ I}=\bigoplus_{\substack{d\mu\in\fh^{*}\\ \mu\in\lambda+\bZ I+pX}} (\,^{\tau}M)_{\lambda+\bZ I}^{d\mu}$$ as required.
	
	It is clear that $\,^\tau M$ satisfies conditions (A) and (B).  For condition (C), we must show that $(U_{\chi})_{\sigma+\bZ I}(\,^\tau M)_{\lambda+\bZ I}\subseteq (\,^\tau M)_{\sigma+\lambda+\bZ I}$ for all $\sigma+\bZ I,\lambda+\bZ I\in X/\bZ I$. This will be clear from the description of the action and the grading once we observe that $\tau^{-1}((U_\chi)_{\sigma+\bZ I})\subseteq (U_\chi)_{-\sigma+\bZ I}$. This holds since if $x\in \fh$ then $\tau^{-1}(x)\in\fh$ and if $x\in \fg_{\alpha}$ then $\tau^{-1}(x)\in \fg_{-w_I\alpha}$ (by \cite{Jan5}). Noting that $\alpha-w_I\alpha\in\bZ I$, we see that $-w_I\alpha+\bZ I=-\alpha+\bZ I$, and so $\tau^{-1}(x)\in \fg_{-\alpha+\bZ I}$. 
	
	All that remains is therefore to check that $U^0$ acts the right way on each $(\,^{\tau}M)_{\lambda+\bZ I}^{d\mu}$. Let $\lambda+\bZ I\in X/\bZ I$, $\mu\in \lambda+\bZ I+pX$, $m\in(\,^{\tau}M)_{\lambda+\bZ I}^{d\mu}=M_{-\lambda+\bZ I}^{d(\tau^{-1}(\mu))}$, and $h\in\fh$. Then $$h\cdot m=\tau^{-1}(h)m=m\pi(\widetilde{\tau^{-1}(\mu)}(\tau^{-1}(h)).$$ What is $\pi(\widetilde{\tau^{-1}(\mu)}(\tau^{-1}(h))$? Recall that $\widetilde{\mu}(h)=\mu+d\mu(h)$ for any $h\in \fh$. So $$\pi(\widetilde{\tau^{-1}(\mu)}(\tau^{-1}(h))= \pi\circ\tau^{-1}(h) + d(\tau^{-1}(\mu))(\tau^{-1}(h)).$$ But $d(\tau^{-1}(\mu))(\tau^{-1}(h))=d\mu(\tau\circ\tau^{-1}(h))=d\mu(h)$. Furthermore, recall that $\pi\circ\tau^{-1}(h)=\,^{\tau}\pi(h)$. Hence, we conclude that $$\pi(\widetilde{\tau^{-1}(\mu)}(\tau^{-1}(h))=\,^{\tau}\pi(h)+d\mu(h)=\,^{\tau}\pi(\widetilde{\mu}(h)).$$ The result follows.
	
\end{proof}
Therefore, $M\mapsto \,^{\tau}M$ is a covariant functor $\sC_A\to\overline{\sC_{\,^{\tau}A}}$. Let write $\bD A$ for the $U^0$-algebra with the same $\bK$-algebra structure as $A$ but with structure map $\bD\pi=\,^\tau(\overline{\pi}):U^0\to A$ extended from $h\mapsto -\pi(\tau^{-1}(h))$ for $h\in\fh$. Combining the two functors just discussed, we hence get a functor $$\bD:\sC_A\to\sC_{\bD A}.$$ Similarly, we write $\overline{\bD} A$ for the $U^0$-algebra with the same $\bK$-algebra structure as $A$ but with structure map $\overline{\bD}\pi:U^0\to A$ extended from $h\mapsto -\pi(\tau(h))$ for $h\in\fh$. Then a similar argument to the above shows that there is a functor $$\overline{\bD}:\sC_A\to\sC_{\overline{\bD}A},\qquad M\mapsto \Hom_{A}(\,^{\tau^{-1}}M,A),$$ where $\,^{\tau^{-1}}M$ denotes the $U_{-\chi}\otimes A$-module with $x\in U_{-\chi}$ acting on  $\,^{\tau^{-1}}M$ as $\tau(x)\in U_\chi$ acts on $M$. It is clear that $\overline{\bD}\bD A=A=\bD\overline{\bD} A$. When $A=F$ is a field, we also have that $\bD\overline{\bD}(M)\cong M\cong\overline{\bD}\bD(M)$ in $\sC_F$. (This is in fact an application of the result that $\bD\overline{\bD}(M)\cong M\cong\overline{\bD}\bD(M)$ in $\sC_A$ for any $M\in\sC_A$ which is free over $A$). Thus we get the following result.

\begin{prop}\label{Antiequiv}
	If $A=F$ is a field, the functor $\bD:\sC_A\to\sC_{\bD A}$ is an anti-equivalence of categories.
\end{prop}

\section{Properties of baby Verma modules}\label{Sec6}

We have already seen how baby Verma modules $Z_{A,\chi}(\lambda)$ can be constructed in the category $\sC_A$. As these are such important modules, we would like to understand a little bit more about their structure in $\sC_A$. We do so in this section. 

\subsection{Irreducibility}\label{Sec6.1}

Of particular importance in this area of study is the case when $A=F$ is a field. If $F=\bK$, the usual results of \cite{Jan4,Jan} apply, but even when $F$ is not necessarily algebraically closed we can derive some of the same results. One result which is of central importance in the study of $\sC_\bK$ is the fact that each baby Verma module has a unique irreducible quotient. We see the same here.

\begin{prop}\label{UniqIrredQuot}
	Let $\lambda\in X$. Then the baby Verma module $Z_{F,\chi}(\lambda)$ has a unique irreducible quotient.
\end{prop}

\begin{proof}
	One way to prove this, which we will follow, is to show that there is a proper $U_\chi\otimes F$-submodule of $Z_{F,\chi}(\lambda)$ in which every submodule of $Z_{F,\chi}(\lambda)$ in $\sC_F$ lies. This will further follow if there is a proper $U^{-}\otimes F$-submodule of $Z_{F,\chi}(\lambda)$ in which every $U^{-}\otimes F$-submodule (and hence every $U_\chi\otimes F$-submodule) lies.
	
	As a $U^{-}\otimes F$-module, it is clear that $Z_{F,\chi}(\lambda)$ is isomorphic to $U^{-}\otimes F$. If $U^{-}\otimes F$ has a unique simple quotient as a $U^{-}\otimes F$-module, then the result will hold. In particular, this will hold if there is a unique irreducible $U^{-}\otimes F$-module such that $U^{-}\otimes F$ is the projective cover of that module. Observing that $U^{-}=U_\chi(\fn^{-})$ and $U_\chi(\fn^{-})\otimes F=U_\chi(\fn^{-}\otimes F)$, this then follows from the same argument in Section 3 and Theorem 10.2 in \cite{Jan}, except applied for fields which are not necessarily algebraically closed.
\end{proof}

Let us write $\Rad Z_{F,\chi}(\lambda)$ for the radical of $Z_{F,\chi}(\lambda)$ in $\sC_F$. Then we define the simple object $$L_{F,\chi}(\lambda)=\frac{Z_{F,\chi}(\lambda)}{\Rad Z_{F,\chi}(\lambda)}\in\sC_F.$$ Proposition~\ref{ZSurj} then shows that all the simple modules in $\sC_F$ are of the form $L_{F,\chi}(\lambda)$ for some $\lambda\in X$.

We know that the category $\sC_F$ has enough projectives and that all modules have finite length (as they are finite-dimensional as $F$-vector spaces). This in particular means that there is a one-to-one correspondence between simple modules and projective indecomposable modules in $\sC_F$. Thus, for $\lambda\in X$, we denote by $Q_{F,\chi}(\lambda)\in\sC_F$ the unique (up to isomorphism) projective indecomposable module in $\sC_F$ with $$\frac{Q_{F,\chi}(\lambda)}{\Rad(Q_{F,\chi}(\lambda))}\cong L_{F,\chi}(\lambda).$$

\subsection{Isomorphisms of baby Verma modules}\label{Sec6.2}

We would now like to briefly discuss when baby Verma modules are isomorphic. In order to do this, we first must recall some definitions relating to the Weyl group. Firstly, given $\alpha\in R$, we define $s_\alpha:X\to X$ as the reflection sending $\lambda\in X$ to $\lambda-\langle\lambda,\alpha^\vee\rangle\alpha$, where $\alpha^\vee$ is the coroot associated to $\alpha$. Furthermore, given $m\in\bZ$, we define $s_{\alpha,m}:X\to X$ to be the map $\lambda\mapsto\lambda-\langle\lambda,\alpha^\vee\rangle\alpha+m\alpha$. Then the {\bf Weyl group} $W$ is defined to be the subgroup of $\Aut_{\bZ}(X)$ generated by the elements $s_\alpha$ for $\alpha\in R$, and the {\bf affine Weyl group} $W_p$ is defined to be the subgroup of $\Aut_{\bZ}(X)$ generated by the elements $s_{\alpha,mp}$ for $\alpha\in R$ and $m\in\bZ$. Note also that, if $\alpha\in R$ and $m\in\bZ$ and we define $t_{m,\alpha}:X\to X$ by $\lambda\mapsto \lambda+m\alpha$, then $s_{\alpha,m}=t_{\alpha,m}\circ s_\alpha$, and so $W_p$ is also the subgroup generated by $W$ and $t_{\alpha,mp}$ for all $\alpha\in R$ and $m\in\bZ$.

Furthermore, we define the following parabolic subgroups: $W_I$ is defined to be the subgroup of $W$ generated by the elements $s_\alpha$ for $\alpha\in R_I$ (so, as above, it is the Weyl group corresponding to $R_I$), and $W_{I,p}$ is defined to be the subgroup of $W_p$ generated by the elements $s_{\alpha,mp}$ for $\alpha\in R_I$ and $m\in\bZ$. By definition, the groups $W$, $W_p$, $W_I$ and $W_{I,p}$ act on $X$. We also need another action, the dot-action, which is defined by $w\cdot \lambda=w(\lambda+\rho)-\rho$. Here $\rho\in X$ is the half sum of positive roots. Note further that these groups similarly act on $\fh^{*}$, both through the usual action and through the dot-action.

When $A=\bK$, we know (from Proposition 11.9 in \cite{Jan}) that $Z_{\bK,\chi}(\lambda)\cong Z_{\bK,\chi}(\mu)$ if and only if $\lambda\in W_{I,p}\cdot\mu$. How does this work for other $A$?

\begin{prop}\label{ImpliesIsom}
	Suppose that $A$ is a commutative, Noetherian $U^0$-algebra, with structure map $\pi:U^0\to A$. Suppose that $\pi(h_\alpha)=0$ for all $\alpha\in R\cap\bZ I$. Then $Z_{A,\chi}(\lambda)\cong Z_{A,\chi}(\mu)$ if $\lambda\in W_{I,p}\cdot\mu$.
\end{prop}

\begin{proof}
	
	Let $\lambda \in X$ and let us fix $\alpha\in I$. It is enough to show that $Z_{A,\chi}(\lambda)\cong Z_{A,\chi}(\lambda+p\alpha)$ and $Z_{A,\chi}(\lambda)\cong Z_{A,\chi}(s_\alpha\cdot\lambda)$.
	
	For the former, write $v_0\in Z_{A,\chi}(\lambda)$ and $w_0\in Z_{A,\chi}(\lambda+p\alpha)$ for the standard generators, so that $v_0\in Z_{A,\chi}(\lambda)_{\lambda+\bZ I}^{d\lambda}$ and $w_0\in Z_{A,\chi}(\lambda+p\alpha)_{\lambda+p\alpha+\bZ I}^{d(\lambda+p\alpha)}=Z_{A,\chi}(\lambda+p\alpha)_{\lambda+\bZ I}^{d\lambda}$. Then, since $e_\beta v_0=0=e_\beta w_0$ for all $\beta\in R^{+}$, there are inverse homomorphisms $Z_{A,\chi}(\lambda)\leftrightarrow Z_{A,\chi}(\lambda+p\alpha)$ with $v_0\leftrightarrow w_0$. Thus, $Z_{A,\chi}(\lambda)\cong Z_{A,\chi}(\lambda+p\alpha)$.
	
	For the latter, suppose $\langle\lambda,\alpha^{\vee}\rangle =mp+a$ for $m\in\bZ$ and $0\leq a<p$, so that we have $s_\alpha\cdot\lambda=\lambda-(a+1)\alpha+mp\alpha$. We therefore want to define an isomorphism $f:Z_{A,\chi}(\lambda-(a+1)\alpha)\to Z_{A,\chi}(\lambda)$ (since, combined with the former observation, this gives an isomorphism $Z_{A,\chi}(s_\alpha\cdot\lambda)\to Z_{A,\chi}(\lambda)$). Let us write $w_0$ for the standard generator of $Z_{A,\chi}(\lambda-(a+1)\alpha)$ and $v_0$ for the standard generator of $Z_{A,\chi}(\lambda)$. Let us define $f:Z_{A,\chi}(\lambda-(a+1)\alpha)\to Z_{A,\chi}(\lambda)$ to be the map sending $w_0$ to $e_{-\alpha}^{a+1}v_0$. We need to ensure this is well defined. 
	
	Firstly, we know that $w_0\in Z_{A,\chi}(\lambda-(a+1)\alpha)_{\lambda-(a+1)\alpha+\bZ I}^{d(\lambda-(a+1)\alpha)}$. Furthermore, from conditions (C) and (D) of the definition of $\sC_A$ we have that $e_{-\alpha}^{a+1}v_0\in Z_{A,\chi}(\lambda)_{\lambda-(a+1)\alpha+\bZ I}^{d(\lambda-(a+1)\alpha)}$. 
	
	Next, suppose $\beta\neq\alpha$ for $\beta\in R^{+}$. Then, since $\alpha$ is simple, it easy to see that $e_\beta e_{-\alpha}^{a+1}v_0=0$.
	
	Finally, we want to look at $e_\alpha e_{-\alpha}^{a+1}v_0$. We may calculate that $$e_\alpha e_{-\alpha}^{a+1}v_0=e_{-\alpha}^av_0(\pi(\widetilde{\lambda}(h_\alpha))+\pi(\widetilde{(\lambda-\alpha)}(h_\alpha))+\cdots + \pi(\widetilde{(\lambda-a\alpha)}(h_\alpha)).$$ Note that $\pi(\widetilde{(\lambda-i\alpha)}(h_\alpha))=\pi(h_\alpha)+d(\lambda-i\alpha)(h_\alpha)$. Thus, 
	\begin{equation*}
		\begin{split}
			e_{\alpha}e_{-\alpha}^{a+1}v_0 & =e_{-\alpha}^av_0(\pi((a+1)h_\alpha)+d\lambda(h_\alpha)+d(\lambda-\alpha)(h_\alpha)+\cdots+d(\lambda-a\alpha)(h_\alpha)) \\
			& =e_{-\alpha}^av_0(\pi((a+1)h_\alpha))+e_{-\alpha}^av_0(\langle\lambda,\alpha^\vee\rangle+\langle\lambda-\alpha,\alpha^\vee\rangle+\cdots+\langle\lambda-a\alpha,\alpha^\vee\rangle)\\
			&
			=e_{-\alpha}^av_0(\pi((a+1)h_\alpha))+e_{-\alpha}^av_0((a+1)a-(1+2+\cdots +a)\langle\alpha,\alpha^\vee\rangle)\\
			&
			=e_{-\alpha}^av_0(\pi((a+1)h_\alpha))+e_{-\alpha}^av_0\left((a+1)a-2\left(\frac{a(a+1)}{2}\right)\right)\\
			&
			=(a+1)e_{-\alpha}^av_0(\pi(h_\alpha)).
		\end{split}
	\end{equation*}
	Here, we use the fact that $d\lambda(h_\alpha)=\langle \lambda,\alpha^\vee\rangle$. Since we assume that $\pi(h_\alpha)=0$, we get that $e_{\alpha}e_{-\alpha}^{a+1}v_0=0$.
	
	Hence, there is a well-defined homomorphism $f:Z_{A,\chi}(\lambda-(a+1)\alpha)\to Z_{A,\chi}(\lambda)$ in $\sC_A$. We need to show it is an isomorphism. Note that $e_{-\alpha}^{p-a-1}w_0$ maps to $v_0$, so $v_0$ lies in the image of $f$. In particular, $f$ is surjective.
	
	Furthermore, $Z_{A,\chi}(\lambda-(a+1)\alpha)$ is a free $A$-module with the following $A$-basis: As usual, we write $\{\beta_1,\ldots,\beta_r\}$ for the positive roots in $R$, although we now assume $\beta_1=\alpha$. The $A$-basis consists of the elements $$e_{-\beta_r}^{k_r}\cdots e_{-\beta_2}^{k_2}e_{-\beta_1}^{k_1}w_0$$ for $0\leq k_i<p$. Similarly, $Z_{A,\chi}(\lambda)$ is a free $A$-module with basis consisting of elements $$e_{-\beta_r}^{k_r}\cdots e_{-\beta_2}^{k_2}e_{-\beta_1}^{k_1}v_0$$ for $0\leq k_i<p$.
	
	From the above, we see that $$f(e_{-\beta_r}^{k_r}\cdots e_{-\beta_2}^{k_2}e_{-\beta_1}^{k_1}w_0)=e_{-\beta_r}^{k_r}\cdots e_{-\beta_2}^{k_2}e_{-\beta_1}^{l}v_0$$ where $k_1+a+1=mp+l$ for $m\in\bZ$ and $0\leq l<p$. In particular, $f$ sends an $A$-basis to an $A$-basis, and so is an isomorphism.
	
	We therefore conclude that $$Z_{A,\chi}(\lambda-(a+1)\alpha)\cong Z_{A,\chi}(\lambda),$$ and hence that, if $\mu\in W_{I,p}\cdot\lambda$, we have $$Z_{A,\chi}(\lambda)\cong Z_{A,\chi}(\mu).$$
	
\end{proof}

\begin{rmk}
	We shall soon develop a more general argument that will easily imply the result just proved. Nonetheless, we include the more explicit proof of this result to highlight where the requirement that $\pi(h_\alpha)=0$ for all $\alpha\in R\cap\bZ I$ has relevance.
\end{rmk}

\section{Regular nilpotent $p$-characters}\label{Sec7}
Let $A$ be a commutative, Noetherian $U^0$-algebra with structure map $\pi:U^0\to A$. As usual, we write $\sC_A$ for the category obtained from this $A$ and $\pi$. The algebraically-closed field $\bK$ is a commutative Noetherian $U^0$-algebra with structure map $\pi^\circ:U^0\to \bK$ sending each $h\in \fh$ to zero. This extends to a map $\pi^\circ:U^0\to\bK\hookrightarrow A$, since $A$ is a $\bK$-algebra. We then write $\sC_A^\circ$ for the category obtained from this $A$ and $\pi^\circ$.

{\bf We shall make the assumption throughout this section that $\pi(h_\alpha)=0$ for all $\alpha\in R$.}

\subsection{An equivalence of categories}\label{Sec7.1}

Notwithstanding the title of this section, we do not yet assume that $\chi$ is regular nilpotent (i.e. we do not assume $I=\Pi$), although our standing assumption in this section will be most meaningful when that is so. The key power of the assumption is that it gives us an equivalence of categories as in the following proposition.

\begin{prop}\label{ThetaEquiv}
	If $\pi(h_\alpha)=0$ for all $\alpha\in R$, there is an equivalence of categories $\Theta_A$ between $\sC_A^\circ$ and $\sC_A$.
\end{prop}

\begin{proof}
	Let us first define $\Theta_A:\sC_A^\circ\to\sC_A$. For each $M\in \sC_A$, we define $\Theta_A(M)$ to be equal to $M$ as an $A$-module. We equip $\Theta_A(M)$ with an $X/\bZ I$-grading by setting $$\Theta_A(M)_{\lambda+\bZ I}=M_{\lambda+\bZ I}$$ and we define the (D)-decomposition summands by $$\Theta_A(M)_{\lambda+\bZ I}^{d\mu}=M_{\lambda+\bZ I}^{d\mu}$$ for $\lambda+\bZ I\in X/\bZ I$ and $\mu\in\lambda+\bZ I+pX$. What remains is to define the $U_\chi$-module structure on $\Theta_A(M)$. We do this in the following way: for $m\in \Theta_A(M)$ and $\alpha\in R$ we define $$e_\alpha \cdot m=e_\alpha m,$$ i.e. each $e_\alpha$ acts on $\Theta_A(M)$ as it does on $M$. For $s\in U^0$ and $m\in \Theta_A(M)_{\lambda+\bZ I}^{d\mu}$, we define $$s\cdot m=m\pi(\widetilde{\mu}(s)),$$ which we extend to an action on all of $M$ through the (D)-decomposition and the grading.
	
	We need to check that this defines a $U_\chi$-module structure on $\Theta_A(M)$. For this, it suffices to check that it gives a $\fg$-module structure such that each $e_\alpha^p$ acts as scalar multiplication by  $\chi(e_\alpha)^p$. This latter point is clear, and checking that $\Theta_A(M)$ is a $\fg$-module is mostly a straightforward case of checking that  commutators of Chevalley basis vectors act in the right way. The only complication in the calculation is in ensuring that $$e_\alpha\cdot(e_{-\alpha}\cdot m)-e_{-\alpha}\cdot(e_\alpha\cdot m)=h_\alpha\cdot m$$ for all $\alpha\in R$ and $m\in\Theta_A(M)$. It is enough to check this for $m\in M_{\lambda+\bZ I}^{d\mu}$. It is straightforward to see that the left-hand side of this equation is then $$m\pi^\circ(\widetilde{\mu}(h_\alpha))$$ while the right-hand side is $$m\pi(\widetilde{\mu}(h_\alpha)).$$ We note now that $$\widetilde{\mu}(h_\alpha)=h_\alpha+d\mu(h_\alpha)$$ and so $$\pi(\widetilde{\mu}(h_\alpha))=\pi(h_\alpha)+d\mu(h_\alpha)$$ and $$\pi^\circ(\widetilde{\mu}(h_\alpha))=\pi^\circ(h_\alpha)+d\mu(h_\alpha).$$ Since, by assumption, $\pi(h_\alpha)=0=\pi^\circ(h_\alpha)$, we indeed have equality.
	
	Hence, $\Theta_A(M)$ is a $U_\chi\otimes A$-module with an $X/\bZ I$-grading and a (D)-decomposition. That $\Theta_A(M)$ satisfies conditions (A), (B), (C), and (D) is easy to check. Therefore, we indeed have $\Theta_A(M)\in \sC_A$.
	
	It is also straightforward to see that any homomorphism $f:M\to N$ in $\sC_A^\circ$ induces a homomorphism $\Theta_A(f):\Theta_A(M)\to\Theta_A(N)$ such that $\Theta_A(f)(m)=f(m)$ for all $m\in M$. Hence, $$\Theta_A:\sC_A^\circ\to\sC_A$$ is a well-defined functor. 
	
	Similarly, we define $\Theta_A^{-1}:\sC_A\to\sC_A^\circ$ by setting $\Theta_A^{-1}(M)$ to have the same $A$-module structure, $X/\bZ I$-grading, (D)-decomposition, and $e_\alpha$-action ($\alpha\in R$) as $M$ does, and by defining the $U^0$ action such that $$sm=m\pi^\circ(\widetilde{\mu}(s))$$ for $s\in U^0$ and $m\in \Theta_A^{-1}(M)_{\lambda+\bZ I}^{d\mu}=M_{\lambda+\bZ I}^{d\mu}$. That this indeed gives an object of $\sC_A^\circ$ can be checked in much the same way that we saw that $\Theta_A(M)$ was in $\sC_A$. It is also a functor in the same way.
	
	All that remains is to see that $\Theta_A\circ\Theta_A^{-1}=\Id=\Theta_A^{-1}\circ\Theta_A$. This is obvious, with condition (D) ensuring that the $U^0$-actions coincide.
\end{proof}

\begin{prop}\label{ThetaComm}
	Let $A'$ be an $A$-algebra, and suppose $\pi(h_\alpha)=0$ for all $\alpha\in R$. For each $M\in\sC_A^\circ$, we have $$\Theta_A(M)\otimes_A A' = \Theta_{A'}(M\otimes_A A').$$
\end{prop}

\begin{proof}
	It is easy to see that $\Theta_A(M)\otimes_A A'$ and $\Theta_{A'}(M\otimes_A A')$ are equal as $A'$-modules (as both are equal to $M\otimes_A A'$), that $$(\Theta_A(M)\otimes_A A')_{\lambda+\bZ I}=\Theta_{A'}(M\otimes_A A')_{\lambda+\bZ I}$$ for all $\lambda+\bZ I\in X/\bZ I$, and that $$(\Theta_A(M)\otimes_A A')_{\lambda+\bZ I}^{d\mu}=\Theta_{A'}(M\otimes_A A')_{\lambda+\bZ I}^{d\mu}$$ for all $\mu\in \lambda+\bZ I +pX$. It is also easy to see that each $e_\alpha$ for $\alpha\in R$ acts the same on $\Theta_A(M)\otimes_A A'$ and $\Theta_{A'}(M\otimes_A A')$. Finally, let $m\otimes a$ lie in the $\bK$-vector space $(M\otimes_{A'} A)_{\lambda+\bZ I}^{d\mu}$ and let $s\in U^0$. Let us write $\pi:U^0\to A$ for the structure map of $A$ and $\widehat{\pi}:U^0\to A'$ for the induced structure map of $A'$. If we view $m\otimes a$ as an element of $(\Theta_A(M)\otimes_A A')_{\lambda+\bZ I}^{d\mu}=\Theta_A(M)_{\lambda+\bZ I}^{d\mu}\otimes_{A} A'$ then we have 
	$$s(m\otimes a)=(sm)\otimes a=(m\pi(\widetilde{\mu}(s)))\otimes a=m\otimes(a\widehat{\pi}(\widetilde{\mu}(s))).$$ On the other hand, if we view $m\otimes a\in \Theta_{A'}(M\otimes_A A')_{\lambda+\bZ I}^{d\mu}$ then we have $$s(m\otimes a)=(m\otimes a)\widehat{\pi}(\widetilde{\mu}(s))=m\otimes(a\widehat{\pi}(\widetilde{\mu}(s)))$$
	as required.
\end{proof}

Given $\lambda\in X$, we may easily see that $Z_{\bK,\chi}(\lambda)\otimes_{\bK} A\in\sC_A^\circ$. Since $\pi(h_\alpha)=0=\pi^\circ(h_\alpha)$ for all $\alpha\in R$, one can show that $$\Theta_A(Z_{\bK,\chi}(\lambda)\otimes_{\bK}A)=Z_{A,\chi}(\lambda).$$ 

Hence, we conclude that, for $\lambda,\mu\in X$,  $$\Hom_{\sC_A}(Z_{A,\chi}(\lambda),Z_{A,\chi}(\mu))\cong \Hom_{\sC_A^\circ}(Z_{\bK,\chi}(\lambda)\otimes_{\bK} A,Z_{\bK,\chi}(\mu)\otimes_{\bK} A)$$ as $A$-modules, and, more generally, $$\Ext^i_{\sC_A}(Z_{A,\chi}(\lambda),Z_{A,\chi}(\mu))\cong \Ext^i_{\sC_A^\circ}(Z_{\bK,\chi}(\lambda)\otimes_{\bK} A,Z_{\bK,\chi}(\mu)\otimes_{\bK} A)$$ for all $i\geq 0$.

The proof of Lemma 3.2 in \cite{AJS} works exactly the same way here, and so, since $A$ is flat over $\bK$, we get that $$\Ext^i_{\sC_A^\circ}(Z_{\bK,\chi}(\lambda)\otimes_{\bK} A,Z_{\bK,\chi}(\mu)\otimes_{\bK} A)\cong \Ext^i_{\sC_\bK}(Z_{\bK,\chi}(\lambda),Z_{\bK,\chi}(\mu))\otimes_{\bK} A$$ for all $i\geq 0$. So we have $$\Ext^i_{\sC_A}(Z_{A,\chi}(\lambda),Z_{A,\chi}(\mu))\cong \Ext^i_{\sC_\bK}(Z_{\bK,\chi}(\lambda),Z_{\bK,\chi}(\mu))\otimes_{\bK} A$$ for all $i\geq 0$ and all $\lambda,\mu\in X$.

If $A$ is a local algebra, with residue field $F$, then we similarly obtain $$\Ext^i_{\sC_F}(Z_{F,\chi}(\lambda),Z_{F,\chi}(\mu))\cong \Ext^i_{\sC_F^\circ}(Z_{\bK,\chi}(\lambda),Z_{\bK,\chi}(\mu))\otimes_{\bK} F$$ for all $i\geq 0$ and all $\lambda,\mu\in X$. Hence there is an $F$-isomorphism $$\Ext^i_{\sC_F}(Z_{F,\chi}(\lambda),Z_{F,\chi}(\mu))\cong \Ext^i_{\sC_A}(Z_{A,\chi}(\lambda),Z_{A,\chi}(\mu))\otimes_{A} F$$ for all $i\geq 0$.

This means, in particular, that if $\Ext^i_{\sC_F}(Z_{F,\chi}(\lambda),Z_{F,\chi}(\mu))=0$ then $$\Ext^i_{\sC_A}(Z_{A,\chi}(\lambda),Z_{A,\chi}(\mu))=\Ext^i_{\sC_A}(Z_{A,\chi}(\lambda),Z_{A,\chi}(\mu))\fm$$ where $\fm$ is the unique maximal ideal of $A$. If $\Ext^i_{\sC_A}(Z_{A,\chi}(\lambda),Z_{A,\chi}(\mu))$ is finitely-generated over $A$ then by Nakayama's lemma this implies that $\Ext^i_{\sC_A}(Z_{A,\chi}(\lambda),Z_{A,\chi}(\mu))=0$. Since each $\Ext$ is a quotient of a submodule of a finitely-generated $A$-module, and $A$ is a Noetherian ring, we indeed have $A$-finite-generation of $\Ext^i_{\sC_A}(Z_{A,\chi}(\lambda),Z_{A,\chi}(\mu))$.

In conclusion, we have the following theorem.

\begin{theorem}
	Let $A$ be a commutative, Noetherian, local $U^0$-algebra with structure map $\pi:U^0\to A$, and let $F$ be the residue field of $A$. Suppose that $\pi(h_\alpha)=0$ for all $\alpha\in R$. Then $\lambda,\mu\in X$ lie in the same block over $A$ if and only if they lie in the same block over $F$.
\end{theorem}

Another result we can get from this equivalence of categories is the following.

\begin{prop}\label{ZRegIsom}
	Suppose that $A$ be a commutative Noetherian $U^0$-algebra, with structure map $\pi:U^0\to A$ such that $\pi(h_\alpha)=0$ for all $\alpha\in R$. Then $Z_{A,\chi}(\lambda)\cong Z_{A,\chi}(\mu)$ implies $\lambda\in W_{p}\cdot\mu$.
\end{prop}

\begin{proof}
	Since $Z_{A,\chi}(\lambda)\cong Z_{A,\chi}(\mu)$ in $\sC_A$, we may use the equivalence of categories $\Theta_A$ to conclude that $$Z_{\bK,\chi}(\lambda)\otimes_{\bK} A\cong Z_{\bK,\chi}(\mu)\otimes_{\bK} A$$ in $\sC_A^\circ$. Now, observing as in \cite[Theorem 9.3]{Jan} (referencing \cite{KW}) that $U(\fg)^G\cong U(\fh)^{W_\bullet}$, we write $\cen_{d\lambda}:U(\fg)^G\to \bK$ for the algebra homomorphism sending $u$ to $u(d\lambda)$, where we view $u\in U(\fg)^G\cong U(\fh)^{W_\bullet}$ as a polynomial function on $\fh^{*}$. Then we have $uz=\cen_{d\lambda}(u)z$ for all $u\in U(\fg)^G$ and $z\in Z_{\bK,\chi}(\lambda)$, and so $uz=\cen_{d\lambda}(u)z$ for all $u\in U(\fg)^G$ and $z\in Z_{\bK,\chi}(\lambda)\otimes_{\bK} A$. Since $Z_{\bK,\chi}(\lambda)\otimes_{\bK} A\cong Z_{\bK,\chi}(\mu)\otimes_{\bK} A$, we therefore conclude that $\cen_{d\lambda}(u)=\cen_{d\mu}(u)$ for all $u\in U(\fg)^G$. Thus, $\cen_{d\lambda}=\cen_{d\mu}$ and so $d\lambda\in W\cdot d\mu$ by Corollary 9.4 in \cite{Jan} (see also \cite{KW}).
	
	We therefore conclude that $\lambda\in W\cdot\mu +pX$, and we know that $\lambda+\bZ I=\mu+\bZ I$ (as both are maximal in the partial order among those elements of $X/\bZ I$ with the property that, say, $Z_{A,\chi}(\lambda)_{\sigma+\bZ I}\neq 0$). As in \cite[Prop 11.9]{Jan}, we pick $w\in W$ such that $\lambda-w\cdot\mu\in pX$. Then $\mu-w\cdot\mu\in \bZ I$ since $w\in W$. Hence, $\lambda-w\cdot\mu = \lambda- \mu +\mu- w\cdot\mu\in\bZ I$. Therefore, $\lambda-w\cdot\mu \in \bZ I\cap pX$, which is equal to $p\bZ I$ as in Section 11.2 in \cite{Jan} (which relies upon Jantzen's standard assumptions). Hence, $\lambda\in W_{p}\cdot\mu$, as required.
\end{proof}
\subsection{Projective covers}\label{Sec7.2}

In this subsection, we do now assume that $\chi$ is regular nilpotent, which we recall means that $I=\Pi$.

Suppose for the moment that $A=F$ is a field, and let $\lambda\in X$. Then we have already observed that $\Theta_F(Z_{\bK,\chi}(\lambda)\otimes_{\bK} F)=Z_{F,\chi}(\lambda)$, so we can derive the following. 

\begin{prop}\label{Irred}
	Let $\lambda\in X$, and suppose $\chi$ is regular nilpotent (i.e. $I=\Pi$). Let $F$ be a field which is a $U^0$-algebra via $\pi:U^0\to F$, with the property that $\pi(h_\alpha)=0$ for all $\alpha\in R$. Then $Z_{\bK,\chi}(\lambda)\otimes_{\bK} F$ is irreducible in $\sC_F^\circ$ and  $Z_{F,\chi}(\lambda)$ is irreducible in $\sC_F$. 
\end{prop}

\begin{proof}
	As just mentioned, $Z_{F,\chi}(\lambda)=\Theta_F(Z_{\bK,\chi}(\lambda)\otimes_{\bK}F).$ Since $\Theta_F$ is an equivalence of categories, it is enough to show that $Z_{\bK,\chi}(\lambda)\otimes_{\bK}F$ is irreducible in $\sC_F^\circ$. Furthermore, as any proper submodule of $Z_{\bK,\chi}(\lambda)\otimes_{\bK}F$ in $\sC_F^\circ$ is also a proper submodule of $Z_{\bK,\chi}(\lambda)\otimes_{\bK}F$ as a $U_\chi(\fg)\otimes F$-module, it is enough to prove that $Z_{\bK,\chi}(\lambda)\otimes_{\bK} F$ is irreducible as a $U_\chi(\fg)\otimes F$-module. (Note here that in $\sC_F^\circ$ all modules are $U_\chi(\fg)\otimes F$-modules, not just $U_\chi\otimes F$-modules).
	
	Since $\chi$ is regular nilpotent, we know that $Z_{\bK,\chi}(\lambda)$ is an irreducible $U_\chi(\fg)$-module \cite[Theorem 4.2]{FP}. So the result will follow if $Z_{\bK,\chi}(\lambda)$ is an absolutely irreducible $U_\chi(\fg)$-module (recalling that a $U_\chi(\fg)$-module $M$ is called {\bf absolutely irreducible} if $M\otimes F$ is an irreducible $U_\chi(\fg)\otimes F$-module for all field extensions $\bK\subseteq F$). Because $Z_{\bK,\chi}(\lambda)$ is finite-dimensional, it follows from Proposition 9.2.5 in \cite{W} that $Z_{\bK,\chi}(\lambda)$ is absolutely irreducible if $\End_{U_\chi(\fg)}(Z_{\bK,\chi}(\lambda))=\bK$. This follows since $\bK$ is algebraically closed.
\end{proof}

As in Section~\ref{Sec6.1}, let us write $Q_{\bK,\chi}(\lambda)$ for the projective cover of $Z_{\bK,\chi}(\lambda)$ in $\sC_\bK$, write $Q_{F,\chi}(\lambda)$ for the projective cover of $Z_{F,\chi}(\lambda)$ in $\sC_F$, and write $Q_{F^\circ,\chi}(\lambda)$ for the projective cover of $Z_{\bK,\chi}(\lambda)\otimes_{\bK} F$ in $\sC_F^\circ$.

\begin{lemma}\label{ProjCovCirc}
	Let $\lambda\in X$, and suppose $\pi(h_\alpha)=0$ for all $\alpha\in R$. Then $Q_{F^\circ,\chi}(\lambda)\cong Q_{\bK,\chi}(\lambda)\otimes_{\bK} F$ in $\sC_F^\circ$.
\end{lemma}

\begin{proof}
	
	We observed in Section~\ref{Sec5.1} that $$\sC_F^\circ=\bigoplus_{\lambda+\bZ I\in X/\bZ I} \sC_F^\circ(\lambda+\bZ I)$$ where $\sC_F^\circ(\lambda+\bZ I)$ may be identified as the category of $F$-finite-dimensional $U_\chi(\fg)\otimes F$-modules $M$ which have an $F$-vector space decomposition  $$M=\bigoplus_{\substack{d\mu\in\fh^{*}\\ \mu\in\lambda+\bZ I+pX}}M^{d\mu}$$ such that $$hm=md\mu(h)$$ for all $m\in M^{d\mu}$ and $h\in\fh$. Note here that we may consider $U_\chi(\fg)\otimes F$-modules rather than $U_\chi\otimes F$-modules because the $U^0$-algebra structure on $F$ is extended from $\bK$. 
	
	Let us write $\sX$ for the category of $F$-finite-dimensional $U_\chi(\fg)\otimes F$-modules $M$ which have a decomposition $$M=\bigoplus_{\sigma\in\Lambda_0} M^\sigma$$ where $$\Lambda_0=\{\sigma:\fh\to\bK\,\vert\,\sigma(h)^p=\sigma(h^{[p]})\,\mbox{for all}\, h\in\fh\}$$ and where $$hm=m\sigma(h)$$ for all $m\in M^\sigma$ and $h\in \fh$. Morphisms in this category are $U_\chi(\fg)\otimes F$-module homomorphisms which preserve the decomposition. It is then straightforward to see that $\sX$ is the direct sum of $\sC_F^\circ(\kappa+\bZ I)$ over all $\kappa+\bZ I\in X/\bZ I$. In other words, $\sX$ is an equivalent category to $\sC_F^\circ$.
	
	Now, it is clear that $Q_{F^\circ,\chi}(\lambda)$ is the projective cover of $Z_{F^\circ,\chi}(\lambda)\in \sX$ and that $Q_{\bK,\chi}(\lambda)\otimes_{\bK} F$ is projective in $\sX$, since all of these objects lie in $\sC_F^\circ(\lambda+\bZ I)$.
	
	Let us now write $\sM_F$ for the category of all $F$-finite-dimensional $U_\chi(\fg)\otimes F$-modules, and $\sM_\bK$ for the category of all $\bK$-finite-dimensional $U_\chi(\fg)$-modules. One can check without any difficulty that $\sX$ is a full subcategory of $\sM_F$. 
	
	Let us write $\widetilde{Q}_{F^\circ,\chi}(\lambda)$ for the projective cover of $Z_{F^\circ,\chi}(\lambda)$ in $\sM_F$. Since $Q_{\bK,\chi}(\lambda)$ is a $U_\chi(\fg)$-module, $Q_{\bK,\chi}(\lambda)\otimes_{\bK} F\in \sM_F$. Furthermore, one can show (using an easier version of Lemma 3.1(a) in \cite{AJS}) that $Q_{\bK,\chi}(\lambda)\otimes_{\bK} F$ is projective in $\sM_F$.
	
	Since this is so, and since $\widetilde{Q}_{F^\circ,\chi}(\lambda)$ is the projective cover of $Z_{F^\circ,\chi}(\lambda)$ in $\sM_F$, there exists a surjection in $\sM_F$ from $Q_{\bK,\chi}(\lambda)\otimes_{\bK} F$ onto $\widetilde{Q}_{F^\circ,\chi}(\lambda)$. Now, Proposition 11.18 in \cite{Jan} show that $$\dim_F(Q_{\bK,\chi}(\lambda)\otimes_{\bK} F)=p^{\dim\fn^{+}}\left| W\cdot d\lambda\right|.$$ Furthermore, the proof of Lemma 10.9 in \cite{Jan} works just as well for $U_\chi(\fg)\otimes F=U_\chi(\fg\otimes F)$-modules (for example, that $U_\chi(\fg)\otimes F$ is symmetric follows from the fact that $U_\chi(\fg)$ is). In other words, we have $$\dim_F(\widetilde{Q}_{F^\circ,\chi}(\lambda))=p^{\dim\fn^{+}}\sum_{\mu\in\Lambda_0}[Z_{F^\circ,\chi}(d\lambda):Z_{F^\circ,\chi}(d\mu)].$$ Clearly $[Z_{F^\circ,\chi}(d\lambda):Z_{F^\circ,\chi}(d\mu)]$ equals 1 if $Z_{F^\circ,\chi}(d\lambda)\cong Z_{F^\circ,\chi}(d\mu)$ and 0 otherwise. Observing that $$\Hom_{\sM_\bK}(Z_{\bK,\chi}(d\lambda), Z_{\bK,\chi}(d\mu))\otimes_{\bK} F\cong \Hom_{\sM_F}(Z_{F^\circ,\chi}(d\lambda), Z_{F^\circ,\chi}(d\mu))$$ we get that $Z_{F^\circ,\chi}(d\lambda)\cong Z_{F^\circ,\chi}(d\mu)$ if and only if $Z_{\bK,\chi}(d\lambda)\cong Z_{\bK,\chi}(d\mu)$. Hence $[Z_{F^\circ,\chi}(d\lambda):Z_{F^\circ,\chi}(d\mu)]=1$ if and only if $d\lambda\in W\cdot d\mu$, and is zero if and only if not. Hence, 
	$$\dim_F(\widetilde{Q}_{F^\circ,\chi}(\lambda))=p^{\dim\fn^{+}}\left|W\cdot d\lambda\right|=\dim_F(Q_{\bK,\chi}(\lambda)\otimes_{\bK} F).$$ Therefore, we have $$\widetilde{Q}_{F^\circ,\chi}(\lambda)\cong Q_{\bK,\chi}(\lambda)\otimes_{\bK} F.$$
	
	Hence, $Q_{\bK,\chi}(\lambda)\otimes_{\bK} F$ is the projective cover of $Z_{F^\circ,\chi}(\lambda)$ in $\sM_F$. Since $Q_{\bK,\chi}(\lambda)\otimes_{\bK} F$ in fact lies in $\sX$, it is also the projective cover of $Z_{F^\circ,\chi}(\lambda)$ in $\sX$. Therefore, $Q_{\bK,\chi}(\lambda)\otimes_{\bK} F\cong Q_{F^\circ,\chi}(\lambda)$ in $\sX$, and so $$Q_{\bK,\chi}(\lambda)\otimes_{\bK} F\cong Q_{F^\circ,\chi}(\lambda)$$ in $\sC_F^\circ$.
	
\end{proof}

\begin{cor}\label{projAF}
	Let $\lambda\in X$ and suppose $\pi(h_\alpha)=0$ for all $\alpha\in R$. Suppose also that $A$ is local with residue field $F$. Then there exists a projective module $Q_{A,\chi}(\lambda)\in\sC_A$ with $Q_{A,\chi}(\lambda)\otimes_A F\cong Q_{F,\chi}(\lambda)$.
\end{cor}

\begin{proof}
	Define $Q_{A,\chi}(\lambda)=\Theta_A(Q_{\bK,\chi}(\lambda)\otimes_{\bK}A)$. This is projective since $Q_{\bK,\chi}(\lambda)\otimes_{\bK}A$ is projective in $\sC_A^\circ$ by Lemma~\ref{projscal} and $\Theta_A$ is an equivalence of categories. We then have \begin{equation*} \begin{split} Q_{A,\chi}(\lambda)\otimes_A F & =\Theta_A(Q_{\bK,\chi}(\lambda)\otimes_{\bK}A)\otimes_A F  =\Theta_F(Q_{\bK,\chi}(\lambda)\otimes_{\bK}A\otimes_A F)
			\\ &		
			\cong\Theta_F(Q_{\bK,\chi}(\lambda)\otimes_{\bK} F)\cong\Theta_F(Q_{F^\circ,\chi}(\lambda))=Q_{F,\chi}(\lambda).
		\end{split}
	\end{equation*}
	Here, the second equality comes from Proposition~\ref{ThetaComm}, the fourth equality comes from Proposition~\ref{ProjCovCirc}, and the fifth equality comes from the fact that an equivalence of categories must map projective covers to projective covers.
\end{proof}

It is important to observe that when $A$ is not local we may still define $$Q_{A,\chi}(\lambda)\coloneqq\Theta_A(Q_{\bK,\chi}(\lambda)\otimes_{\bK}A),$$ which will still be a projective module in $\sC_A$.

Note that $Q_{A,\chi}(\lambda)$ lies in $\sC_A(\lambda+\bZ I)$ since $Q_{\bK,\chi}(\lambda)\in\sC_\bK(\lambda+\bZ I)$. Furthermore, from this construction it is clear that, as $Q_{\bK,\chi}(\lambda)$ has a $Z$-filtration in which the only module appearing is $Z_{\bK,\chi}(\lambda)$, which appears $\left| W\cdot d\lambda\right|$ times \cite[Prop. 11.18]{Jan}, the module $Q_{A,\chi}(\lambda)$ also has a $Z$-filtration in which the only module appearing is $Z_{A,\chi}(\lambda)$, similarly appearing exactly $\left| W\cdot d\lambda\right|$ times. In particular, each $Q_{A,\chi}(\lambda)$ is free over $A$ with rank $p^{\dim\fn^{+}}\left| W\cdot d\lambda\right|$.

Furthermore, if $Z_{A,\chi}(\mu)$ is a factor in any $Z$-filtration of $Q_{A,\chi}(\lambda)$, then there must exist a non-zero homomorphism $Z_{A,\chi}(\mu)\to Z_{A,\chi}(\lambda)$, using the just discussed $Z$-filtration. Since we have  $\Hom_{\sC_A}(Z_{A,\chi}(\mu), Z_{A,\chi}(\lambda))\cong \Hom_{\sC_\bK}(Z_{\bK,\chi}(\mu), Z_{\bK,\chi}(\lambda))\otimes_{\bK}A$, we are able to conclude that $\Hom_{\sC_\bK}(Z_{\bK,\chi}(\mu), Z_{\bK,\chi}(\lambda))\neq 0$ and so $Z_{\bK,\chi}(\mu)\cong Z_{\bK,\chi}(\lambda)$, which implies $Z_{A,\chi}(\mu)\cong Z_{A,\chi}(\lambda)$.

\begin{prop}\label{QIisom}
	Let $\lambda,\mu\in X$ and suppose $\pi(h_\alpha)=0$ for all $\alpha\in R$. Then $Q_{A,\chi}(\lambda)\cong Q_{A,\chi}(\mu)$ if and only if $\lambda\in W_p\cdot\mu$.
\end{prop}
\begin{proof}
	For the forward implication, we have that $$\Hom_{\sC_A}(Q_{A,\chi}(\lambda),Q_{A,\chi}(\mu))\cong \Hom_{\sC_\bK}(Q_{\bK,\chi}(\lambda),Q_{\bK,\chi}(\mu))\otimes_{\bK} A$$ as in the previous subsection. Hence, $\Hom_{\sC_\bK}(Q_{\bK,\chi}(\lambda),Q_{\bK,\chi}(\mu))\neq 0$. Since $Q_{\bK,\chi}(\lambda)$ has a filtration all of whose sections are isomorphic to $Z_{\bK,\chi}(\lambda)$, and $Q_{\bK,\chi}(\mu)$ has similarly, we conclude that $$\Hom_{\sC_\bK}(Z_{\bK,\chi}(\lambda),Z_{\bK,\chi}(\mu))\neq 0,$$ implying (as these baby Verma modules are irreducible) that $Z_{\bK,\chi}(\lambda)\cong Z_{\bK,\chi}(\mu)$ and so $\lambda\in W_p\cdot\mu$.
	
	Conversely, $\lambda\in W_p\cdot\mu$ implies $Z_{\bK,\chi}(\lambda)\cong Z_{\bK,\chi}(\mu)$, and so $Q_{\bK,\chi}(\lambda)\cong Q_{\bK,\chi}(\mu)$. This then easily implies that $Q_{A,\chi}(\lambda)\cong Q_{A,\chi}(\mu)$.
\end{proof}

Let us therefore write $\Lambda$ for the set of $W_p$-dot-orbits on $X$, so that elements of $\Lambda$ enumerate the pairwise non-isomorphic $Q_{A,\chi}(\lambda)$ and also, by pairing Proposition~\ref{ImpliesIsom} with Proposition~\ref{ZRegIsom}, the pairwise non-isomorphic baby Verma modules $Z_{A,\chi}(\lambda)$ (recalling that throughout this subsection $\chi$ is regular nilpotent). We shall also use the notation $\Lambda$ for a fundamental domain of the $W_p$-dot-action on $X$, so that we may view elements of $\Lambda$ as being in $X$.

We may now observe that these projective $Q_{A,\chi}(\lambda)$ are ubiquitous in $\sC_A$.

\begin{lemma}\label{RegProjCov}
	Let $M\in\sC_A$ and suppose $\pi(h_\alpha)=0$ for all $\alpha\in R$. Then there exists $Q\in\sC_A$ with $Q\twoheadrightarrow M$ and such that $$Q=\bigoplus_{\mu\in \Lambda} Q_{A,\chi}(\mu)^{m_\mu}$$ for some $m_\mu\geq 0$.
\end{lemma}

\begin{proof}
	
	Let $P\in\sC_\bK''$. This means that $P$ is an $X/p\bZ I$-graded finite-dimensional $\bK$-module. Then $P\otimes_{\bK} A\in (\sC_A^\circ)''$ is free as an $A$-module (where $(\sC_A^\circ)''$ is the category analogous to $\sC_A''$ for $A$ with the structure map $U^0\to\bK\hookrightarrow A$). Conversely, suppose $Q\in (\sC_A^\circ)''$, so that $Q$ is an $X/p\bZ I$-graded finitely-generated $A$-module, and suppose that $Q$ is free (of finite rank) over $A$. Let $x_1,\ldots,x_n$ be a free basis of $A$ (where each $x_i$ is homogeneous). Then $\bigoplus_{i=1}^n x_i\bK$ is an $X/p\bZ I$-graded finite-dimensional $\bK$-module, with $ (\bigoplus_{i=1}^n x_i\bK)\otimes_{\bK} A\cong Q$. In other words, every $Q\in (\sC_A^\circ)''$ which is free over $A$ is isomorphic to an object of the form $N\otimes_{\bK} A\in(\sC_A^\circ)''$ for some $N\in\sC_\bK''$.
	
	Now, let $M\in\sC_A$. When we view $M$ as an element in $\sC_A''$, there exists a projective module $Q\in\sC_A''$ with $Q\twoheadrightarrow M$ (since $\sC_A''$ has enough projectives). We may assume $Q$ is free over $A$ (since $Q\in\sC_A''$ projective implies that each $Q_{\lambda+\bZ I}^{d\mu}$ is projective over $A$ so is a direct summand of a free module). Hence, there exists $P\in\sC_\bK''$ with $P\otimes_{\bK}A\cong Q$. It is then straightforward to check that $$\Phi_A(Q)\cong\Theta_A(\Phi_\bK(P)\otimes_{\bK}A).$$
	
	Since $P$ is projective in $\sC_\bK''$, $\Phi_\bK(P)$ is projective in $\sC_\bK$. Hence, $$\Phi_\bK(P)=\bigoplus_{\mu\in \Lambda} Q_{\bK,\chi}(\mu)^{m_\mu}$$ for some $m_\mu\geq 0$. Therefore $$\Phi_\bK(P)\otimes_{\bK}A=\bigoplus_{\mu\in \Lambda} (Q_{\bF,\chi}(\mu)\otimes_{\bK} A)^{m_\mu}$$ and so $$\Phi_A(Q)\cong \bigoplus_{\mu\in \Lambda} Q_{A,\chi}(\mu)^{m_\mu}.$$ Since $Q\twoheadrightarrow M$ in $\sC_A''$, we have $\Phi_A(Q)\twoheadrightarrow M$ in $\sC_A$. We therefore get the result.
\end{proof}

In the following proposition we consider the extent to which such a decomposition of $Q$ is unique.

\begin{prop}\label{UniqDecomp}
	Suppose that $A$ is local with residue field $F$, and that $\pi(h_\alpha)=0$ for all $\alpha\in R$. Let $Q\in\sC_A$ with $$Q\cong Q_{A,\chi}(\lambda_1)\oplus\cdots \oplus Q_{A,\chi}(\lambda_r)\cong Q_{A,\chi}(\mu_1)\oplus\cdots\oplus Q_{A,\chi}(\mu_s).$$ Then $r=s$ and there exists $\sigma\in S_r$ such that $$Q_{A,\chi}(\lambda_i)\cong Q_{A,\chi}(\mu_{\sigma(i)})$$ in $\sC_A$ for all $1\leq i\leq r$.
\end{prop}

\begin{proof}
	Since $Q_{A,\chi}(\lambda)\otimes_{A} F\cong Q_{F,\chi}(\lambda)$ by Corollary~\ref{projAF}, we have that  
	$$Q_{F,\chi}(\lambda_1)\oplus\cdots\oplus Q_{F,\chi}(\lambda_r)\cong Q_{F,\chi}(\mu_1)\oplus\cdots\oplus Q_{F,\chi}(\mu_s).$$ Since in $\sC_F$ the $Q_{F,\chi}(\lambda)$ are the projective covers of the irreducible modules $Z_{F,\chi}(\lambda)$ and since all objects in $\sC_F$ have finite length, we may apply the Krull-Schmidt theorem. Hence, $r=s$ and there exists $\sigma\in S_r$ such that $$Q_{F,\chi}(\lambda_i)\cong Q_{F,\chi}(\mu_{\sigma(i)})$$ in $\sC_F$ for all $1\leq i\leq r$.
	
	Now, $Q_{F,\chi}(\lambda_i)$ is (isomorphic to) a direct summand of $Q_{F,\chi}(\mu_{\sigma(i)})$, and we know $Q_{A,\chi}(\lambda_i)\otimes_{A} F\cong Q_{F,\chi}(\lambda_i)$ and $Q_{A,\chi}(\mu_{\sigma(i)})\otimes_{A} F\cong Q_{F,\chi}(\mu_{\sigma(i)})$. Hence, by a similar argument to that of Remark 4.18 in \cite{AJS}, we get that $Q_{A,\chi}(\lambda_i)$ is (isomorphic to) a direct summand of $Q_{A,\chi}(\mu_{\sigma(i)})$. 
	
	Thus, there exists $M\in\sC_A$ with $$Q_{A,\chi}(\mu_{\sigma(i)})\cong Q_{A,\chi}(\lambda_i)\oplus M,$$ and so $$Q_{F,\chi}(\mu_{\sigma(i)})\cong Q_{F,\chi}(\lambda_i)\oplus (M\otimes_{A} F),$$ implying that $M\otimes_{A} F=0$. But $M\in\sC_A$ is projective, as it is a direct summand of a projective module, and so $M$ is projective over $A$. (This follows by a similar argument to the proof of Lemma 2.7(c) in \cite{AJS}, using our $\Phi_A$ in place of the functor denoted $\Phi_A$ in \cite{AJS}). Since $M$ is finitely-generated projective over $A$ (as $A$ is Noetherian) and $A$ is local, it is free over $A$, and so $M\otimes_{A} F=0$ implies $M=0$. Therefore, 
	$$Q_{A,\chi}(\mu_{\sigma(i)})\cong Q_{A,\chi}(\lambda_i),$$
	as required.

	%
	%
\end{proof}

\begin{lemma}\label{QFiltDecomp}
	Suppose $A$ is local with residue field $F$, and suppose that $\pi(h_\alpha)=0$ for all $\alpha\in R$. Let $\lambda_1,\ldots,\lambda_r\in \Lambda$ and $P,N\in\sC_A$ with $$Q_{A,\chi}(\lambda_1)\oplus\cdots\oplus Q_{A,\chi}(\lambda_r)=P\oplus N$$ in $\sC_A$. Then $N$ and $P$ each decompose into direct sums of $Q_{A,\chi}(\lambda_i)$'s.
\end{lemma}

\begin{proof}
	Since we know, in $\sC_A$, that $$Q_{A,\chi}(\lambda_1)\oplus\cdots\oplus Q_{A,\chi}(\lambda_r)\cong P\oplus N$$ we see that, in $\sC_F$, we have $$Q_{F,\chi}(\lambda_1)\oplus\cdots\oplus Q_{F,\chi}(\lambda_r)\cong(P\otimes_{A} F)\oplus (N\otimes_{A} F).$$ Clearly $P\otimes_{A} F$ and $N\otimes_{A} F$ are projective in $\sC_F$, and since the $Q_{F,\chi}(\lambda_i)$ are the projective covers of the irreducible modules in $\sC_F$, each projective module in $\sC_F$ is a direct sum of some $Q_{F,\chi}(\lambda)$'s. Hence, there exist $\mu_1,\ldots,\mu_s\in \Lambda$ and $\tau_1,\ldots,\tau_t\in \Lambda$ such that $$P\otimes_{A} F\cong Q_{F,\chi}(\mu_1)\oplus\cdots\oplus Q_{F,\chi}(\mu_s)$$ and $$N\otimes_{A} F\cong Q_{F,\chi}(\tau_1)\oplus\cdots\oplus Q_{F,\chi}(\tau_t).$$ Therefore, we have $$Q_{F,\chi}(\lambda_1)\oplus\cdots\oplus Q_{F,\chi}(\lambda_r)\cong Q_{F,\chi}(\mu_1)\oplus\cdots\oplus Q_{F,\chi}(\mu_s)\oplus Q_{F,\chi}(\tau_1)\oplus\cdots\oplus Q_{F,\chi}(\tau_t).$$ Now, by Proposition~\ref{UniqDecomp} we may assume (permuting the terms if necessary) that $s+t=r$ and that $\lambda_1=\mu_1, \lambda_2=\mu_2,\ldots,\lambda_s=\mu_s,\lambda_{s+1}=\tau_1,\ldots,\lambda_{r-1}=\tau_{t-1}$ and $\lambda_r=\tau_t$.
	
	Let us define the map $f_1:Q_{A,\chi}(\lambda_1)\oplus\cdots\oplus Q_{A,\chi}(\lambda_s)\to P$ as the following composition
	$$Q_{A,\chi}(\lambda_1)\oplus\cdots\oplus Q_{A,\chi}(\lambda_s)\hookrightarrow Q_{A,\chi}(\lambda_1)\oplus\cdots\oplus Q_{A,\chi}(\lambda_r)\xrightarrow{\sim}P\oplus N\twoheadrightarrow P.$$ The above discussion then shows that $f_1\otimes 1:Q_{F,\chi}(\lambda_1)\oplus\cdots\oplus Q_{F,\chi}(\lambda_s)\to P\otimes_{A} F$ is an isomorphism in $\sC_F$. By Nakayama's lemma, we hence have that $f_1$ is surjective (as an $A$-module homomorphism, and so in $\sC_A$). We define $f_2:Q_{A,\chi}(\lambda_{s+1})\oplus\cdots\oplus Q_{A,\chi}(\lambda_r)\to N$ analogously, and it is surjective by the same argument. We then define the map $$f:Q_{A,\chi}(\lambda_1)\oplus\cdots\oplus Q_{A,\chi}(\lambda_r)\to P\oplus N,\qquad f=(f_1,f_2).$$ It is surjective since $f_1$ and $f_2$ are. Furthermore, both $Q_{A,\chi}(\lambda_1)\oplus\cdots\oplus Q_{A,\chi}(\lambda_r)$ and $P\oplus N$ are free $A$-modules with the same finite rank (since each $Q_{A,\chi}(\lambda_i)$ is and since the two modules are isomorphic by assumption). Any surjection between free $A$-modules of the same finite rank is an isomorphism, so $f$ is an isomorphism. Hence, $f_1$ and $f_2$ are isomorphisms, and we get the result.
\end{proof}

\section{Arbitrary standard Levi form}\label{Sec8}

In this section, let $A$ be a commutative, Noetherian $U^0$-algebra with structure map $\pi:U^0\to A$. As usual, we write $\sC_A$ for the category obtained from this $A$ and $\pi$. We assume $\chi$ is in standard Levi form, with associated subset $I$ of simple roots. 

{\bf We shall make the assumption throughout this section that $\pi(h_\alpha)=0$ for all $\alpha\in R_I=R\cap\bZ I$.}

\subsection{The module $Q_{A,\chi}^I(\lambda)$}\label{Sec8.1}

As an initial matter, under this assumption, we can characterise the irreducible modules in $\sC_A^I$ and $\sC_A^{I,+}$ when $A=F$ is a field.
\begin{prop}\label{AllIrreds}
	Let $A=F$, a field with the property that $\pi(h_\alpha)=0$ for all $\alpha\in R_I$. Let $\lambda\in X$. Then $Z_{F,I,\chi}(\lambda)$ is irreducible in both $\sC_F^I$ and $\sC_F^{I,+}$ (being viewed as an element in the latter by trivial extension). Furthermore, every irreducible object in $\sC_F^I$ and in $\sC_F^{I,+}$ is of the form $Z_{F,I,\chi}(\mu)$ for some $\mu\in X$.
\end{prop}

\begin{proof}
	We already know that each $Z_{F,I,\chi}(\lambda)$ is irreducible in $\sC_F^I$ by Proposition~\ref{Irred}. Since any submodule of $Z_{F,I,\chi}(\lambda)$ in $\sC_F^{I,+}$ would also be a submodule in $\sC_F^I$, it follows that it is also irreducible in $\sC_F^{I,+}$. That each irreducible object in $\sC_F^I$ is of this form follows from Proposition~\ref{ZSurj}. What remains is hence to show that the irreducible objects of $\sC_F^{I,+}$ are all of the form $Z_{F,I,\chi}(\mu)$ for $\mu\in X$.
	
	Let $M$ be irreducible in $\sC_F^{I,+}$. We may view $M$ as an element of $\sC_F^{I}$, where it necessarily has a composition series of $Z_{F,\chi,I}(\mu)$'s. As discussed in Section~\ref{Sec5.1}, $M$ has a decomposition $M=\bigoplus_{\lambda+\bZ I} M_{\lambda+\bZ I}$ in $\sC_F^I$. Let $\lambda+\bZ I\in X/\bZ I$ such that $M_{\lambda+\bZ I}\neq 0$ and such that $M_{\mu+\bZ I}\neq 0$ implies $\mu+\bZ I\ngeq\lambda+\bZ I$. This in particular means that $e_\alpha m=0$ for all $\alpha\in R^{+}\setminus R_I$ and $m\in M_{\lambda+\bZ I}$. This further means that $M_{\lambda+\bZ I}\in \sC_F^{I,+}$, and so since $M$ is irreducible in $\sC_F^{I,+}$ we have $M=M_{\lambda+\bZ I}$.
	
	In $\sC_F^{I}$ there exists $\nu\in X$ such that $Z_{F,I,\chi}(\nu)$ embeds in $M=M_{\lambda+\bZ I}$, since all simple objects in $\sC_F^I$ are of this form. Let $f:Z_{F,I,\chi}(\nu)\to M$ be this inclusion. For $\alpha\in R^{+}\setminus R_I$, we observe that $f(e_\alpha v)\in f(Z_{F,I,\chi}(\nu)_{\nu+\alpha+\bZ I})=0$ and $e_\alpha f(v)\in M_{\nu+\alpha+\bZ I}=0$ for all $v\in Z_{F,I,\chi}(\nu)$ (observing that $\nu+\bZ I=\lambda+\bZ I$ by necessity, and that $Z_{F,\chi,I}(\nu)\in\sC_{F}^I(\lambda+\bZ I)$). Since $f$ is a morphism in $\sC_F^{I}$ with this property, it is in fact a morphism in $\sC_F^{I,+}$. Therefore, since $M$ and $Z_{F,\chi,I}(\nu)$ are simple in $\sC_F^{I,+}$, we conclude that $M\cong Z_{F,\chi,I}(\nu)$. This concludes the proof.

\end{proof}

Returning to the case of arbitrary $A$, we note in particular that $\sC_A^I$ satisfies the assumptions of the previous section. There is thus an equivalence of categories $$\Theta_A^I:(\sC_A^I)^\circ\xrightarrow{\sim}\sC_A^I$$ (where we write $(\sC_A^I)^\circ$ for what we called $\sC_A^\circ$ in Subsection~\ref{Sec7.1}). For $\lambda\in X$, we may therefore define the module $Q_{A,I,\chi}(\lambda)\in\sC_A^I$ as in Corollary~\ref{projAF}, namely, $$Q_{A,I,\chi}(\lambda)\coloneqq\Theta_A^I(Q_{\bK,I,\chi}(\lambda)\otimes_{\bK} A)\in\sC_A^I$$ where $Q_{\bK,I,\chi}(\lambda)\in \sC_\bK^I$ is the projective cover of $Z_{\bK,I,\chi}(\lambda)$ in $\sC_\bK^I$.

If $A$ is a local ring with residue field $F$, we once again have the property that $Q_{A,I,\chi}(\lambda)\otimes_A F\cong Q_{F,I,\chi}(\lambda)\in\sC_F^I$, where $Q_{F,I,\chi}(\lambda)$ is the projective cover of the simple baby Verma module $Z_{F,I,\chi}(\lambda)$ in $\sC_F^I$.

For any $A$ and $\lambda\in X$, we hence define the object $Q_{A,\chi}^I(\lambda)$ in $\sC_A$ as
$$Q_{A,\chi}^{I}(\lambda)\coloneqq\Gamma_{A,\chi}(Q_{A,I,\chi}(\lambda)),$$ recalling the definition of $\Gamma_{A,\chi}$ from Section~\ref{Sec4.4}, and viewing $Q_{A,I,\chi}$ as a module in $\sC_A^{I,+}$ via the surjection $U^{I}U_I^{+}\twoheadrightarrow U^I$.
If $A$ is local with residue field $F$, it is clear that
$$Q_{A,\chi}^{I}(\lambda)\otimes_A F\cong Q_{F,\chi}^{I}(\lambda),$$ but in what follows we don't assume $A$ is local unless explicitly mentioned.

\begin{prop}\label{ZIsom}
	Suppose that $\pi(h_\alpha)=0$ for all $\alpha\in R_I$. If $Z_{A,\chi}(\lambda)\cong Z_{A,\chi}(\mu)$ then $\lambda\in W_{I,p}\cdot\mu$.
\end{prop}

\begin{proof}
	That $Z_{A,\chi}(\lambda)$ and $Z_{A,\chi}(\mu)$ are isomorphic in $\sC_A$ implies that $\lambda+\bZ I=\mu+\bZ I$ and so $Z_{A,\chi}(\lambda)_{\lambda+\bZ I}$ is isomorphic to $Z_{A,\chi}(\mu)_{\mu+\bZ I}$ in $\sC_A^I$. We know, as in Section~\ref{Sec4.4}, that $Z_{A,\chi}(\lambda)_{\lambda+\bZ I}\cong Z_{A,I,\chi}(\lambda)$ in $\sC_A^I$. Therefore, we have that $Z_{A,I,\chi}(\lambda)\cong Z_{A,I,\chi}(\mu)$ in $\sC_A^I$. By Proposition~\ref{ZRegIsom}, we therefore have that $\lambda\in W_{I,p}\cdot\mu$
\end{proof}

\begin{rmk}
	We already saw the converse to this proposition in Proposition~\ref{ImpliesIsom}. That result now follows easily since $\lambda\in W_{I,p}\cdot\mu$ implies $Z_{\bK,I,\chi}(\lambda)\cong Z_{\bK,I,\chi}(\mu)$, and so implies $Z_{\bK,I,\chi}(\lambda)\otimes_{\bK} A\cong Z_{\bK,I,\chi}(\mu)\otimes_{\bK} A$, and so implies $Z_{A,I,\chi}(\lambda)=\Theta_A^I(Z_{\bK,I,\chi}(\lambda)\otimes_{\bK} A)\cong \Theta_A^I(Z_{\bK,I,\chi}(\mu)\otimes_{\bK} A)=Z_{A,I,\chi}(\mu)$, which finally implies $Z_{A,\chi}(\lambda)\cong\Gamma_{A,\chi}(Z_{A,I,\chi}(\lambda))\cong \Gamma_{A,\chi}(Z_{A,I,\chi}(\mu))\cong Z_{A,\chi}(\mu)$.
\end{rmk}

\begin{cor}\label{ZIsomFull}
	Let $\lambda,\mu\in X$ and suppose that $\pi(h_\alpha)=0$ for all $\alpha\in R_I$. Then $Z_{A,\chi}(\lambda)\cong Z_{A,\chi}(\mu)$ if and only if $\lambda\in W_{I,p}\cdot\mu$.
\end{cor}

\begin{prop}
	Let $\lambda,\mu\in X$ and suppose that $\pi(h_\alpha)=0$ for all $\alpha\in R_I$. Then $Q_{A,\chi}^I(\lambda)\cong Q_{A,\chi}^I(\mu)$ if and only if $\lambda\in W_{I,p}\cdot\mu$.
\end{prop}

\begin{proof}
	If $Q_{A,\chi}^I(\lambda)\cong Q_{A,\chi}^I(\mu)$ then we have $\lambda+\bZ I=\mu+\bZ I$ as the maximal element of $X/\bZ I$ appearing non-trivially in the grading. Then $Q_{A,I,\chi}(\lambda)\cong Q_{A,\chi}^I(\lambda)_{\lambda+\bZ I}\cong Q_{A,\chi}^I(\mu)_{\mu+\bZ I}\cong Q_{A,I,\chi}(\mu)$, which implies $\lambda\in W_{I,p}\cdot\mu$ by Proposition~\ref{QIisom}.
	
	Conversely, if $\lambda\in W_{I,p}\cdot\mu$ then $Q_{A,I,\chi}(\lambda)\cong Q_{A,I,\chi}(\mu)$ in $\sC_A^I$, which obviously implies that $Q_{A,\chi}^I(\lambda)\cong Q_{A,\chi}^I(\mu)$.
\end{proof}

We write $\Lambda_I$ for the set of $W_{I,p}$-dot-orbits on $X$, so that elements of $\Lambda$ index the pairwise non-isomorphic $Q_{A,\chi}^I(\lambda)$ and also the pairwise non-isomorphic baby Verma modules $Z_{A,\chi}(\lambda)$. As for the set $\Lambda$ in the previous section, we also use the notation $\Lambda_I$ for a fundamental domain of the $W_{I,p}$-dot-action on $X$, allowing us to view elements of $\Lambda_I$ as being in $X$. Note, of course, that if $\lambda$ and $\mu$ in $X$ lie in the same $W_{I,p}$-dot-orbit then  $\lambda+\bZ I=\mu+\bZ I$, so we may talk about $\lambda+\bZ I$ for $\lambda\in\Lambda_I$ without any ambiguity.

\begin{prop}\label{QIZFilt}
	Let $\lambda\in X$ and suppose that $\pi(h_\alpha)=0$ for all $\alpha\in R_I$. Then $Q_{A,\chi}^{I}(\lambda)$ has a $Z$-filtration.
\end{prop}

\begin{proof}
	First, we show $Q_{A,I,\chi}(\lambda)$ has a $Z$-filtration in $\sC_A^I$. Since $Q_{A,I,\chi}(\lambda)=\Theta_A^I(Q_{\bK,I,\chi}(\lambda)\otimes_{\bK} A)$ and $\Theta_A^I(Z_{\bK,\chi}(\mu)\otimes_{\bK} A)=Z_{A,\chi}(\mu)$ for all $\mu\in X$, it is enough to show that $Q_{\bK,I,\chi}(\lambda)\otimes_{\bK} A\in\sC_A^\circ$ has a filtration with sections of the form $Z_{\bK,\chi}(\mu)\otimes_{\bK} A$ for $\mu\in X$.
	Since $-\otimes_{\bK}A$ is an exact functor $\sC_{\bK}\to\sC_A^\circ$, this follows from the fact that $Q_{\bK,I,\chi}(\lambda)$ has a $Z$-filtration in $\sC_\bK$ as in \cite[Prop. 11.18]{Jan}.
	
	Now, viewing $Q_{A,I,\chi}(\lambda)$ as an element of $\sC_A^{I,+}$ via trivial extension, it has a filtration in $\sC_A^{I,+}$ with sections of the form $Z_{A,I,\chi}(\mu)\in\sC_A^{I,+}$ for $\mu\in X$. The induction functor $\Gamma_{A,\chi}$ is an exact functor $\sC_A^{I,+}\to\sC_A$, so $Q_{A,\chi}^I(\lambda)=\Gamma_{A,\chi}(Q_{A,I,\chi}(\lambda))$ has a filtration with sections of the form $\Gamma_{A,\chi}(Z_{A,I,\chi}(\mu))\in\sC_A$ for $\mu\in X$. Since $\Gamma_{A,\chi}(Z_{A,\chi}(\mu))$ is equal to $Z_{A,\chi}(\mu)$ in $\sC_A$, the result follows.
\end{proof}

\begin{rmk}
	The proof of Proposition~\ref{QIZFilt} demonstrates the structure of a $Z$-filtration of $Q_{A,\chi}^{I}(\lambda)$. Namely, since the $Z$-filtration of $Q_{\bK,I,\chi}(\lambda)\in\sC_\bK^{I}$ consists of $\left| W_I\cdot d\lambda\right|$ factors each isomorphic to $Z_{\bK,I,\chi}(\lambda)$, the $Z$-filtration of $Q_{A,\chi}^{I}(\lambda)$ consists of $\left| W_I\cdot d\lambda\right|$ factors each isomorphic to $Z_{A,\chi}(\lambda)$. Furthermore, if $Z_{A,\chi}(\mu)$ is a factor of any $Z$-filtration of $Q_{A,\chi}^I(\lambda)$ then the above discussion shows that $\mu+\bZ I\leq \lambda+\bZ I$, while consideration of the ranks of $Q_{A,\chi}^I(\lambda)$ and $Q_{A,\chi}^I(\lambda)_{\lambda+\bZ I}$ as $A$-modules guarantees that $\mu+\bZ I\geq\lambda+\bZ I$. This then implies that $Z_{A,I,\chi}(\mu)$ appears in a $Z$-filtration of $Q_{A,I,\chi}(\lambda)$, proving that $Z_{A,I,\chi}(\mu)\cong Z_{A,I,\chi}(\lambda)$ and  $Z_{A,\chi}(\mu)\cong Z_{A,\chi}(\lambda)$.
\end{rmk}

\begin{dfn}
	Let $M\in \sC_A$. We say that $M$ has a {\bf $Q$-filtration} if it has a filtration in $\sC_A$ whose sections are all of the form $Q_{A,\chi}^I(\lambda)$ for $\lambda\in X$.
\end{dfn}

\begin{rmk}
	Proposition~\ref{QIZFilt} clearly implies that if $M\in\sC_A$ has a $Q$-filtration then it has a $Z$-filtration.	
\end{rmk}

\begin{prop}\label{ProjResQ}
	Let $M$ be a module in $\sC_A$ with a $Q$-filtration, and suppose that $\pi(h_\alpha)=0$ for all $\alpha\in R_I$. Then $M$ has a projective resolution $P_\bullet$ in $\sC_A$ such that, for each $N\in \sC_A$, there exists an integer $r\geq 0$ such that $$\Hom_{\sC_A}(P_i,N)=0\quad \mbox{for all}\,\, i\geq r.$$
\end{prop}

\begin{proof}
	We first prove this result for $M=Q_{A,\chi}^I(\lambda)$ with $\lambda\in X$. Since $Q_{A,I,\chi}(\lambda)$ is projective in $\sC_A^I$, when we extend it to a module in $\sC_A^{I,+}$ it satisfies the assumptions of Lemma~\ref{ProjResI}. Hence, there exists a resolution $\widetilde{P}_\bullet$ in $\sC_A^{I,+}$ with the property given in that lemma. Applying the functor $\Gamma_{A,\chi}$ to this resolution, we get a projective resolution $P_\bullet=\Gamma_{A,\chi}(\widetilde{P})_\bullet$ of $\Gamma_{A,\chi}(Q_{A,I,\chi}(\lambda))$ which, by Frobenius reciprocity (see Proposition~\ref{FrobRecPara}), satisfies the requirements of this lemma.
	
	As in \cite{AJS} Lemma 2.15, this then implies the result for all $M\in\sC_A$ with a $Q$-filtration using induction and the algebraic mapping cone.
\end{proof}

The following proposition is just Proposition 3.4 in \cite{AJS} adapted to our setting.

\begin{prop}\label{ExtQFilt}
	Let $M\in \sC_A$ have a $Q$-filtration, and let $N$ be a module in $\sC_A$ which is projective as an $A$-module. Suppose that $\pi(h_\alpha)=0$ for all $\alpha\in R_I$, and suppose further that for all maximal ideals $\fm$ of $A$ and all $i> 0$ $$\Ext_{\sC_{A/\fm}}^i(M\otimes_{A}A/\fm,N\otimes_{A} A/\fm)=0.$$ Then the following results hold.
	\begin{enumerate}
		\item The homomorphism space $\Hom_{\sC_A}(M,N)$ is projective as an $A$-module.
		\item For all $A$-algebras $A'$, $$\Ext_{\sC_{A'}}^i(M\otimes_A A',N\otimes_A A')=0\quad\mbox{for all}\,\, i>0.$$ 
		\item For all $A$-algebras $A'$, $$\Hom_{\sC_A}(M,N)\otimes_{A} A'\cong \Hom_{\sC_{A'}}(M\otimes_A A',N\otimes_A A').$$
	\end{enumerate}
\end{prop}

\begin{proof}
	By Proposition~\ref{ProjResQ}, there exists a projective resolution $P_\bullet$ of $M$ in $\sC_A$ and an integer $r\geq 0$ such that $\Hom_{\sC_A}(P_i,N)=0$ for all $i\geq r$. Since $M$ has a $Q$-filtration it has a $Z$-filtration, and so is free over $A$. Furthermore, all the $P_i$ are projective over $A$ since they are projective in $\sC_A$ (recall this argument from the proof of Proposition~\ref{UniqDecomp}). 
	
	The proof then works exactly the same way as the proof of \cite{AJS} Proposition 3.4, once we observe that \cite{AJS} Proposition 3.3 generalises to our setting precisely.
\end{proof}

Now is a good time to note that as well as the equivalence of categories $$\Theta_A^I:(\sC_A^I)^\circ\xrightarrow{\sim}\sC_A^I$$ there is also an equivalence of categories $$\Theta_A^{I,+}:(\sC_A^{I,+})^\circ\xrightarrow{\sim}\sC_A^{I,+}$$ defined in the same way. This works since we don't need the actions of $e_\alpha e_{-\alpha}-e_{-\alpha}e_\alpha$ and $h_\alpha$ to coincide here for $\alpha\notin R_I$.

\begin{lemma}\label{ProjDec}
	Let $M\in \sC_{\bK}^I$ be projective, and suppose that $\pi(h_\alpha)=0$ for all $\alpha\in R_I$. Then $\Theta^I_A(\Phi_{\bK}^{I,+}(M)\otimes A)\in\sC_A^I$ is a direct sum of modules of the form $Q_{A,\chi,I}(\mu)$ for $\mu\in \Lambda_I$.
\end{lemma}

\begin{proof}
	Since $M\in \sC_{\bK}^I$ is projective, $\sF(M)$ is a projective $U^I$-module, where $\sF$ is the forgetful functor. We have that $\sF(M)$ is a projective $U_\chi(\fg_I)$-module, since all modules in $\sC_{\bK}^I$ are in fact $U_\chi(\fg_I)$-modules. Recall that we write $\fp$ for the parabolic $\fp=\fg_I\oplus\fu^{+}$.
	
	We have $\sF(\Phi_{\bK}^{I,+}(M))=\widetilde{\Phi}_{\bK}^{I,+}(\sF(M))$, where $$\widetilde{\Phi}_{\bK}^{I,+}:\Mod(U_\chi(\fg_I))\to\Mod(U_\chi(\fp)),\qquad M\mapsto U_\chi(\fp)\otimes_{U_\chi(\fg_I)} M = U^I U_I^{+}\otimes_{U^I} M.$$ is the induction functor. It is easy to check Frobenius reciprocity, i.e. that for $L$ a $U_\chi(\fg_I)$-module and $N$ a $U_\chi(\fp)$-module, we have $$\Hom_{U_\chi(\fp)}(\widetilde{\Phi}_{\bK}^{I,+}(L),N)=\Hom_{U_\chi(\fg_I)}(L,N).$$ Hence, $\widetilde{\Phi}_{\bK}^{I,+}:\Mod(U_\chi(\fg_I))\to\Mod(U_\chi(\fp))$ sends projectives to projectives, and so $\widetilde{\Phi}_{\bK}^{I,+}(\sF(M))$ is a projective $U_\chi(\fp)$-module.
	
	Since $U_\chi(\fp)$ is free over $U_\chi(\fg_I)$, we conclude that $\widetilde{\Phi}_{\bK}^{I,+}(\sF(M))$ is a projective $U_\chi(\fg_I)$-module. Since $\sF(\Phi_{\bK}^{I,+}(M))=\widetilde{\Phi}_{\bK}^{I,+}(\sF(M))$, we conclude that $\Phi_{\bK}^{I,+}(M)$ is projective in $\sC_\bK^I$.
	
	In particular, this means that $$\Phi_{\bK}^{I,+}(M)=\bigoplus_{\mu\in \Lambda_I} Q_{\bK,I,\chi}(\mu)^{m_\mu}$$ for some $m_\mu\geq 0$. This further means that $$\Phi_{\bK}^{I,+}(M)\otimes_{\bK} A=\bigoplus_{\mu\in \Lambda_I} Q_{\bK,I,\chi}(\mu)^{m_\mu}\otimes_{\bK} A$$ and so $$\Theta_A^{I,+}(\Phi_{\bK}^{I,+}(M)\otimes_{\bK} A)=\bigoplus_{\mu\in \Lambda_I} Q_{A,I,\chi}(\mu)^{m_\mu}.$$
	
	It is straightforward to check that $$\Theta_A^{I,+}(\Phi_{\bK}^{I,+}(M)\otimes_{\bK} A)=\Phi_A^{I,+}(\Theta_A^I(M\otimes_{\bK} A))$$ as elements of $\sC_A^I$. Hence, we conclude that $$\Phi_A^{I,+}(\Theta_A^I(M\otimes_{\bK} A))=\bigoplus_{\mu\in \Lambda_I} Q_{A,I,\chi}(\mu)^{m_\mu}.$$
\end{proof}

\begin{lemma}\label{PhiQFilt}
	Suppose that $\pi(h_\alpha)=0$ for all $\alpha\in R_I$. Let $M\in\sC_A^{I,+}$ (resp. in $\sC_A^I$) which has a filtration with sections of the form $Q_{A,I,\chi}(\lambda)$ for $\lambda\in X$. Then $\Gamma_{A,\chi}(M)$ (resp. $\Phi_A^I(M)$) has a $Q$-filtration.
\end{lemma}
\begin{proof}
	If $M\in \sC_A^{I,+}$ has a filtration with sections of the form $Q_{A,I,\chi}(\lambda)$, then $\Gamma_{A,\chi}(M)$ has a $Q$-filtration since $\Gamma_{A,\chi}$ is exact and $\Gamma_{A,\chi}(Q_{A,I,\chi}(\lambda))=Q_{A,\chi}^I(\lambda)$.
	
	If $M\in \sC_A^I$ has a filtration with sections of the form $Q_{A,I,\chi}(\lambda)$, then we would like to show the same for $\Phi_A^{I,+}(M)$. Firstly, we observe that since each $Q_{A,I,\chi}(\lambda)$ is projective in $\sC_A^I$, we in fact have that $M$ is a direct sum of modules of the form $Q_{A,I,\chi}(\lambda)$. So we may assume $M=Q_{A,I,\chi}(\lambda)$. 
	
	We may then apply Lemma~\ref{ProjDec} to conclude that, in $\sC_A^I$, 
	$$\Phi_A^{I,+}(Q_{A,I,\chi}(\lambda))\cong\bigoplus_{\mu\in \Lambda_I} Q_{A,I,\chi}(\mu)^{m_\mu}$$ for some $m_\mu\geq 0$. Note that each $Q_{A,I,\chi}(\mu)^{m_\mu}$ lies entirely in grade $\mu+\bZ I$ and that $\Phi_A^{I,+}(Q_{A,I,\chi}(\lambda))$ is in $\sC_A^{I,+}$. Let $\mu_0+\bZ I$ be maximal with the property that $\Phi_A^{I,+}(Q_{A,I,\chi}(\lambda))_{\mu+\bZ I}\neq 0$. Then each $Q_{A,I,\chi}(\mu)^{m_\mu}$ in the decomposition with $\mu+\bZ I=\mu_0+\bZ I$ is a submodule of $\Phi_A^{I,+}(Q_{A,I,\chi}(\lambda))$ in $\sC_A^{I,+}$. We may hence put these at the bottom of a filtration, take the quotient, and iterate the process. At the end of this process, we have a filtration of $\Phi_A^{I,+}(Q_{A,I,\chi}(\lambda))$ in $\sC_A^{I,+}$ whose sections are all of the form $Q_{A,I,\chi}(\mu)$ (viewed as elements of $\sC_A^{I,+}$). We may then apply the first part of the result to conclude that $\Phi_A^I(M)$ has a $Q$-filtration.
	
\end{proof}

\begin{prop}\label{QHomIm}
	Let $M\in\sC_A$ and suppose that $\pi(h_\alpha)=0$ for all $\alpha\in R_I$. There exists a projective module $P\in\sC_A$ with a $Q$-filtration such that $M$ is a homomorphic image of $P$.
\end{prop}

\begin{proof}
	If we view $M$ as lying in $\sC_A^I$, then Lemma~\ref{RegProjCov} shows that  there exists $Q\in\sC_A^I$ with $$Q=\bigoplus_{\mu\in \Lambda_I} Q_{A,I,\chi}(\mu)^{m_\mu}$$ for some $m_\mu\geq 0$ with $Q\twoheadrightarrow M$ in $\sC_A^I$. Lemma~\ref{PhiQFilt} then implies that $P\coloneqq\Phi_A^I(Q)$ has a $Q$-filtration, and it is projective since $Q$ is projective in $\sC_A^I$ and $\Phi_A^I$ sends projectives to projectives. Furthermore, since $Q\twoheadrightarrow M$ in $\sC_A^I$, we get that $P=\Phi_A^I(Q)$ surjects onto $M$ in $\sC_A$.
\end{proof}

\begin{cor}
	Suppose that $\pi(h_\alpha)=0$ for all $\alpha\in R_I$. Let $M\in\sC_A$ have a $Q$-filtration. Then $M$ is projective in $\sC_A$ if and only if $M\otimes_{A} A/\fm$ is projective for all maximal ideals $\fm$ of $A$.
\end{cor}

\begin{proof}
	The proof works the same way as that of Corollary 3.5 in \cite{AJS}, using our Proposition~\ref{projscal}, Proposition~\ref{ExtQFilt}, and Proposition~\ref{QHomIm} where relevant.
\end{proof}

Let us consider what $P$ is (or could be) in Proposition~\ref{QHomIm} when $M=Q_{A,\chi}^I(\lambda)$. There is a natural surjection $\Phi_A^{I,+}(Q_{A,I,\chi}(\lambda))\twoheadrightarrow Q_{A,I,\chi}(\lambda)$ in $\sC_A^{I,+}$. Hence $P=\Phi_A^I(Q_{A,I,\chi}(\lambda))=\Gamma_{A,\chi}(\Phi_A^{I,+}(Q_{A,I,\chi}(\lambda)))$ surjects onto $Q_{A,\chi}^I(\lambda)=\Gamma_{A,\chi}(Q_{A,I,\chi}(\lambda))$, and it is projective and has a $Q$-filtration by Lemma~\ref{PhiQFilt}.

For each $\lambda\in X$, let us therefore define $$\Xi_{A,\chi}^I(\lambda)=\Phi_A^I(Q_{A,I,\chi}(\lambda))\in\sC_A.$$ 

We note now that $\Phi_A^{I,+}(Q_{A,I,\chi}(\lambda))_{\lambda+\bZ I}=Q_{A,I,\chi}(\lambda)\in\sC_A^I$ (as in Section~\ref{Sec4.4}, since $Q_{A,I,\chi}(\lambda)_{\lambda+\bZ I}=Q_{A,I,\chi}(\lambda)$) and $\Phi_A^{I,+}(Q_{A,I,\chi}(\lambda))_{\mu+\bZ I}\neq 0$ implies $\mu+\bZ I\geq\lambda+\bZ I$. By the proof of Lemma~\ref{PhiQFilt}, $\Phi_A^{I,+}(Q_{A,I,\chi}(\lambda))$ has a filtration with factors of the form $Q_{A,I,\chi}(\mu)$. By the previous discussion, we must have $\mu+\bZ I>\lambda+\bZ I$ for each such $\mu$ which is not $\lambda$, and $Q_{A,I,\chi}(\lambda)$ appears exactly once (at the top). Hence, applying $\Gamma_{A,\chi}$, we conclude that $\Xi_{A,\chi}^I(\lambda)$ has a $Q$-filtration with $Q_{A,\chi}^I(\lambda)$ appearing exactly once (at the top) and with each other factor being of the form $Q_{A,\chi}^I(\mu)$ for $\mu\in X$ with $\mu+\bZ I>\lambda+\bZ I$.

\subsection{Projective covers}\label{Sec8.2}

We would now like to prove a result similar to Corollary~\ref{projAF} when $\chi$ is not necessarily regular nilpotent, i.e., to find, when $A$ is a local algebra with residue field $F$, a projective module $Q_{A,\chi}(\lambda)\in\sC_A$ with $Q_{A,\chi}(\lambda)\otimes_{A} F\cong Q_{F,\chi}(\lambda)$ in $\sC_F$. To do this, we need to establish some preliminary results.

\begin{prop}
	Suppose that $\pi(h_\alpha)=0$ for all $\alpha\in R_I$. Let $\lambda\in X$ and let $N\in\sC_A$ with $\lambda+\bZ I\not<\mu+\bZ I$ for all $\mu+\bZ I\in X/\bZ I$ with $N_{\mu+\bZ I}\neq 0$. Then $$\Ext_{\sC_A}^1(Q_{A,\chi}^I(\lambda),N)=0.$$
\end{prop}

\begin{proof}
	Let $$0\to N\xrightarrow{i} M\xrightarrow{p} Q_{A,\chi}^I(\lambda)\to 0$$ be an exact sequence in $\sC_A$. Then $$0\to N_{\lambda+\bZ I}\xrightarrow{i_{\lambda+\bZ I}} M_{\lambda+\bZ I}\xrightarrow{p_{\lambda+\bZ I}}Q_{A,\chi}^I(\lambda)_{\lambda+\bZ I}\to 0$$ is an exact sequence in $\sC_A^I$. Since $Q_{A,\chi}^I(\lambda)_{\lambda+\bZ I}\cong Q_{A,I,\chi}(\lambda)$ in $\sC_A^I$, the exact sequence in $\sC_A^I$ is $$0\to N_{\lambda+\bZ I}\xrightarrow{i_{\lambda+\bZ I}} M_{\lambda+\bZ I}\xrightarrow{p_{\lambda+\bZ I}} Q_{A,I,\chi}(\lambda)\to 0,$$ which splits since $Q_{A,I,\chi}(\lambda)$ is projective in $\sC_A^I$. Hence there exists a map $$f:Q_{A,I,\chi}(\lambda)\to M_{\lambda+\bZ I}\subseteq M$$ in $\sC_A^I$ with $p_{\lambda+\bZ I}\circ f=\Id$. Since $N_{\sigma+\bZ I}=Q_{A,\chi}^I(\lambda)_{\sigma+\bZ I}=0$ for all $\sigma+\bZ I\in X/\bZ I$ with $\sigma+\bZ I>\lambda+\bZ I$, we must have that $M_{\sigma+\bZ I}=0$ for all $\sigma+\bZ I\in X/\bZ I$ with $\sigma+\bZ I>\lambda+\bZ I$. Hence, $f$ is in fact a morphism in $\sC_A^{I,+}$.
	
	By Frobenius reciprocity, we get a morphism $$\widetilde{f}:Q_{A,\chi}^I(\lambda)=\Gamma_{A,\chi}(Q_{A,I,\chi}(\lambda))\to M$$ extending $f$. Given $u\otimes m\in Q_{A,\chi}^I(\lambda)$, we then get $$p\circ\widetilde{f}(u\otimes m)=p(uf(m))=u(pf(m))=u\otimes m.$$ Hence, our original exact sequence splits.
\end{proof}

\begin{cor}\label{QExt}
	Suppose that $\pi(h_\alpha)=0$ for all $\alpha\in R_I$. Let $\lambda,\mu\in X$ with $\lambda+\bZ I\not<\mu+\bZ I$. Then $$\Ext_{\sC_A}^1(Q_{A,\chi}^I(\lambda),Q_{A,\chi}^I(\mu))=0.$$
\end{cor}

\begin{cor}
	Let $M\in\sC_A$ have a $Q$-filtration, and suppose that $\pi(h_\alpha)=0$ for all $\alpha\in R_I$. Then $M$ has a (possibly different) $Q$-filtration $$0=M_0\subseteq M_1\subseteq M_2\subseteq\cdots\subseteq M_r=M$$ with sections $M_{i}/M_{i-1}\cong Q_{A,\chi}(\lambda_i)$ for $\lambda_i\in X$ such that $\lambda_i+\bZ I>\lambda_j+\bZ I$ implies $i<j$.
\end{cor}

\begin{lemma}\label{QFiltLambda}
	Suppose that $\pi(h_\alpha)=0$ for all $\alpha\in R_I$. Let $M\in\sC_A$ have a $Q$-filtration, and suppose $\lambda+\bZ I\in X/\bZ I$ has the property that $M_{\lambda+\bZ I}\neq 0$ and $M_{\sigma+\bZ I}=0$ for all $\sigma+\bZ I\in X/\bZ I$ with $\sigma+\bZ I>\lambda+\bZ I$. Then $M$ has a (possibly different) $Q$-filtration $$0=M_0\subseteq M_1\subseteq M_2\subseteq\cdots\subseteq M_r=M$$ with sections $M_{i}/M_{i-1}\cong Q_{A,\chi}(\lambda_i)$ for $\lambda_i \in X$ with the property that there exists $1\leq k\leq r$ such that $$\lambda_i+\bZ I=\lambda+\bZ I,\quad \mbox{for all}\quad i\leq k$$
	and $$\lambda_i+\bZ I\neq\lambda+\bZ I\quad\mbox{implies}\quad i>k.$$
\end{lemma}

\begin{proof}
	Suppose not. By the assumptions on $\lambda+\bZ I$, there must exist $1\leq l\leq r$ with $\lambda_l+\bZ I=\lambda+\bZ I$. Then there exists $1\leq i<j\leq r$ such that $\lambda_i+\bZ I\neq\lambda+\bZ I$, and $\lambda_j+\bZ I =\lambda+\bZ I$. Therefore, there exists $1\leq i<r$ such that $\lambda_i+\bZ I\neq \lambda+\bZ I$ and $\lambda_{i+1}+\bZ I=\lambda+\bZ I$. There is then a short exact sequence $$0\to \frac{M_i}{M_{i-1}}\to \frac{M_{i+1}}{M_{i-1}}\to \frac{M_{i+1}}{M_i}\to 0$$ which we may of course write as $$0\to Q_{A,\chi}^I(\lambda_i)\to \frac{M_{i+1}}{M_{i-1}}\to Q_{A,\chi}^I(\lambda_{i+1})\to 0.$$ If $\Ext_{\sC_A}^i(Q_{A,\chi}^I(\lambda_{i+1}),Q_{A,\chi}^I(\lambda_i))\neq 0$, then by Corollary~\ref{QExt}, we have $$\lambda_i+\bZ I>\lambda_{i+1}+\bZ I=\lambda+\bZ I.$$ However, $M_{\lambda_i+\bZ I}\neq 0$ since $Q_{A,\chi}^I(\lambda_i)$ is a section of a $Q$-filtration of $M$ and $Q_{A,\chi}^I(\lambda_i)_{\lambda_i+\bZ I}\neq 0$. This contradicts our assumption on $\lambda+\bZ I$, and so we must have that $\Ext_{\sC_A}^i(Q_{A,\chi}^I(\lambda_{i+1}),Q_{A,\chi}^I(\lambda_i))=0$. In particular, the short exact sequence above splits, and so we may swap $Q_{A,\chi}^I(\lambda_i)$ and $Q_{A,\chi}^I(\lambda_{i+1})$ in the $Q$-filtration. Iterating this argument, we obtain the result.
\end{proof}

\begin{rmk}
	The proof of the previous lemma shows that, if $M$ has a $Q$-filtration, then $M$ also has a $Q$-filtration with the given property and with the same $Q_{A,\chi}^I(\lambda)$ appearing the same number of times in the new filtration as in the original one.
\end{rmk}

For the remainder of the subsection, we assume that $A$ is a local ring with residue field $F$. 

\begin{prop}\label{LeviUniqDecomp}
	Suppose that $\pi(h_\alpha)=0$ for all $\alpha\in R_I$. Suppose further that $A$ is local with residue field $F$, and let $M\in\sC_A$ have a $Q$-filtration. Let $$0=M_0\subseteq M_1\subseteq M_2\subseteq\cdots\subseteq M_r=M$$ and $$0=N_0\subseteq N_1\subseteq N_2\subseteq\cdots\subseteq N_s=N$$ be two $Q$-filtrations of $M$, with sections $M_{i}/M_{i-1}\cong Q_{A,\chi}(\lambda_i)$ and $N_{j}/N_{j-1}\cong Q_{A,\chi}(\mu_j)$ for $\lambda_i,\mu_j\in X$. Then $r=s$ and there exists $\sigma\in S_r$ such that $$Q_{A,\chi}^I(\lambda_i)\cong Q_{A,\chi}^I(\mu_{\sigma(i)})$$ for all $1\leq i\leq r$.
\end{prop}

\begin{proof}
	Let us apply induction on $r+s$. If $r=s=1$, it is clear. Let $\lambda+\bZ I\in X/\bZ I$ have the property that $M_{\lambda+\bZ I}\neq 0$ and $M_{\kappa+\bZ I}=0$ for all $\kappa+\bZ I\in X/\bZ I$ with $\kappa+\bZ I>\lambda+\bZ I$. By Lemma~\ref{QFiltLambda}, we may assume that there exists $1\leq k_1\leq r$ and $1\leq k_2\leq s$ such that $\lambda_i+\bZ I=\lambda+\bZ I$ for all $i\leq k_1$, $\lambda_i+\bZ I\neq\lambda+\bZ I$ for all $i>k_1$, $\mu_j+\bZ I=\lambda+\bZ I$ for all $j\leq k_2$ and $\mu_j+\bZ I\neq \lambda+\bZ I$ for all $j>k_2$. As in the previous remark, we see that making this assumption doesn't change which $Q_{A,\chi}^I(\mu)$ appear or how many times, so this assumption is permissible.
	
	This implies, in particular, that $M_{\lambda+\bZ I}=(M_{k_1})_{\lambda+\bZ I}=(M_{k_2})_{\lambda+\bZ I}$. Since, if $\lambda_i+\bZ I=\lambda+\bZ I$, we have that $Q_{A,\chi}^I(\lambda_i)_{\lambda+\bZ I}=Q_{A,I,\chi}(\lambda_i)\in\sC_A^I$, and each $Q_{A,I,\chi}(\lambda_i)$ is projective in $\sC_A$, we get in $\sC_A^I$ that \begin{equation*}
		\begin{split} M_{\lambda+\bZ I} & \cong Q_{A,I,\chi}(\lambda_1)\oplus\cdots \oplus Q_{A,I,\chi}(\lambda_{k_1}) \\ &  \cong Q_{A,I,\chi}(\mu_1)\oplus\cdots \oplus Q_{A,I,\chi}(\mu_{k_2}).
		\end{split}
	\end{equation*}
	Proposition~\ref{UniqDecomp} then says that $k_1=k_2$ and that there exists $\omega\in S_{k_1}$ such that $$Q_{A,I,\chi}(\lambda_i)\cong Q_{A,I,\chi}(\mu_{\omega(i)})$$ for all $1\leq i\leq k_1$, and so $$Q_{A,\chi}^I(\lambda_i)\cong Q_{A,\chi}^I(\mu_{\omega(i)})$$ for all $1\leq i\leq k_1$.
	
	We now want to show that $M_{k_1}=N_{k_2}$. Let us write $L$ for the submodule $U_\chi M_{\lambda+\bZ I}$ of $M$ (i.e. the minimal submodule of $M$ containing $M_{\lambda+\bZ I}$). Note that $(U_\chi M_{\lambda+\bZ I})_{\nu+\bZ I}=(U_\chi)_{\nu-\lambda+\bZ I}M_{\lambda+\bZ I}=M_{\nu+\bZ I}\cap U_\chi M_{\lambda+\bZ I}$, and $$(U_\chi M_{\lambda+\bZ I})_{\nu+\bZ I}^{d\mu}=\sum_{\substack{d\epsilon\in\fh^{*}\\ \epsilon\in\lambda+\bZ I +pX}}(U_\chi)_{\nu-\lambda+\bZ I}^{d(\mu-\epsilon)}M_{\lambda+\bZ I}^{d\epsilon}=M_{\nu+\bZ I}^{d\mu}\cap U_\chi M_{\lambda+\bZ I},$$ so $L$ is indeed a submodule of $M$ in $\sC_A$. Since $M_{k_1}$ and $N_{k_2}$ are submodules of $M$, and we have $M_{\lambda+\bZ I}\subseteq M_{k_1}\cap N_{k_2}$, we clearly have $L\subseteq M_{k_1}$ and $L\subseteq N_{k_2}$.
	
	Conversely, we have that $M_1\cong Q_{A,\chi}^I(\lambda_1)$ lies inside $U_\chi\cdot (M_1)_{\lambda+\bZ I}$, that $M_2$ lies inside $M_1+U_\chi\cdot (M_2)_{\lambda+\bZ I}$ and so forth. In particular, we get that $M_{k_1}$ lies inside $U_\chi \cdot(M_{k_1})_{\lambda+\bZ I}=U_\chi M_{\lambda+\bZ I}$, so $M_{k_1}=L$. By the same argument, $N_{k_2}=L$, and so $M_{k_1}=N_{k_2}$.
	
	Now, $M/L$ has two $Q$ filtrations, one with $r-k_1$ factors and one with $s-k_2$ factors. The result follows by induction, since $k_1,k_2\geq 1$.

\end{proof}

\begin{prop}\label{LeviQFiltDecomp}
	Suppose that $A$ is a local ring with residue field $F$, and that $\pi(h_\alpha)=0$ for all $\alpha\in R_I$. Let $M\in\sC_A$ have a $Q$-filtration, and suppose $M=N\oplus P$ in $\sC_A$. Then $N$ and $P$ have $Q$-filtrations.
\end{prop}

\begin{proof}
	Let $\lambda+\bZ I\in X/\bZ I$ be such that $M_{\lambda+\bZ I}\neq 0$ and $M_{\sigma+\bZ I}=0$ for all $\sigma+\bZ I\in X/\bZ I$ with $\sigma+\bZ I>\lambda+\bZ I$. By Lemma~\ref{QFiltLambda}, we may assume that the $Q$-filtration  $$0=M_0\subseteq M_1\subseteq M_2\subseteq\cdots\subseteq M_r=M,$$ with sections $M_{i}/M_{i-1}\cong Q_{A,\chi}(\lambda_i)$ for $\lambda_i \in X$, has the property that there exists $1\leq k\leq r$ such that $$\lambda_i+\bZ I=\lambda+\bZ I,\quad \mbox{for all}\quad i\leq k$$
	and $$\lambda_i+\bZ I\neq\lambda+\bZ I\quad\mbox{implies}\quad i>k.$$ 
	
	
	
	We have that $M_{\lambda+\bZ I}=N_{\lambda+\bZ I}\oplus P_{\lambda+\bZ I}$.  
	In $\sC_A^I$, we observe that $$M_{\lambda+\bZ I}=Q_{A,I,\chi}(\lambda_1)\oplus\cdots\oplus Q_{A,I,\chi}(\lambda_k)$$ since all of the $Q_{A,I,\chi}(\lambda_i)$ are projective. By Lemma~\ref{QFiltDecomp}, both $P_{\lambda+\bZ I}$ and $N_{\lambda+\bZ I}$ decompose into $Q_{A,I,\chi}(\mu)$'s in $\sC_A^I$. More specifically, one may see from the proof of Lemma~\ref{QFiltDecomp} that $(M_1)_{\lambda+\bZ I}=Q_{A,I,\chi}(\lambda_1)$ embeds into, say, $N_{\lambda+\bZ I}$. Since $N$ is a submodule of $M$ in $\sC_A$, and $Q_{A,I,\chi}(\lambda)$ generates $Q_{A,\chi}^I(\lambda)$ in $\sC_A$, we conclude that $M_1\subseteq N$, and so $M/M_1=(N/M_1)\oplus P$. The result follows by induction on the length of the $Q$-filtration.
\end{proof}

\begin{cor}\label{ProjQFilt}
	Suppose that $A$ is a local ring with residue field $F$ and that $\pi(h_\alpha)=0$ for all $\alpha\in R_I$. Any $M\in\sC_A$ which is projective in $\sC_A$ has a $Q$-filtration.
\end{cor}
\begin{proof}
	By Proposition~\ref{QHomIm}, there exists a projective module $P\in\sC_A$ with a $Q$-filtration such that $P$ surjects onto $M$ in $\sC_A$. Since $M$ is projective in $\sC_A$, the surjection splits and $M$ is a direct summand of $P$. The result then follows from Proposition~\ref{LeviQFiltDecomp}.
\end{proof}

We would like to know how many times a given $Q_{A,\chi}^I(\lambda)$ appears in  a $Q$-filtration of certain projective modules. In order to determine this, let us briefly discuss how the functor $\bD$ discussed earlier interacts with our standing assumptions in this section.

\begin{prop}
	Let $A$ be a commutative Noetherian $U^0$-algebra with structure map $\pi:U^0\to A$. Suppose $\pi(h_\alpha)=0$ for all $\alpha\in R_I$. Then $\overline{\pi}(h_\alpha)=\,^{\tau}\pi(h_\alpha)=\bD\pi(h_\alpha)=\overline{\bD}\pi(h_\alpha)=0$ for all $\alpha\in R_I$.
\end{prop}
\begin{proof}
	This is clear for $\overline{\pi}$, and the result for $\bD\pi$ will follow from the result for $\,^{\tau}\pi$. The result for $\overline{\bD}\pi$ will be similar, so we only prove $\,^{\tau}\pi(h_\alpha)=0$ for all $\alpha\in R_I$. We have $$\tau^{-1}(h_\alpha)=\tau^{-1}([e_\alpha,e_{-\alpha}])=[\tau^{-1}(e_\alpha),\tau^{-1}(e_{-\alpha})]=k[e_{-w_I\alpha},e_{w_I\alpha}]=k h_{w_I\alpha}$$ for some $k\in\bK^{*}$ since $\tau^{-1}(\fg_\alpha)=\fg_{-w_I\alpha}$ by \cite{Jan5}. Since $\alpha\in R_I$, $w_I\alpha\in R_I$, and so $\pi(h_{w_I\alpha})=0$. Thus $\,^{\tau}\pi(h_\alpha)=\pi(\tau^{-1}(h_\alpha))=0$.
\end{proof}

In particular, once we assume $\pi(h_\alpha)=0$ for all $\alpha\in R_I$, all the results which need this assumption will also hold in categories over the modified $U^0$-algebras. We will use this without comment going forward. Furthermore, we note from the calculations in \cite{Jan5} that $\tau$ sends $U^I$ to $U^I$, $U_I^{+}$ to $U_I^{-}$, and vice versa. So the automorphism $\tau:\fg_I\to \fg_I$ which we use to define the functor $\bD:\sC_A^I\to\sC_{\bD A}^I$ is just the restriction of the automorphism $\tau:\fg\to\fg$ that we have been using throughout.

\begin{prop}\label{CommDiag1}
	Suppose $\pi(h_\alpha)=0$ for all $\alpha\in R_I$. The following diagram commutes:
	$$\xymatrix{
		(\sC_A^I)^\circ\ar@{->}[rr]^{\Theta_A^I} \ar@{->}[d]^{\bD} & & \sC_A^I \ar@{->}[d]^{\bD}& & \\
		(\sC_{\bD A}^I)^\circ \ar@{->}[rr]^{\Theta_{\bD A}^I} & & \sC_{\bD A}^I. & & \\
	}
	$$
\end{prop}

\begin{proof}
	It is not difficult to check that, for $\lambda+\bZ I\in X/\bZ I$, $\mu\in\lambda+\bZ I+pX$ and $M\in\sC_A^I$, we have $\bD(M)_{\lambda+\bZ I}^{d\mu}\cong\bD(M_{\lambda+\bZ I}^{-d(\tau^{-1}(\mu))})$ as $U^0\otimes A$-modules. With this in mind, the computation is straightforward, noting that $\tau^{-1}(\fg_\alpha)\subseteq \fg_{-w_I\alpha}$ for all $\alpha\in R_I$.
\end{proof}

Note that $\bD A=A$ if $\pi(h)=0$ for all $h\in\fh$, and so, if $A=F$ is a field and this holds, $\bD$ is a duality on $\sC_F^I$.

\begin{prop}\label{CommDiag2}
	The following diagram commutes:
	$$\xymatrix{
		\sC_\bK^I \ar@{->}[rr]^{-\otimes_{\bK} A} \ar@{->}[d]^{\bD} & & (\sC_A^I)^\circ \ar@{->}[d]^{\bD}& & \\
		\sC_{\bK}^I \ar@{->}[rr]^{-\otimes_{\bK} A} & & (\sC_{A}^I)^\circ. & & \\
	}
	$$
\end{prop}

\begin{proof}
	This is straightforward once we observe that $\Hom_{A}(M\otimes_{\bK}A,A)$ and $\Hom_{\bK}(M,\bK)\otimes_{\bK}A$ are isomorphic as $U_\chi\otimes A$-modules.
\end{proof}

The following proposition tells us that the functor $\bD$ ``fixes'' irreducible objects in $\sC_F$.

\begin{prop}\label{DualIrred}
	Suppose that $A=F$ is a field and that $\pi(h_\alpha)=0$ for all $\alpha\in R_I$. Then $\bD(L_{F,\chi}(\lambda))\cong L_{\bD F,\chi}(\lambda)$ for all $\lambda\in X$.
\end{prop}

\begin{proof}
	Let $\lambda\in X$. Then one may easily check that $\bD(L_{F,\chi}(\lambda))_{\lambda+\bZ I}$ is isomorphic to $\bD(L_{F,\chi}(\lambda)_{\lambda+\bZ I})$ in $\sC_F^I$, since $\tau$ preserves $U^I$. We know that $L_{F,\chi}(\lambda)_{\lambda+\bZ I}\cong Z_{F,I,\chi}(\lambda)$ in $\sC_F$, since the irreducible $Z_{F,\chi}(\lambda)_{\lambda+\bZ I}=Z_{F,I,\chi}(\lambda)$ surjects onto the non-trivial $L_{F,\chi}(\lambda)_{\lambda+\bZ I}$. Then $\bD(Z_{F,I,\chi}(\lambda))=\bD(Z_{\bK,I,\chi}(\lambda))\otimes_{\bK} F$ from Propositions~\ref{CommDiag1} and \ref{CommDiag2}. By \cite{Jan}, we have $\bD(Z_{\bK,I,\chi}(\lambda))\cong Z_{\bK,I,\chi}(\lambda)$ in $\sC_\bK^I$, and so $\bD(Z_{F,I,\chi}(\lambda))\cong Z_{\bD F,I,\chi}(\lambda)$. We therefore have $\bD(L_{F,\chi}(\lambda))_{\lambda+\bZ I}\cong Z_{\bD F,I,\chi}(\lambda)$. Since $\bD(L_{F,\chi}(\lambda))$ is an irreducible object in $\sC_{\bD F}$ from the anti-equivalence of Proposition~\ref{Antiequiv}, we have $\bD(L_{F,\chi}(\lambda))\cong L_{\bD F,\chi}(\mu)$ for some $\mu\in X$. Then we have $Z_{\bD F,I,\chi}(\lambda)\cong Z_{\bD F,I,\chi}(\mu)$ and so $\lambda\in W_{I,p}\cdot\mu$ by Proposition~\ref{ZIsom}. Thus, $L_{\bD F,\chi}(\mu)\cong L_{\bD F,\chi}(\lambda)$, and so $\bD(L_{F,\chi}(\lambda))\cong L_{\bD F,\chi}(\lambda)$.
\end{proof}

We aim to generalise Corollary~\ref{projAF} to the current setting. This involves an approach similar to the proof of Proposition 4.18 in \cite{AJS}. Let $\nu+\bZ I\in X/\bZ I$. Recall that $Q_{F,\chi}(\lambda)\in\sC_F$ is the projective cover of the irreducible module $L_{F,\chi}(\lambda)\in\sC_F$. Let us write $Q_{F,\chi}^{\nu+\bZ I}(\lambda)$ for the object $T^{\nu+\bZ I}(Q_{F,\chi}(\lambda))\in\sC_F(\leq\nu+\bZ I)$. In order to understand the structure of the projective covers $Q_{F,\chi}(\lambda)$ and $Q_{F,\chi}^{\nu+\bZ I}(\lambda)$, we need the following proposition. Its proof, and that of the theorem following it, uses techniques similar to those of Nakano in \cite{Nak}.
\begin{prop}
	Suppose $\pi(h_\alpha)=0$ for all $\alpha\in R_I$. The module $Q_{F,\chi}^I(\lambda)$ is the projective cover of $Z_{F,I,\chi}(\lambda)$ in $\sC_F^{I,-}$.
\end{prop}

\begin{proof}
	Recall that $Q_{F,\chi}^I(\lambda)=\Gamma_{F,\chi}(Q_{F,I,\chi}(\lambda))$, where $Q_{F,I,\chi}(\lambda)$ is the projective cover of $Z_{F,I,\chi}(\lambda)$ in $\sC_F^I$. By Frobenius reciprocity, we have $$\Hom_{\sC_F^I}(Q_{F,I,\chi}(\lambda),\Gamma_{F,\chi}(Q_{F,I,\chi}(\lambda)))\cong \Hom_{\sC_{F}^{I,-}}(\Phi_F^{I,-}(Q_{F,I,\chi}(\lambda)),\Gamma_{F,\chi}(Q_{F,I,\chi}(\lambda)).$$
	The morphism $m\mapsto 1\otimes m$ in $\sC_F^I$ thus extends to a morphism $$\Phi_F^{I,-}(Q_{F,I,\chi}(\lambda))\to\Gamma_{F,\chi}(Q_{F,I,\chi}(\lambda))$$ in $\sC_F^{I,-}$, which is surjective since $1\otimes Q_{F,I,\chi}(\lambda)$ generates $\Gamma_{F,\chi}(Q_{F,I,\chi}(\lambda))$ in $\sC_F^{I,-}$. It is easy to see that $$\dim_F(\Phi_F^{I,-}(Q_{F,I,\chi}(\lambda)))=p^{\dim\fu^{-}}\dim_F(Q_{F,I,\chi}(\lambda))=\dim_F(\Gamma_{F,\chi}(Q_{F,I,\chi}(\lambda)))$$ and so they are isomorphic in $\sC_F^{I,-}$.
	
	Now, $\Phi_F^{I,-}(Q_{F,I,\chi}(\lambda))$ is projective in $\sC_F^{I,-}$, and we have $$\Hom_{\sC_F^{I,-}}(\Phi_F^{I,-}(Q_{F,I,\chi}(\lambda)),Z_{F,I,\chi}(\mu))\cong \Hom_{\sC_F^I}(Q_{F,I,\chi}(\lambda),Z_{F,I,\chi}(\mu))$$
	and
	$$\Hom_{\sC_F^I}(Q_{F,I,\chi}(\lambda),Z_{F,I,\chi}(\mu))=\twopartdef{F}{Z_{F,I,\chi}(\lambda)\cong Z_{F,I,\chi}(\mu),}{0}{\mbox{otherwise.}}$$
	
	This implies that $Q_{F,\chi}^I(\lambda)\cong \Phi_F^{I,-}(Q_{F,I,\chi}(\lambda))$ is a projective indecomposable module with unique irreducible head $Z_{F,I,\chi}(\lambda)$ in $\sC_F^{I,-}$. The result follows.
\end{proof}

Since $Q_{F,\chi}(\lambda)$ is projective in $\sC_F$, it has a $Q$-filtration by Proposition~\ref{ProjQFilt}. Restricting to $\sC_F^{I,-}$, the $Q_{F,\chi}^I(\mu)$ are projective, and so we have $$Q_{F,\chi}(\lambda)=\bigoplus_{\mu\in\Lambda_I} Q_{F,\chi}^I(\mu)^{m_\mu}$$ where each $$m_\mu=(Q_{F,\chi}(\lambda):Q_{F,\chi}^I(\mu)),$$ the number of appearances of $Q_{F,\chi}^I(\mu)$ in the $Q$-filtration of $Q_{F,\chi}(\lambda)$. This number is well-defined by Proposition~\ref{LeviUniqDecomp}. 
\begin{theorem}
	Suppose $\pi(h_\alpha)=0$ for all $\alpha\in R_I$, and let $\lambda,\mu\in X$. Then $(Q_{F,\chi}(\lambda):Q_{F,\chi}^I(\mu))=[Z_{\bD F,\chi}(\mu), L_{\bD F,\chi}(\lambda)]$.
\end{theorem}

\begin{proof}
	We have the following chain of equalities, denoting the functor $\Hom_F(-,F)$ as $\,^{*}$.
	\begin{equation*}
		\begin{split}
			(Q_{F,\chi}(\lambda):Q_{F,\chi}^I(\mu)) & =\dim_F\Hom_{\sC_F^{I,-}}(Q_{F,\chi}(\lambda),Z_{F,I,\chi}(\mu)) \\ & = \dim_F\Hom_{\overline{\sC_{\overline{F}}^{I,-}}}(Z_{F,I,\chi}(\mu)^{*}, Q_{F,\chi}(\lambda)^{*}) \\ & = \dim_F\Hom_{\overline{\sC_{\overline{F}}}}(\Gamma'_{\overline{F},-\chi}(Z_{F,I,\chi}(\mu)^{*}), Q_{F,\chi}(\lambda)^{*}) \\ & = \dim_F\Hom_{\sC_{F}}(Q_{F,\chi}(\lambda),\Gamma'_{\overline{F},-\chi}(Z_{F,I,\chi}(\mu)^{*})^{*}) \\ & = [\Gamma'_{\overline{F},-\chi}(Z_{F,I,\chi}(\mu)^{*})^{*}:L_{F,\chi}(\lambda)] \\ & =[\Gamma'_{\overline{F},-\chi}(Z_{F,I,\chi}(\mu)^{*}):L_{F,\chi}(\lambda)^{*}] \\ & 
			=[\,^{\tau}\Gamma'_{\overline{F},-\chi}(Z_{F,I,\chi}(\mu)^{*}):L_{\bD F,\chi}(\lambda)] \\ & = [\Gamma_{\bD F,\chi}(\bD(Z_{F,I,\chi}(\mu))):L_{\bD F,\chi}(\lambda)] \\ & = [Z_{\bD F,\chi}(\mu):L_{\bD F,\chi}(\lambda)].
		\end{split}
	\end{equation*}
	Let us explain where these equalities come from. The first equality comes from the fact that $$\Hom_{\sC_F^{I,-}}(Q_{F,\chi}^I(\lambda),Z_{F,I,\chi}(\mu))=\twopartdef{F}{\lambda\in W_{I,p}\cdot \mu,}{0}{\mbox{not.}}$$ The second equality comes from an analogue of Corollary~\ref{HomDual} for the category $\sC_F^{I,-}$ (nothing in its proof depending upon the objects being $U_\chi\otimes F$-modules rather than $U_I^{-}U^I\otimes F$-modules). The third equality comes from Frobenius reciprocity, which works just as well for $U_{-\chi}\otimes F$-modules as for $U_\chi\otimes F$-modules. The fourth equality also follows from Corollary~\ref{HomDual}. The fifth equality is a general fact about projective covers. The sixth equality comes from applying $\Hom_F(-,F)$, noting that the simplicity of $L_{F,\chi}(\lambda)$ in $\sC_F$ implies the simplicity of  $L_{F,\chi}(\lambda)^{*}$ in $\overline{\sC_{\overline{F}}}$. The seventh equality comes from applying the functor $\overline{\sC_{\overline{F}}}\to \sC_{\bD F}$, $M\mapsto \,^{\tau} M$, and using Proposition~\ref{DualIrred}. The eighth equality will be addressed momentarily. Finally, the ninth equality comes from Proposition~\ref{DualIrred}, since that implies $\bD(Z_{F,I,\chi}(\mu))\cong Z_{\bD F,I,\chi}(\mu)$ in $\sC_F^I$.
	
	For the eighth equality, suppose $M\in\overline{\sC_{\overline{F}}^I}$ with $M=M_{-\lambda+\bZ I}$. We want to show that $$\,^{\tau}\Gamma'_{\overline{F},-\chi}(M)\cong \Gamma_{\bD F,\chi}(\,^{\tau} M)$$ in $\sC_{\bD F}$. By Frobenius reciprocity, we have $$\Hom_{\sC_{\bD F}^{I,+}}(\,^{\tau}M,\,^{\tau}\Gamma'_{\overline{F},-\chi}(M))\cong\Hom_{\sC_{\bD F}}(\Gamma_{\bD F,\chi}(\,^{\tau} M),\,^{\tau}\Gamma'_{\overline{F},-\chi}(M)),$$ noting that $\tau$ means the same thing in all appearances here, since we observed above that the $\tau$ used in $\sC_{\bD F}^I$ is just a restriction of the $\tau$ used in $\sC_{\bD F}$. Here, $\,^{\tau}M$ is viewed as lying in $\sC_{\bD F}^{I,+}$ by trivial extension. There is a morphism in $\sC_{\bD F}^{I,+}$ from $\,^{\tau}M$ to $\,^{\tau}\Gamma'_{\overline{F},-\chi}(M)$ which sends $m\in\,^{\tau}M$ to $1\otimes m\in \,^{\tau}\Gamma'_{\overline{F},-\chi}(M)$. This is a morphism in $\sC_{\bD F}^{I,+}$ since, if $u\in (U_{I}^{+})_{\sigma+\bZ I}$ for $\sigma+\bZ I>0+\bZ I$, and $z\in (\,^{\tau}\Gamma'_{\overline{F},-\chi}(M))_{\lambda+\bZ I}$, we have $u\cdot z\in \Gamma'_{\overline{F},-\chi}(M)_{-\lambda-\sigma+\bZ I}$ (using here that $\tau^{-1}(\fu^{+})=\fu^{-}$ as in \cite{Jan5}). But $\Gamma'_{\overline{F},-\chi}(M)_{-\lambda-\sigma+\bZ I}=0$ whenever $\sigma+\bZ I>0+\bZ I$, by construction, so we indeed have a morphism in $\sC_{\bD F}^{I,+}$.

	This therefore induces a morphism $\Gamma_{\bD F,\chi}(\,^{\tau} M)\to \,^{\tau}\Gamma'_{\overline{F},-\chi}(M)$ in $\sC_{\bD F}$. It is clear that $1\otimes M$ generates $\,^{\tau}\Gamma'_{\overline{F},-\chi}(M)$ in $\sC_{\bD F}$. Hence, there is a surjective morphism in $\sC_{\bD F}$:
	$$\Gamma_{\bD F,\chi}(\,^{\tau} M)\to \,^{\tau}\Gamma'_{\overline{F},-\chi}(M).$$ 
	Since the $F$-dimensions coincide, this must be an isomorphism, as required.
\end{proof}

\begin{cor}
	Suppose $\pi(h_\alpha)=0$ for all $\alpha\in R_I$, and let $\lambda,\mu\in X$. Then $(Q_{F,\chi}(\lambda):Q_{F,\chi}^I(\lambda))=1$ and $(Q_{F,\chi}(\lambda):Q_{F,\chi}^I(\mu))$ is non-zero only if $\mu+\bZ I\geq\lambda+\bZ I$.
\end{cor}
\begin{proof}
	It is clear that $[Z_{\bD F,\chi}(\mu):L_{\bD F,\chi}(\lambda)]\neq 0$ only if $\mu+\bZ I\geq\lambda+\bZ I$. Furthermore, since $L_{\bD F,\chi}(\lambda)_{\lambda+\bZ I}=Z_{\bD F,\chi}(\lambda)_{\lambda+\bZ I}$, the simple module $L_{\bD F,\chi}(\lambda)$ can only appear once in a composition series of $Z_{F,\chi}(\lambda)$ (it clearly must appear at least once).
\end{proof}

\begin{cor}\label{QTruncFact}
	Suppose $\pi(h_\alpha)=0$ for all $\alpha\in R_I$, let $\lambda,\mu\in X$, and let $\nu+\bZ I\in X/\bZ I$ with $\lambda+\bZ I\leq \nu+\bZ I$. Then $(Q_{F,\chi}^{\nu+\bZ I}(\lambda):Q_{F,\chi}^I(\lambda))=1$ and $(Q_{F,\chi}^{\nu+\bZ I}(\lambda):Q_{F,\chi}^I(\mu))$ is non-zero only if $\mu+\bZ I\geq\lambda+\bZ I$ and $\mu+\bZ I\leq\nu+\bZ I$.
\end{cor}

\begin{cor}
	Suppose $\pi(h_\alpha)=0$ for all $\alpha\in R_I$, and let $\lambda\in X$. Then $Q_{F,\chi}^I(\lambda)$ is projective if and only if $Z_{\bD F,\chi}(\lambda)$ is irreducible.
\end{cor}

We are now ready to prove the main result of this section, which should be compared with Proposition 4.18 in \cite{AJS}.

\begin{theorem}\label{ProjCovQ}
	Suppose that $\pi(h_\alpha)=0$ for all $\alpha\in R_I$. For all $\lambda\in X$ with $\lambda+\bZ I\leq \mu+\bZ I$, there exists a projective module $Q_{A,\chi}^{\nu+\bZ I}(\lambda)\in\sC_A(\leq\nu+\bZ I)$ with $Q_{A,\chi}^{\nu+\bZ I}(\lambda)\otimes_{A} F\cong Q_{F,\chi}^{\nu+\bZ I}(\lambda)$.
\end{theorem}

\begin{proof}
	Recall the definition of $\Xi_{A,\chi}^I(\lambda)$ from the end of Subsection~\ref{Sec8.1}. We know that $\Xi_{A,\chi}^I(\lambda)\in\sC_A$ is projective and has a $Q$-filtration where $Q_{A,\chi}^I(\lambda)$ appears once (at the top) and all other factors are of the form $Q_{A,\chi}^I(\mu)$ for $\mu\in X$ with $\mu+\bZ I>\lambda+\bZ I$. Set $Q=T^{\nu+\bZ I}(\Xi_{A,\chi}^I(\lambda))\in\sC_A(\leq\nu+\bZ I)$. This is a projective module in $\sC_A(\leq\nu+\bZ I)$ by Lemma~\ref{TruncProj}. It has a $Q$-filtration by Corollary~\ref{ProjQFilt} and there exists a surjection $Q\twoheadrightarrow Q_{A,\chi}^I(\lambda)$.
	
	Then $Q\otimes_{A} F$ is projective in $\sC_F(\leq\nu+\bZ I)$ by a similar argument to that of Lemma~\ref{projscal}. Since the simple objects in $\sC_F(\leq\nu+\bZ I)$ are precisely the objects $L_{F,\chi}(\mu)$ for $\mu\in \Lambda_I$ with $\mu+\bZ I\leq \nu+\bZ I$, and since each $Q_{F,\chi}^{\nu+\bZ I}(\mu)$ is the projective cover of such $L_{F,\chi}(\mu)$ in $\sC_F(\leq\nu+\bZ I)$, every projective object in $\sC_F(\leq\nu+\bZ I)$ is a direct sum of some $Q_{F,\chi}^{\nu+\bZ I}(\mu)$ with $\mu\in \Lambda_I$ with $\mu+\bZ I \leq \nu+\bZ I$. Hence, $Q\otimes_{A} F$ is a direct sum of some copies of $Q_{F,\chi}^{\nu+\bZ I}(\mu)$ with $\mu\in \Lambda_I$, $\mu+\bZ I\leq\nu+\bZ I$. 
	
	If $Q_{F,\chi}^{\nu+\bZ I}(\mu)$ appears in the direct sum decomposition then, by Corollary~\ref{QTruncFact}, $Q_{F,\chi}^I(\mu)$ appears in the $Q$-filtration of $Q\otimes_{A} F$. Therefore, from what we know about the $Q$-filtration, we must have $\mu+\bZ I\geq \lambda+\bZ I$. Furthermore, there can be at most one $Q_{F,\chi}^{\nu+\bZ I}(\mu)$ appearing in the filtration with $\mu+\bZ I=\lambda+\bZ I$. Since both $Q\otimes_{A} F$ and $Q_{F,\chi}^{\nu+\bZ I}(\lambda)$ surject onto $L_{F,\chi}(\lambda)$, and since $Q_{F,\chi}^{\nu+\bZ I}(\lambda)$ is the projective cover of $L_{F,\chi}(\lambda)$ in $\sC_F(\leq\lambda+\bZ I)$, the module $Q_{F,\chi}^{\nu+\bZ I}(\lambda)$ must appear at least once in the decomposition. We therefore have $$Q\otimes_{A} F\cong Q_{F,\chi}^{\nu+\bZ I}(\lambda)\oplus\bigoplus_{\lambda+\bZ I<\mu+\bZ I\leq \nu+\bZ I} Q_{F,\chi}^{\nu+\bZ I}(\mu)^{m(\mu)}$$ for some $m(\mu)\geq 0$.

	If $\lambda+\bZ I=\nu+\bZ I$, we may therefore take $Q_{A,\chi}^{\nu+\bZ I}(\lambda)=Q$. The remainder of the proof works in exactly the same way as Proposition 4.18 does in \cite{AJS}, using an analogue of Proposition 3.3 in \cite{AJS} where necessary.

\end{proof}

\begin{theorem}\label{ProjCovQ2}
	Suppose that $\pi(h_\alpha)=0$ for all $\alpha\in R_I$. For all $\lambda\in X$, there exists a projective module $Q_{A,\chi}(\lambda)\in\sC_A$ with $Q_{A,\chi}(\lambda)\otimes_{A} F\cong Q_{F,\chi}(\lambda)$. This module is uniquely determined up to isomorphism. Furthermore, any projective module in $\sC_A$ can be expressed as a direct sum of some $Q_{A,\chi}(\lambda)$.	
\end{theorem}

\begin{proof}
	One may easily check that $\Xi_{A,\chi}^I(\lambda)\in\sC_A$ lies inside $\sC_A(\leq 2(p-1)\rho+\bZ I)$. Choosing then $\nu+\bZ I\in X/\bZ I$ with $\nu+\bZ I\geq 2(p-1)\rho+\bZ I$, we get that $\Xi_{A,\chi}^I(\lambda)=T^{\nu+\bZ I} \Xi_{A,\chi}^I(\lambda)$, which is the object $Q$ from the previous proof. Then $Q$ and $Q_{A,\chi}^{\nu+\bZ I}(\lambda)$ are projective in $\sC_A$, and we have $Q_{A,\chi}^{\nu+\bZ I}(\lambda)\otimes_{A} F=Q_{F,\chi}^{\nu+\bZ I}(\lambda)$. Since there is a surjection $\Xi_{F,\chi}^I(\lambda)\twoheadrightarrow Q_{F,\chi}^I(\lambda)\twoheadrightarrow Z_{F,\chi}(\lambda)\twoheadrightarrow L_{F,\chi}(\lambda)$, and since $Q_{F,\chi}(\lambda)$ is the projective cover of $L_{F,\chi}(\lambda)$ in $\sC_F$, there is a surjection $\Xi_{F,\chi}^I(\lambda)\twoheadrightarrow Q_{F,\chi}(\lambda)$. In particular, $Q_{F,\chi}(\lambda)\in\sC_F(\leq 2(p-1)\rho+\bZ I)$ and so $Q_{F,\chi}^{\nu+\bZ I}(\lambda)=Q_{F,\chi}(\lambda)$. Therefore, we may take $Q_{A,\chi}(\lambda)=Q_{A,\chi}^{\nu+\bZ I}(\lambda)$.
	
	Suppose $Q_{A,\chi}(\lambda)$ and $\widetilde{Q}_{A,\chi}(\lambda)$ are both projective modules in $\sC_A$ such that $Q_{A,\chi}(\lambda)\otimes_{A} F\cong \widetilde{Q}_{A,\chi}(\lambda)\otimes_{A} F\cong Q_{F,\chi}(\lambda)$. By Remark 4.18 in \cite{AJS} (adapted to this case), there exists $M\in\sC_A$ such that $Q_{A,\chi}(\lambda)\cong\widetilde{Q}_{A,\chi}(\lambda)\oplus M$. We must then have $M\otimes_{A} F=0$. But $M$ is projective in $\sC_A$, so has a $Q$-filtration by Corollary~\ref{ProjQFilt}, and so is free over $A$. This implies $M=0$. A similar argument proves the final claim.
\end{proof}

\section{Index of notation}

In this section we give a list of the notation used in this paper. We omit notation which is only used in individual proofs.

\hfill\\{\bf Section~\ref{Sec2}. Notation} 
\hfill\\	
$\bK$: Algebraically closed field of characteristic $p>0$.\\
$G$: Connected reductive algebraic group over $\bK$.\\
$T$: Maximal torus of $G$ of rank $d$.\\
$B$: Borel subgroup of $G$ containing $T$.\\
$X$: Character group $X=X(T)$.\\
$\fg,\fh,\fb$: Lie algebras of $G$, $T$, $B$ respectively.\\
$R$: Root system of $\fg$ with respect to $T$.\\
$R^{+}=\{\beta_1,\ldots,\beta_r\}$: Positive roots in $R$ corresponding to $B$.\\
$\Pi=\{\alpha_1,\ldots,\alpha_n\}$: Simple roots in $R^{+}$.\\
$Y(T)$: Cocharacter group of $T$.\\
$\langle\cdot,\cdot\rangle$: Natural pairing $X(T)\times Y(T)\to\bZ$.\\
$\alpha^\vee$: The coroot corresponding to $\alpha\in R$.\\
$x\mapsto x^{[p]}$: $p$-th power map on $\fg$, $\fb$, $\fh$.\\
$\{e_\alpha, h_i\mid\alpha\in R, 1\leq i\leq d\}$: Chosen basis of $\fg$, with $e_\alpha\in\fg_\alpha$ for $\alpha\in R$, and $h_i^{[p]}=h_i$ for $1\leq i\leq d$.\\
$h_\alpha\coloneqq[e_\alpha,e_{-\alpha}]$.\\
$\fn^{+}\coloneqq \bigoplus_{\alpha\in R^{+}}\fg_\alpha$.\\
$\fn^{-}\coloneqq \bigoplus_{\alpha\in R^{+}}\fg_{-\alpha}$.\\
$I$: Subset of $\Pi$.\\
$\chi$: Linear form on $\fg$ in standard Levi form corresponding to $I\subseteq\Pi$.\\
$\geq$: Partial ordering on $X/\bZ I$.

\hfill\\{\bf Section~\ref{Sec3}. The category $\sC_A$}
\hfill\\
{\bf Subsection~\ref{Sec3.1}. Definition of algebras}\hfill\\
$U(\fg)$: Universal enveloping algebra of $\fg$.\\
$U_\chi(\fg)$: Reduced enveloping algebra of $\fg$.\\
$U_\chi\coloneqq U(\fg)/\langle e_\alpha^p-\chi(e_\alpha)^p\mid\alpha\in R\rangle$.\\
$U^{+}\coloneqq U_\chi(\fn^{+})=U_0(\fn^{+})\subseteq U_\chi$.\\
$U^{-}\coloneqq U_\chi(\fn^{-})\subseteq U_\chi$.\\
$U^0\coloneqq U(\fh)\subseteq U_\chi$.\\
$U^I$: Subalgebra of $U_\chi$ generated by $\fh$ and root vectors $e_\alpha$ for $\alpha\in R\cap \bZ I$.\\
$R_I\coloneqq R\cap\bZ I$.\\
$R_I^{+}\coloneqq R^{+}\cap\bZ I$.\\
$\fg_I$: Lie subalgebra generated by $\fh$ and root vectors $e_\alpha$ for $\alpha\in R_I$.\\
$\fu^{+}$: Lie subalgebra generated by $e_\alpha$ for $\alpha\in R^{+}\setminus\bZ I$.\\
$\fu^{-}$: Lie subalgebra generated by $e_\alpha$ for $\alpha\in R^{-}\setminus\bZ I$.\\
$U_I^{+}\coloneqq U_\chi(\fu^{+})=U_0(\fu^{+})\subseteq U_\chi$.\\
$U_I^{-}\coloneqq U_\chi(\fu^{-})=U_0(\fu^{-})\subseteq U_\chi$.\\
$U_\chi=\bigoplus_{\lambda+\bZ I\in X/\bZ I} (U_\chi)_{\lambda+\bZ I}$: Standard $X/\bZ I$-grading on $U_\chi$.\\
$(U_\chi)_{\lambda+\bZ I}=\bigoplus_{d\mu\in \fh^{*}, \,\mu\in \lambda+\bZ I +pX }(U_\chi)_{\lambda+\bZ I}^{d\mu}:$ Weight space decomposition of $(U_\chi)_{\lambda+\bZ I}$.\\
$\mu\mapsto \widetilde{\mu}$: Group homomorphism $X\to \Aut_{\bK-alg}(U^0)$ from Lemma~\ref{Aut}.\\
\hfill
{\bf Subsection~\ref{Sec3.2}. Definition of the category $\sC_A$}\hfill\\
$\pi:U^0\to A$: Commutative Noetherian algebra over $U^0$.\\
$\sC_A$: Category of $X/\bZ I$-graded $U_\chi\otimes A$-modules satisfying Conditions (A), (B), (C) and (D).\\
$M_{\lambda+\bZ I}$: $\lambda+\bZ I$-graded part of $M$.\\
$M_{\lambda+\bZ I}^{d\mu}$: (D)-decomposition summand of $M_{\lambda+\bZ I}$ corresponding to $d\mu\in\fh^{*}$, $\mu\in\lambda+\bZ I +pX$.\\
$\widetilde{\sC_A}$: Category of $X/p\bZ I$-graded $U_\chi\otimes A$-modules satisfying conditions (A'), (B'), (C') and (D').	
\hfill\\
{\bf Subsection~\ref{Sec3.3}. Categories defined over subalgebras}\hfill\\
$\sC_A'$: Category of $X/\bZ I$-graded $U^0U^{+}\otimes A$-modules satisfying Conditions (A), (B), (C) and (D).\\
$\widehat{\sC_A'}$: Category of $X$-graded $U^0U^{+}\otimes A$-modules satisfying Conditions (\^{A}), (\^{B}), (\^{C}) and (\^{D}).\\
$\Upsilon_A$: Functor $\widehat{\sC_A'}\to \sC_A'$ (or $\widehat{\sC_A''}\to \sC_A''$) from Proposition~\ref{UpsilonDef}.\\
$M_\sigma$: $\sigma$-graded part of $X$-graded module $M$.\\
$\sC_A''$: Category of $X/\bZ I$-graded $U^0\otimes A$-modules satisfying Conditions (A), (B), (C) and (D).\\
$\widehat{\sC_A''}$: Category of $X$-graded $U^0\otimes A$-modules satisfying Conditions (\^{A}), (\^{B}), (\^{C}) and (\^{D}).	

\hfill\\{\bf Section~\ref{Sec4}. Induction}
\hfill \\
{\bf Subsection~\ref{Sec4.1}. Induction along $U^0$}\hfill\\
$\Phi_A'$: Induction functor $\sC_A''\to \sC_A'$.\\
$\Phi_A$: Induction functor $\sC_A''\to\sC_A$.\\	
{\bf Subsection~\ref{Sec4.2}. Baby Verma modules}\hfill\\
$Z_{A,\chi}$: Induction functor $\sC_A'\to \sC_A$.\\
$A^\lambda$: Object in $\sC_A''$ as defined in this subsection, corresponding to $\lambda\in X$.\\
$Z_{A,\chi}(\lambda)\coloneqq Z_{A,\chi}(A^\lambda)$ with $A^\lambda$ viewed as an object of $\sC_A'$ by inflation, for $\lambda\in X$.\\
{\bf Subsection~\ref{Sec4.3}. $Z$-Filtrations}\hfill\\
$\widehat{\Phi_A'}$: Functor $\widehat{\sC_A''}\to\widehat{\sC_A'}$ defined analogously to $\Phi_A'$.\\
{\bf Subsection~\ref{Sec4.4}. Induction from parabolic subalgebras}\hfill\\
$\sC_A^I$: Category of $X/\bZ I$-graded $U^I\otimes A$ modules satisfying Conditions (A), (B), (C) and (D).\\
$\fp$: Parabolic subalgebra $\fg_I\oplus\fu^{+}$ of $\fg$.\\
$\sC_A^{I,+}$: Category of $X/\bZ I$-graded $U^IU_I^{+}\otimes A$ modules satisfying Conditions (A), (B), (C) and (D).\\
$\Gamma_{A,\chi}$: Induction functor $\sC_{A}^{I,+}\to\sC_A$.\\
$\Phi_A^{I,+}$: Induction functor $\sC_A^I\to\sC_A^{I,+}$.\\
$\Phi_A^I$: Induction functor $\sC_A^I\to\sC_A$.\\
$Z_{A,I,\chi}(\lambda)$: Baby Verma module in $\sC_A^I$ corresponding to $\lambda\in X$.\\
$\fp'$: Parabolic subalgebra $\fu^{-}\oplus\fg_I$ of $\fg$.\\
$\sC_A^{I,-}$: Category of $X/\bZ I$-graded $U_I^{-}U^I\otimes A$ modules satisfying Conditions (A), (B), (C) and (D).\\
$\Gamma_{A,\chi}'$: Induction functor $\sC_{A}^{I,-}\to\sC_A$.\\
$\Phi_A^{I,-}$: Induction functor $\sC_A^I\to\sC_A^{I,-}$.

\hfill\\{\bf Section~\ref{Sec5}. Miscellaneous}
\hfill\\
{\bf Subsection~\ref{Sec5.1}. The category $\sC_A^I$}\hfill\\
$\sC_A^I(\lambda+\bZ I)$: The full subcategory of objects $M$ in $\sC_A^I$ with $M_{\lambda+\bZ I}=M$.\\
{\bf Subsection~\ref{Sec5.2}. Truncations}\hfill\\
$\sC_A(\leq\nu+\bZ I)$: The full subcategory $M\in \sC_A$ with  $M_{\lambda+\bZ I}\neq 0$ only if $\lambda+\bZ I\leq \nu+\bZ I$.\\
$T^{\lambda+\bZ I}$: Truncation functor $\sC_A\to\sC_A(\leq\lambda+\bZ I)$.\\
{\bf Subsection~\ref{Sec5.3}. Extension of scalars}\hfill\\
$\sG\sC_A$: The category of $X/\bZ I$-graded $U_\chi\otimes A$-modules satisfying conditions (A), (F), (C) and (D).
$\pi:U^0\to A\hookrightarrow A'$: Commutative $A$-algebra.\\
{\bf Subsection~\ref{Sec5.4}. Duality}\hfill\\
$\overline{\sC_A}$: The category analogous to $\sC_A$ for $U_{-\chi}\otimes A$-modules.\\
$\overline{\pi}:U^0\to\overline{A}$: $U^0$-algebra defined by $\overline{\pi}(h)=-\pi(h)$ for $h\in\fh$.\\
$\tau$: Automorphism of $G$ (and of $\fg)$ defined in \cite{Jan5}.\\
$W_I$: Weyl group corresponding to $R_I$.\\
$w_I$: Longest element in $W_I$.\\
$\,^{\tau}\pi:U^0\to\,^{\tau}A$: $U^0$-algebra defined by $\,^{\tau}\pi(h)=\pi(\tau^{-1}(h))$ for $h\in\fh$.\\
$\,^{\tau}M$: The $U_{-\chi}\otimes A$-module with the same $A$-module structure as $M$ and with the $U_{-\chi}$ action given by $x\cdot m=\tau^{-1}(x)m$ for $x\in U_{-\chi}$ and $m\in M$.\\
$\bD\pi:U^0\to\bD A$: $U^0$-algebra defined by $\bD\pi=\,^{\tau}(\overline{\pi})$.\\
$\bD$: Duality functor $\sC_A\to\sC_{\bD A}$.\\
$\overline{\bD}\pi:U^0\to\overline{\bD} A$: $U^0$-algebra defined by $\overline{\bD}\pi(h)=-\pi(\tau(h))$ for $h\in\fh$.\\
$\overline{\bD}$: ``Inverse'' duality functor $\sC_A\to\sC_{\overline{\bD} A}$.

\hfill\\{\bf Section~\ref{Sec6}. Properties of baby Verma modules}
\hfill\\
{\bf Subsection~\ref{Sec6.1}. Irreducibility}\hfill\\
$\pi:U^0\to F$: Commutative Noetherian $U^0$-algebra which is a field.\\
$L_{F,\chi}(\lambda)$: Unique irreducible quotient of $Z_{F,\chi}(\lambda)$, for $\lambda \in X$.\\
$Q_{F,\chi}(\lambda)$: Projective cover of $L_{F,\chi}(\lambda)$ in $\sC_F$, for $\lambda\in X$.\\
{\bf Subsection~\ref{Sec6.2}. Isomorphisms of baby Verma modules}\hfill\\
$s_\alpha$: Reflection $X\to X$ sending $\lambda\in X$ to $\lambda-\langle\lambda,\alpha^\vee\rangle\alpha$.\\
$s_{\alpha,m}$: Reflection $X\to X$ sending $\lambda\in X$ to $\lambda-\langle\lambda,\alpha^\vee\rangle\alpha+m\alpha$.\\
$W$: Weyl group of $R$.\\
$t_{m,\alpha}$: Translation $X\to X$ sending $\lambda\in X$ to $\lambda+m\alpha$.\\
$W_p$: $p$-affine Weyl group of $R$, subgroup of $\Aut_\bZ(X)$ generated by $W$ and $t_{\alpha,m}$ for $\alpha\in R$, $m\in\bZ$.\\ 
$W_I$: Weyl group of $R_I$.\\
$W_{I,p}$: $p$-affine Weyl group of $R_I$, subgroup of $\Aut_\bZ(X)$ generated by $W$ and $t_{\alpha,m}$ for $\alpha\in R_I$, $m\in\bZ$.\\
$\rho$: Half-sum of positive roots in $X$.

\hfill\\{\bf Section~\ref{Sec7}. Regular nilpotent $p$-characters}\hfill\\
{\em Recall that in this section $\chi$ is regular nilpotent, and $\pi(h_\alpha)=0$ for all $\alpha\in R$.}\\	
$\pi^\circ:U^0\to A$: $U^0$-algebra defined by $\pi^\circ(h)=0$ for $h\in\fh$. \hfill\\
{\bf Subsection~\ref{Sec7.1}. An equivalence of categories}\hfill\\
$\sC_A^\circ$: Category defined analogously to $\sC_A$ for $U^0$-algebra $\pi^\circ:U^0\to A$.\\
$\Theta_A$: Equivalence of categories $\sC_A^\circ\to\sC_A$ defined in Proposition~\ref{ThetaEquiv}.\\	
{\bf Subsection~\ref{Sec7.2}. Projective covers}\hfill\\
$Q_{\bK,\chi}(\lambda)$: Projective cover of $Z_{\bK,\chi}(\lambda)$ in $\sC_{\bK}$, for $\lambda\in X$.\\
$Q_{F,\chi}(\lambda)$: Projective cover of $Z_{F,\chi}(\lambda)$ in $\sC_{F}$, for $\lambda\in X$.\\
$Q_{F^\circ,\chi}(\lambda)$: Projective cover of $Z_{F^\circ,\chi}(\lambda)$ in $\sC_{F}^\circ,$ for $\lambda\in X$.\\
$Q_{A,\chi}(\lambda)\coloneqq \Theta_A(Q_{\bK,\chi}(\lambda)\otimes_{\bK}A)$ for $\lambda\in X$.\\
$\Lambda$: Set of $W_p$-dot-orbits on $X$, also fundamental domain for $W_p$-dot-action on $X$.

\hfill\\
{\bf Section~\ref{Sec8}. Arbitrary standard Levi form}\hfill\\
{\em Recall that in this section $\chi$ is in standard Levi form corresponding to $I\subseteq\Pi$, and $\pi(h_\alpha)=0$ for all $\alpha\in R_I$.}\hfill\\
{\bf Subsection~\ref{Sec8.1}. The module $Q_{A,\chi}^I(\lambda)$}\hfill\\
$(\sC_A^I)^\circ$: Category defined analogously to $\sC_A^I$ for $U^0$-algebra $\pi^\circ:U^0\to A$.\\
$\Theta_A^I$: Equivalence of categories $(\sC_A^I)^\circ\xrightarrow{\sim}\sC_A^I$ defined in this Subsection.\\
$Q_{\bK,I,\chi}(\lambda)$: Projective cover of $Z_{\bK,I,\chi}(\lambda)$ in $\sC_{\bK}^I$, for $\lambda\in X$.\\
$Q_{A,I,\chi}(\lambda)\coloneqq \Theta_A^I(Q_{\bK,I,\chi}(\lambda))$ for $\lambda\in X$.\\
$Q_{A,\chi}^I(\lambda)\coloneqq \Gamma_{A,\chi}(Q_{A,I,\chi}(\lambda))$ for $\lambda\in X$.\\
$\Lambda_I$: Set of $W_{I,p}$-orbits on $X$, also fundamental domain for $W_{I,p}$-action on $X$.\\
$\Theta_A^{I,+}$: Equivalence of categories $(\sC_A^{I,+})^\circ\xrightarrow{\sim} \sC_{A}^{I,+}$.\\
$\Xi_{A,\chi}^I(\lambda)\coloneqq\Phi_A^I(Q_{A,I,\chi}(\lambda))\in\sC_A$ for $\lambda\in X$.\\
{\bf Subsection~\ref{Sec8.2}. Projective covers}\hfill\\
$Q_{F,\chi}^{\nu+\bZ I}(\lambda)\coloneqq T^{\nu+\bZ I}(Q_{F,\chi}(\lambda))\in\sC_A(\leq\nu+\bZ I)$ for $\lambda\in X$.\\
$Q_{A,\chi}^{\nu+\bZ I}(\lambda)$: Projective module in $\sC_A(\leq\nu+\bZ I)$ with $Q_{A,\chi}^{\nu+\bZ I}(\lambda)\otimes_{A} F\cong Q_{F,\chi}^{\nu+\bZ I}(\lambda)$ when $A$ is a local ring with residue field $F$, for $\lambda\in X$.\\
$Q_{A,\chi}(\lambda)$: Projective module in $\sC_A$ with $Q_{A,\chi}(\lambda)\otimes_{A} F\cong Q_{F,\chi}(\lambda)$ when $A$ is a local ring with residue field $F$, for $\lambda\in X$.







\end{document}